\documentclass[11pt]{imanum} 

\usepackage{xcolor}
\usepackage{stmaryrd}
\usepackage{psfrag}
\usepackage{subfig}
\usepackage{graphicx}
\usepackage{float}
\usepackage{color}
\usepackage{standalone}
\usepackage{etoolbox}
\usepackage{amsmath}\usepackage{dsfont}
\usepackage{amsfonts}
\usepackage{ wasysym }
\usepackage{amssymb}
\usepackage{mymacros}
\usepackage{array} 
\usepackage{amsthm}
\usepackage{tikz}
\usepackage{appendix}
\usepackage{cite} \usepackage{mathtools}
\usepackage{bbm}
\usepackage{float}
\usepackage{tabularx}
\usepackage{enumerate}

\usetikzlibrary{arrows}

\usepackage{soul}

\newtheorem{assumption}{\sc Assumption}[section]

\newcommand{\GtG}{{\Gamma\shortrightarrow\Gamma}}
\newcommand{\Gtg}{{\Gamma\shortrightarrow\gamma}}
\newcommand{\gtG}{{\gamma\shortrightarrow\Gamma}}
\newcommand{\gtg}{{\gamma\shortrightarrow\gamma}}

\newcommand{\BDdudn}{\left(C_uC_\Gamma {k^{\beta-1/2}}J(k)\e^{-\tau_\Gamma p_\Gamma} +\left[1+C_\calG(k)\right]C_\gamma(k)\e^{-\tau_\gamma(k) p_\gamma}\right)}

\newcommand{\href}[2]{#2}
\newcommand{\url}[1]{#1}

\jno{drnxxx}
\received{}
\revised{}

\begin{document}

\newcommand{\DOFshp}{{N_\gamma}}
\newcommand{\DOFsHNA}{{N_\Gamma}}
\newcommand{\DOFsTotal}{N}

\newcommand{\numSides}{{\mathcal{N}_\Gamma}}
\newcommand{\numScats}{{\mathcal{N}_\gamma}}

\newcommand{\Vhp}{V_{\DOFshp}^{{hp}}(\gamma)}
\newcommand{\VHNA}{V_{\DOFsHNA}^{\mathop{\rm HNA}}(\Gamma)} 
\newcommand{\V}{V_{\DOFsTotal}^{\mathop{\rm HNA}^*}(\Gamma,\gamma)}

\newcommand{\bigSol}{v_\Gamma^{\DOFsHNA}}
\newcommand{\smlSol}{v_\gamma^{\DOFshp}}

\newcommand{\numLayersj}{{n_j}}
\newcommand{\numLayers}{n}

\newcommand{\Nought}{\DOFsTotal_0}
\newcommand{\MSM}{u_{\max}(k)}
\newcommand{\uinc}{u^i}

\newcommand{\resub}[1]{{\color{black}#1}}

\title{A high frequency boundary element method for scattering by a class of multiple  obstacles
}
\shorttitle{High frequency BEM for multiple obstacles}
\author{%
	{\sc
		Andrew Gibbs\resub{\thanks{Corresponding author. Email: andrew.gibbs@ucl.ac.uk},\\[2pt]
		Dept. of Mathematics, University College London, WC1H 0AY, UK}\\
		{\sc and}\\[6pt]
		Simon \resub{N.} Chandler-Wilde\thanks{Email: s.n.chandler-wilde@reading.ac.uk},\\[2pt]
		Dept. of Mathematics and Statistics, University of Reading, RG6 6AX, UK\\
			{\sc and}\\[6pt]
			Ste\resub{phen} Langdon\resub{\thanks{Email: stephen.langdon@brunel.ac.uk}\\[2pt]
			Dept. of Mathematics, Brunel University, London, UB8 3PH, UK\\
			{\sc and}\\[6pt]}
		Andrea Moiola\thanks{Email: andrea.moiola@unipv.it},\\[2pt]
		Dept. of Mathematics, University of Pavia, 27100, Italy}
}

\shortauthorlist{A. Gibbs \emph{et al.}}

\maketitle

\begin{abstract}
	{We propose a boundary element method for problems of time-harmonic acoustic scattering by multiple obstacles in two dimensions, at least one of which is a convex polygon. By combining a Hybrid Numerical Asymptotic (HNA) approximation space on the convex polygon with standard polynomial-based approximation spaces on each of the other obstacles, we show that the number of degrees of freedom required in the HNA space to maintain a given accuracy needs to grow only logarithmically with respect to the frequency, as opposed to the (at least) linear growth required by standard polynomial-based schemes.  
		\resub{This method is thus most effective when the convex polygon is many wavelengths in diameter and the small obstacles have a combined perimeter comparable to the problem wavelength.}
	}
	{Helmholtz, high frequency, multiple scattering, integral equations, BEM, hp discretisation, HNA method.}
\end{abstract}


\newcommand{\smoothGuys}{BrGeMo:04,DoGrSm:07,EcOz:17,Ec:18,EcEr:19}

\section{Introduction}\label{s:intro}
Standard finite or boundary element methods for wave scattering problems, with piecewise polynomial approximation spaces, typically require a{t least a} fixed number of degrees of freedom per wavelength to maintain accuracy as the frequency of the incident wave increases. 
This dependence can lead to a requirement for an excessively large number of degrees of freedom at high frequencies.

For certain geometries, the Hybrid Numerical Asymptotic (HNA) approach (see, e.g., \citet{ACTA} and the references therein) overcomes this restriction by absorbing the high frequency asymptotic behaviour into the approximation space. This is implemented via a Boundary Element Method (BEM), which is particularly effective as the high frequency behaviour need only be captured on the surface of the obstacle. \resub{When constructing an HNA method, a key ingredient is an understanding of the high frequency asymptotics of the underlying physical problem. Such methods are well-studied for strictly convex smooth scatterers, see for example \citet{\smoothGuys}, and the many references therein. This paper is instead concerned with HNA for polygonal scatterers, specifically extending this method to multiple obstacles. We discuss the relevant literature in detail now.}
 
 In contrast with standard BEM, in which the number of degrees of freedom {(DOFs)} required to accurately represent the solution depends linearly on the frequency, the {number of DOFs needed to achieve a given accuracy} (for scattering by a convex polygon in two-dimensions) was shown to depend only {logarithmically on}
the frequency for the $h$-BEM version of HNA in \citet{ChLa:07}, this improved to the $hp$-BEM version in \citet{HeLaMe:13_}.
These ideas were extended, in \citet{ChHeLaTw:15}, to a certain class of non-convex polygons, with the high frequency asymptotics arising from re-reflections and partial illumination (shadowing) being fully {captured by a careful choice of approximation space.}
Similar ideas have been applied to penetrable obstacles in \citet{GrHeLa:15,GrHeLa:17}
and to two- and three-dimensional screens in \citet{ChHeLa:14} and \citet{La:15} respectively.
All of these methods are, broadly speaking, for single obstacles and for plane wave incidence (although an extension to {other incident fields}
is discussed in Remark \ref{re:MSHNA_uigen}).
In this paper we extend the HNA method to a class of more general multiple scattering configurations. 
Although here we focus on the case where at least one of the obstacles is a convex polygon, \resub{the ideas we present may be applied to the same problem where this obstacle is (for example) strictly convex and smooth. T}he key ingredient is that this obstacle be amenable, for the corresponding single scattering problem, to solution by an HNA-BEM.

Problems of high frequency scattering by one large relatively simple obstacle and one (or many) small obstacle(s) are potentially of practical \resub{interest. An approach used in \citet{LeLuSa:17} and \citet{Pe..:17} for such problems} is to appeal to high frequency asymptotics on the large obstacle, via Geometrical/Physical Optics approximation, and approximate the solution on (or in some neighbourhood of) the small components using a standard BEM/FEM.
This approximation works well at sufficiently high frequencies, but ignores diffracted waves emanating from the large obstacle, and so is not controllably accurate across all frequencies. 
Moreover, a Geometrical Optics approach will include a ray-tracing algorithm, which typically requires that the multiple scattering problem is solved \emph{iteratively}\resub{, see \citet{EcRe:09,AnBoEcRe:10, GeBrRe:05}}. This involves reformulation as a Neumann series consisting 
entirely of operators on a single scatterer. More generally, iterative approaches are common in multiple scattering problems and 
work well for certain configurations. 
However, the iterative approach cannot be applied to all such problems: the Neumann series will diverge for cases where the separation of the obstacles is {too} small.
 \resub{A method which accelerates the rate of convergence in the Neumann series is presented in \citet{BoEcRe:16}.}

In this paper, we present a method which is particularly effective for high-frequency time-harmonic scattering by one large obstacle and one (or many) small obstacle(s). \resub{Specifically, since we use an oscillatory basis on the large obstacle, this can be many wavelengths long. On the small obstacle(s), we use a piecewise-polynomial basis, so the method is most effective when their combined perimeter is comparable to one wavelength.}
In contrast to the other methods currently available for similar problems, the method we present in this paper is controllably accurate and does not need to be solved iteratively, whilst the only constraint on the separation of the obstacles is that they must be $O(\lambda)$ apart, where $\lambda$ denotes wavelength, hence the obstacles may be very close together at high frequencies.

\subsection{Outline of the paper}
In \S\ref{s:prob_stat} we describe in detail the class of multiple scattering problem we are aiming to solve. 

In \S\ref{ss:bdry_rep} we extend the representation on which the HNA-BEM for a single convex polygon is based \resub{\citep{ChLa:07,HeLaMe:13}}, to account for the contribution to the solution from neighbouring scatterers. \resub{In particular, we derive a representation \eqref{dudn_rep}, which decomposes the Neumann trace on the large obstacle into the sum of a physical optics term, the diffracted waves, and the contribution arising from interaction with the small obstacles. In Theorem \ref{Mu_bdMS} we bound the Helmholtz solution in the complement of the scatterers with bounding constant algebraic in the wavenumber. We describe the singular behaviour of the envelope of the diffracted waves in Theorem \ref{th:big_bounds}.
	
In \S\ref{s:HNAspace} we construct an approximation space enriched with oscillatory basis functions, designed to represent the solution on the large obstacle with a number of degrees of freedom that does not need to increase significantly as frequency grows. We also describe some potential advantages of using an approximation space based on a single mesh, as opposed to the overlapping meshes of \citet{ChHeLa:14, ChLa:07,HeLaMe:13_}.
In addition, in \S\ref{s:hpBEM} we define a (standard) piecewise-polynomial space on the small scatterer. 
 Conditions for exponential convergence of the best approximation in terms of polynomial degree on the large and small scatterers are given by Theorem \ref{co:bestApproxGamma} and Proposition \ref{co:hpBAE} respectively.

In \S\ref{s:Galerkin_method} a Galerkin method using this approximation space is outlined, alongside related error estimates of the total field and far-field pattern. Numerical results for an implementation of this method are presented in~\S\ref{s:results}. While Theorem \ref{th:anz_approx_bd} ensures exponential convergence when the big scatterer is a convex polygon and the small scatterers are analytic, \S\ref{s:results} demonstrates that exponential convergence is still possible when the small scatterers have corners.
}

Finally, in the appendix, we introduce an alternative boundary integral equation formulation, which is provably coercive under certain geometric constraints. This gives us explicit quasi-optimality estimates, which when combined with results in earlier sections could be used to give explicit error estimates for a certain class of multiple scattering configurations.

\section{Problem statement}
\label{s:prob_stat}
We consider the two-dimensional problem of time-harmonic {acoustic} scattering by $\numScats+1$ sound-soft scatterers,
at least one of which is an $\numSides$-sided convex polygon. 
In addition to this convex polygon, we assume that the other $\numScats$ obstacles are pairwise disjoint with Lipschitz piecewise-$C^1$ boundary. 
Denote the interior of the convex polygon by $\resub{\Omega}\subset\R^2$ {and its boundary by}
$\Gamma:=\partial \resub{\Omega}$. 
We denote by $\Gamma_j$ the $j$th side of $\Gamma$, for $j=1,\ldots,\numSides$. 
The bounded open set
$\resub{\omega}:=\bigcup_{i=1}^{\numScats}{\resub{\omega_i} \subset \R^2\setminus \overline{\resub{\Omega}}}$ represents the collection of the $\numScats$ other obstacles, which are denoted $\resub{\omega_i}$, for $i=1,\ldots,\numScats$. 
We denote the combined Lipschitz boundary of these by $\gamma:=\partial \resub{\omega}$.
The unbounded exterior domain is denoted $D:=\R^2\setminus(\overline{\resub{\Omega}}\cup\overline{\resub{\omega}})$, with boundary $\partial D=\Gamma\cup\gamma$. The normal derivative operator (or Neumann trace) is defined as $\partial /\partial \bfn:=\bfn\cdot\nabla$, in which $\bfn=(n_1,n_2)$ denotes the unit 
	normal directed into $D$; we denote also $\bfn_j:=\bfn|_{\Gamma_j}$ and $\bfn_\gamma:=\bfn|_\gamma$.
We assume that the distance between $\resub{\Omega}$ and $\resub{\omega}$ is positive, so that $\deD$ is a Lipschitz boundary. A simple example of a geometric configuration that fits inside of this framework is depicted in Figure~\ref{fig:diag}. We note that throughout the paper, it is the quantities $\Gamma$ and $\gamma$ which are used most frequently.

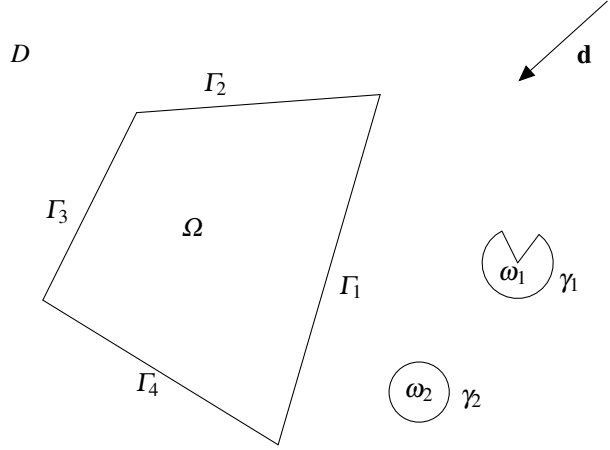
\begin{figure}[htb]
\centering
\definecolor{zzttqq}{rgb}{0,0,0}
\begin{tikzpicture}[line cap=round,line join=round,>=triangle 45,x=1.0cm,y=1.0cm]
\clip(2.,1.5) rectangle (11.,8.5);
\fill[color=zzttqq,fill=zzttqq,fill opacity=0] (2.4971256782692848,4.023866330078361) -- (3.7440277156113346,6.517670404762459) -- (6.981597917832797,6.758300622495135) -- (5.625318508794077,2.0988245882169516) -- cycle;
\draw [color=zzttqq] (2.4971256782692848,4.023866330078361)-- (3.7440277156113346,6.517670404762459);
\draw [color=zzttqq] (3.7440277156113346,6.517670404762459)-- (6.981597917832797,6.758300622495135);
\draw [color=zzttqq] (6.981597917832797,6.758300622495135)-- (5.625318508794077,2.0988245882169516);
\draw [color=zzttqq] (5.625318508794077,2.0988245882169516)-- (2.4971256782692848,4.023866330078361);
\draw [shift={(8.809883752518017,4.518568294018614)},color=zzttqq,fill=zzttqq,fill opacity=0]  (0,0) --  plot[domain=-4.253531723918849:0.9270493819054484,variable=\t]({1.*0.4729949160773011*cos(\t r)+0.*0.4729949160773011*sin(\t r)},{0.*0.4729949160773011*cos(\t r)+1.*0.4729949160773011*sin(\t r)}) -- cycle ;
\draw [->] (10.,8.) -- (8.81913776233687,6.933304417209809);

{
	\draw  (7.5,2.8) circle [radius=0.4];
	\draw (7.5,2.8) node{$\resub{\omega_2}$};
	\draw (8.2,2.7)node{${\gamma_2}$};}

\draw (4.5,5) node{$\resub{\Omega}$};
\draw (6.6,4.2) node{$\Gamma_1$};
\draw (4.8,6.9) node{$\Gamma_2$};
\draw (2.7,5.2) node{$\Gamma_3$};
\draw (3.9,2.9) node{$\Gamma_4$};

{\draw (8.75,4.35) node{$\resub{\omega_1}$};
	\draw (9.5,4.25) node{$\gamma_1$};}

\draw (9.7,7.3) node{$\bfd$};
\draw (2.2,7.3) node{$D$};

\end{tikzpicture}
\caption{Problem consisting of a convex \emph{four}-sided polygon (hence $\numSides=4$) and \emph{two} other obstacles (hence $\numScats=2$).}\label{fig:diag}
\end{figure}

We aim to solve the following boundary value problem (BVP): given the incident plane wave
\begin{equation}\label{eq:PW}
u^i(\bfx):=\e^{\imag k \bfx\cdot\bfd},\qquad \bfx\in\R^2,
\end{equation}
where $k:=2\pi/\lambda>0$ denotes the wavenumber (for wavelength $\lambda$) and $\bfd\in\R^2$ is a unit direction vector, determine the total field $u\in C^2(D)\cap C(\bar D)$ such that
\begin{align}
\Delta u+k^2u=0&\qquad\text{in }D,\label{Helmholtz}\\
u=0&\qquad\text{on }\partial D=\Gamma\cup\gamma\label{BC}
\end{align}
and $u^s:=u-u^i$ satisfies the Sommerfeld radiation condition \citet[(3.62)]{CoKr:13}
\begin{equation}\label{SRC}
\left(\pdiv{}{r}-\imag k \right)u^s(\bfx) = o(r^{-1/2}),\quad\text{as }r:=|\bfx|\rightarrow\infty.
\end{equation}
Problems for a broader class of incident field $u^i$ are discussed briefly in Remark \ref{re:MSHNA_uigen}.

The BVP \eqref{Helmholtz}--\eqref{SRC} can be reformulated as a boundary integral equation (BIE). 
We denote the single layer potential
$S_k:L^2(\partial D)\rightarrow C^2(D)$ by
\begin{equation}\label{def:Sk}
S_k\varphi(\bfx):=\int_{\partial D}\Phi_k(\textbf{x},\textbf{y})\varphi(\textbf{y})\dd{s}(\textbf{y}),\quad\bfx \in D,
\end{equation}
where $\Phi_k(\bfx,\bfy):=(\imag/4)H^{(1)}_0(k|\bfx-\bfy|)$ is the fundamental solution of \eqref{Helmholtz}, in which $H^{(1)}_0$ denotes the Hankel function of the first kind and order zero. 
If $u$ satisfies the BVP \eqref{Helmholtz}--\eqref{SRC},
then  $\pdivl{u}{\bfn}\in L^2(\partial D)$ and the following Green's representation holds (see, e.g.,\ \citet[Theorem~2.43]{ACTA})
\begin{equation}\label{Green}
u=u^i-S_k\pdiv{u}{\bfn}\quad\text{in } D.
\end{equation}

\begin{definition}[Combined potential operator]\label{def:BIE_A}
The standard combined potential operator $\calA_{k,\eta}:L^2(\partial D)\rightarrow L^2(\partial D)$ (see, e.g.,\ \citet{CoKr:13,ACTA}) is defined by
\[
\calA_{k,\eta}:=\frac12\calI+\calD_k'-\imag\eta\calS_k,
\]
where $\calI$ is the identity operator, $\eta\in\R\setminus\{0\}$ is a coupling parameter,
\[
\calS_k\varphi(\bfx):=\int_{\partial D}\Phi_k(\bfx,\bfy)\varphi(\bfy)\dd{s}(\bfy),\quad\bfx \in \partial D,\quad\varphi\in L^2(\partial D),
\]
denotes the single layer operator and
\[
\calD_k'\varphi(\bfx):=\int_{\partial D}\pdiv{\Phi_k(\bfx,\bfy)}{\bfn(\bfx)}\varphi(\bfy)\dd{s}(\bfy),\quad\bfx \in \partial D,\quad\varphi\in L^2(\partial D),
\]
denotes the adjoint of the double-layer operator.
\end{definition}

From \eqref{Green}, the BVP \eqref{Helmholtz}--\eqref{BC} can be reformulated as a BIE \citep[(2.69), (2.114)]{ACTA}
\begin{equation}\label{BIE}
\calA_{k,\eta}\pdiv{u}{\bfn}=f_{k,\eta},\quad\text{on }\partial D,
\end{equation}
where the right-hand side data $f_{k,\eta}\in L^2(\partial D)$ is
\begin{equation}\label{def:f}
f_{k,\eta}=\left(\pdiv{}{\bfn}-\imag\eta \right)u^i.
\end{equation}
It follows from \citet[Theorem~2.27]{ACTA} that $\calA_{k,\eta}$ is invertible. 
\resub{A standard variational form of \eqref{BIE} is
	\begin{equation}\label{eq:var_form}
	\left(\calA_{k,\eta}\pdiv{u}{\bfn},w\right)_{L^2(\Gamma\cup\gamma)}=\left(f_{k,\eta},w\right)_{L^2(\Gamma\cup\gamma)},\quad\text{for all }w\in L^2(\Gamma\cup\gamma),
	\end{equation}
	which can be approximated by a piecewise-polynomial Galerkin BEM. Our approach differs from this in that: (i) we decompose the unknown $\pdivl{u}{\bfn}$ into a known physical optics term, a diffracted term, and a term which expresses the leading order behaviour on $\Gamma$ in terms of the solution on $\gamma$ (see \S\ref{ss:bdry_rep}); (ii) we approximate the diffracted term on $\Gamma$ using an oscillatory basis (see \S\ref{s:HNAspace}). The use of this basis is justified by the representation and regularity results in \S\ref{sec:anal}. This leads to a new variational formulation that is equivalent to \eqref{eq:var_form}. This is shown in \eqref{eq:new_var1}--\eqref{eq:new_var2} with the resulting Galerkin scheme given in equations \eqref{Galerkin_eqns1}--\eqref{Galerkin_eqns2}.}

\resub{\subsection{Geometric assumptions}}

In related literature, there appears to be no single consistent definition of the term \emph{polygon}, so we shall clarify a definition that is appropriate for this paper.
\begin{definition}[Polygon]\label{def:polygonThisThesis}
	We say $\Upsilon{\subset\R^2}$ is a polygon if it is a bounded Lipschitz open set with a boundary consisting 
	of {a finite number of} straight line segments.
\end{definition}
We note that Definition \ref{def:polygonThisThesis} permits multiple disconnected shapes, whereas other conventions in related literature do not. As we impose that $\resub{\Omega}$ is convex, it cannot consist of disconnected components. On the other hand, $\resub{\omega}$ may consist of disconnected components. Many results that follow hold for a subclass of polygons, which we define now (as in, e.g., \citet[Definition~1.1]{Sp:14}).
\begin{definition}[Non-trapping polygon]\label{def:nontrapping}
	We say that a polygon $\Upsilon$ (in the sense of Definition \ref{def:polygonThisThesis}) is non-trapping if:
	\begin{enumerate}[(i)]
		\item No three vertices of $\Upsilon$ are co-linear; 
		\item For a ball $B_R$ with radius $R>0$ sufficiently large that $\Upsilon\subset B_R$, there exists a $T(R)<\infty$ such that all billiard trajectories that start inside of $B_R\setminus\Upsilon$ at time $T=0$ and miss the vertices of $\Upsilon$ will leave $B_R$ by time $T(R)$.
	\end{enumerate} 
\end{definition}

Previous analyses of HNA methods (e.g.,  \citet{HeLaMe:13_,ChHeLaTw:15}) have \resub{instead} relied upon convergence and regularity estimates for scattering obstacles which are convex or star-shaped (introduced formally in Definition \ref{def:star_comb}), a property not enjoyed by multiple scattering configurations. \resub{However, for configurations which satisfy the conditions of Definition \ref{def:nontrapping}, bounds on the Dirichlet-to-Neumann (DtN) maps are known \citep{BaSpWu:16}, which will provide an alternative route to bounding the solution to \eqref{Helmholtz}--\eqref{SRC} in \S\ref{ss:u_max}.}

\resub{In addition to the theory of \citet{BaSpWu:16} for non-trapping polygons, we shall consider a certain class of trapping configurations, for which bounds on DtN maps were recently derived in \citet{ChSpGiSm:17}, building on the earlier work of \citet{GaMuSp:16, BaSpWu:16}. These estimates will form a key component of our numerical analysis, in particular enabling us to bound the solution to \eqref{Helmholtz}--\eqref{SRC} in \S\ref{ss:u_max}, and obtain best approximation on $\gamma$ in Proposition \ref{co:hpBAE}.} A formal definition of \resub{these so-called $(R_0,R_1)$} configurations will follow, but these may be loosely interpreted as configurations $\Upsilon$ which are star-shaped outside of some ball. There is a second ball inside of the first, whose radius is sufficiently small, and inside of which some trapping may occur.

\begin{definition}[$(R_0,R_1)$ configuration]\label{def:R0R1}
	For $0<R_0<R_1$ we say that a Lipschitz $\Upsilon$ is an $(R_0,R_1)$  configuration if there exists a $\chi\in C^3[0,\infty)$ which satisfies
	\begin{enumerate}[(i)]
		\item $\chi(|\bfx|)=0$ for $0\leq |\bfx| \leq R_0$, $\chi(|\bfx|)=1$ for $|\bfx|\geq R_1, 0<\chi(|\bfx|)<1,$ for $R_0\leq |\bfx| \leq R_1$,
		\item $0\leq\chi'(|\bfx|)\leq 4$, for $|\bfx|>0$,
	\end{enumerate}
	such that $\bfZ(\bfx)\cdot\bfn(\bfx)\geq0$ for all $\bfx\in\partial\Upsilon$ for which the normal $\bfn(\bfx)$ is defined, where
	\[
	\bfZ(\bfx):=(x_1\chi(\bfx),x_2),\quad\bfx=(x_1,x_2)\in\R^2.
	\]
	
\end{definition}

Naturally, one can rotate the coordinate system if required to ensure the above conditions hold. 
For further explanation and examples of $(R_0,R_1)$ configurations, we refer to \citet[\S1.2.1]{ChSpGiSm:17}.

\section{Representation and regularity of solution on {$\Gamma$}} \label{s:Regularity_results}
The structure of this section is as follows: In \S\ref{ss:bdry_rep} we extend the single scattering HNA ansatz of \citet[(3.5)]{ChLa:07} to a multiple scattering configuration, introducing a new operator which accounts for the other obstacle(s). 
In \S\ref{ss:u_max} we bound the solution {of the multiple scattering problem in the domain $D$}, 
a necessary component of the best approximation estimates that follow. In \S\ref{sec:anal} we show that the envelopes of the diffracted waves, which the HNA space is designed to approximate, behave similarly to the single scattering problem (under very reasonable assumptions). This means that the HNA space of \citet{HeLaMe:13_} may be used {on the convex polygon in the} multiple scattering approximation without any modification (though with a different leading order term).

\subsection{The representation formula for the Neumann trace on {$\Gamma$}}\label{ss:bdry_rep}
As in \citet[\S3]{ChLa:07}, we will extend a single side $\Gamma_j$ of $\resub{\Omega}$, and solve the resulting half-plane problem, to obtain an explicit representation for $\pdivl{u}{\bfn}$ on $\Gamma_j$ in terms of known oscillatory functions on $\Gamma_j$ and {(in contrast to \citet[\S3]{ChLa:07})} $\pdivl{u}{\bfn}$ on $\gamma$. 
This representation will form the ansatz used for the discretisation.
Throughout this section, when $u$ or $u^s$ is restricted to $\Gamma\cup\gamma$, it is assumed that the exterior trace has been taken. 
Considering a single side $\Gamma_j$ of $\resub{\Omega}$, $1\le j\le \numSides$, define $\Gamma_j^+$ and $\Gamma_j^-$ as the infinite extensions of $\Gamma_j$, each as a straight half line in the clockwise and anti-clockwise direction (about the interior $\resub{\Omega}$) respectively (see Figure \ref{fig:UHPBR}).
Denote by $U_j$ the (open)
half-plane with boundary $\Gamma_j^\infty:=\Gamma_j^-\cup\Gamma_j\cup\Gamma_j^+$, chosen such that $U_j$ does not contain $\resub{\Omega}$. 
{We informally call $U_j$ the \emph{upper} half-plane relative to $\Gamma_j$.}
On $\Gamma_j^\infty$, the unit normal $\bfn_j$ points into $U_j$.
Define the half-plane Dirichlet Green's function
\[
G_j(\bfx,\bfy):=\Phi_k(\bfx,\bfy)-\Phi_k({\tbfx^j},\bfy),\quad\bfx\neq\bfy,
\]
where ${\tbfx^j}$ is the reflection of $\bfx$ across $\Gamma_j$. 
Formally, $\bfx={\tbfx^j}$ when $\bfx\in\Gamma_j$, otherwise ${\tbfx^j}\neq\bfx$ satisfies $\text{dist}(\{\bfx\},\Gamma_j^\infty)=\text{dist}(\{{\tbfx^j}\},\Gamma_j^\infty)=\tfrac{1}{2}|\bfx-{\tbfx^j}|$. 
{It follows that
\begin{equation}\label{eq:GPhi}
\pdiv{G_j(\bfx,\bfy)}{\bfn_j(\bfy)}=2\pdiv{\Phi_k(\bfx,\bfy)}{\bfn_j(\bfy)}
\qquad\text{and}\qquad 
G_j(\bfx,\bfy)=0,
\qquad\text{for}\;\bfy\in\Gamma_j^\infty.
\end{equation}
}
We {let $B_R$ be an open ball} 
of radius $R$ centred at the origin, with $R$ chosen sufficiently large that $U_j\cap\resub{\omega}\subset B_R$, i.e.\ all the scatterers in the relative upper half-plane lie inside the ball. 

\begin{figure}[!ht]
\centering
\begin{tikzpicture}[line cap=round,line join=round,>=triangle 45,x=2.727272727272727cm,y=2.8846153846153846cm]
\clip(-1.7,-0.5) rectangle (2.7,2.1);
\draw [dash pattern=on 2pt off 2pt,domain=-1.7:2.7] plot(\x,{(-0.-0.*\x)/1.62315236268});
\draw (-0.623152362678,0.)-- (1.,0.);
\draw [dotted] (-0.623152362678,0.)-- (-1.24785257803,-0.77607524246);
\draw [dotted] (1.,0.)-- (1.38844550082,-1.09408788293);
\draw [dotted] (-1.24785257803,-0.77607524246)-- (0.352271233247,-1.94408670507);
\draw [dotted] (1.38844550082,-1.09408788293)-- (0.352271233247,-1.94408670507);
\draw [dash pattern=on 1pt off 1pt on 2pt off 4pt] (0.499990520159,0.) ellipse (5.56482521376cm and 5.88587282225cm);
\draw (0.322169371539,0.84780524774)-- (1.1680129821,1.13341393038);
\draw (1.1680129821,1.13341393038)-- (1.4354302054,0.412731602718);
\draw (1.4354302054,0.412731602718)-- (0.322169371539,0.84780524774);
\draw [dotted] (1.69002707677,-0.358751658633)-- (2.30506666008,-0.317178522042);
\draw [dotted] (2.30506666008,-0.317178522042)-- (2.41227850569,0.397890600975);
\draw [shift={(1.88122068808,2.05762187735)}] plot[domain=4.66417753593:5.02206121401,variable=\t]({1.*1.7426216788*cos(\t r)+0.*1.7426216788*sin(\t r)},{0.*1.7426216788*cos(\t r)+1.*1.7426216788*sin(\t r)});
\draw [dotted] (1.69002707677,-0.358751658633)-- (1.74694296034,0.);
\draw (2.35262191543,0.)-- (2.41227850569,0.397890600975);
\draw (1.79778709185,0.316998669863)-- (1.74694296034,0.);
\draw [rotate around={-35.7321637372:(-0.481512480006,0.684675983971)}] (-0.481512480006,0.684675983971) ellipse (1.40047009309cm and 0.680257445958cm);
\draw [->] (-0.672093534067,0.533725470984) -- (-0.809174817081,0.380988939852);
\draw [->] (1.29090405739,0.802225740292) -- (1.52442036404,0.881806170504);
\draw [->] (1.76516092891,0.113583842101) -- (1.57237094995,0.130580324526);
\draw [->] (0.639998446081,0.) -- (0.639998446081,0.173203067561);
\draw [->] (2.24367915274,0.) -- (2.24367915274,0.183858753319);
\draw (0,0.1) node {$\Gamma_j$};
\draw (-1.1,-0.1) node {$\Gamma_j^-$};
\draw (1.3,-0.1) node {$\Gamma_j^+$};
\draw (0,-0.3) node {$\resub{\Omega}$};
\draw (0.67,-0.1) node {$\bfn_j$};
\draw (2.27,-0.1) node {$\bfn_j$};
\draw (-0.4,2) node {$\partial B_R$};
\draw (-0.35,1) node {$\gamma_1$};
\draw (-0.5,0.7) node {$\resub{\omega_1}$};
\draw (-0.6,0.4) node {$\bfn_{\gamma_1}$};
\draw (0.7,1.1) node {$\gamma_{2}$};
\draw (1,0.8) node {$\resub{\omega_2}$};
\draw (1.5,0.75) node {$\bfn_{\gamma_{2}}$};
\draw (2,0.1) node {$\resub{\omega_3}$};
\draw (2,0.4) node {${\gamma_3}$};
\draw (1.65,0.2) node {$\bfn_{\gamma_3}$};
\end{tikzpicture}
\caption{Configuration with (at least) four scatterers.
The relative upper half-plane $U_j$ is the area above the line $\Gamma_j^\infty=\Gamma_j^-\cup\Gamma_j\cup\Gamma_j^+$. 
Note the intersection of $\resub{\omega_3}$ (the right-hand scatterer) with $\Gamma_j^+\subset\Gamma_j^\infty$; $\bfn_j$ points into $\resub{\omega_3}\cap U_j$ whilst $\bfn_{\gamma_3}$ points out of $\resub{\omega_3}\cap U_j$ and into $D\cap U_j$.}
\label{fig:UHPBR}
\end{figure}
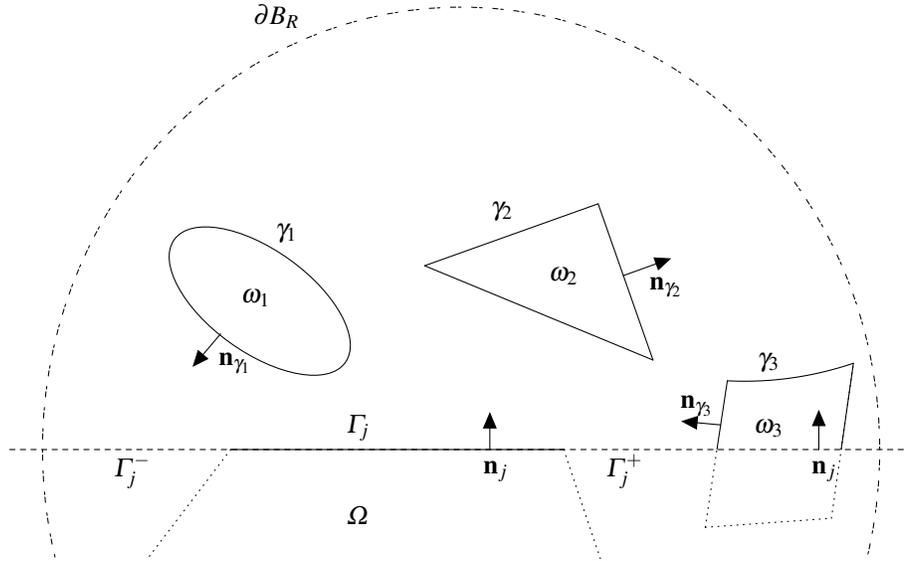

Green's second identity can now be applied to $G_j(\bfx,\cdot)$ and $u^s{=u-u^i}$ {for} 
${\bfx~\in}~D~\cap~ U_j~\cap~ B_R${, together with \eqref{eq:GPhi},} 
to obtain
\begin{align}
u^s(\bfx)=&\;2\int_{\Gamma^\infty_j\cap B_R\setminus\resub{\omega}}
\pdiv{\Phi_k(\bfx,\bfy)}{\bfn_j(\bfy)}u^s(\bfy)\dd{s(\bfy)}\nonumber\\
&+\int_{\gamma\cap U_j}\left[\pdiv{G_j(\bfx,\bfy)}{\bfn_\gamma(\bfy)}u^s(\bfy) - G_j(\bfx,\bfy)\pdiv{u^s(\bfy)}{\bfn_\gamma}
\right]\dd{s(\bfy)}\nonumber\\
&-\int_{\partial B_R\cap U_j}\left[\pdiv{G_j(\bfx,\bfy)}{r}u^s(\bfy) - G_j(\bfx,\bfy)\pdiv{u^s(\bfy)}{r}
\right]\dd{s(\bfy)},\label{G2rep1}
\end{align}
where $\pdivl{}{\bfn_j}=\bfn_j\cdot\nabla$ and $\pdivl{}{\bfn_\gamma}=\bfn_\gamma\cdot\nabla$, $\bfn_j$ and $\bfn_\gamma$ are the unit normal vector fields pointing {into} $D\cap U_j\cap B_R$ from $\Gamma^\infty_j\cap B_R\setminus\resub{\omega}$ and from $\gamma\cap U_j$, respectively, and $\pdivl{}{r}=\frac{\bfy}{|\bfy|}\cdot\nabla$ denotes the normal derivative on $\partial B_R\cap U_j$ pointing {out} of $D\cap U_j\cap B_R$.
As $R\rightarrow\infty$, the third integral vanishes by the same reasoning as in, e.g., \citet[Theorem~2.4]{CoKr:13}. 
The representation \eqref{G2rep1} then becomes
\begin{align}
u^s(\bfx)=2\int_{\Gamma^\infty_j\setminus\resub{\omega}}
\pdiv{\Phi_k(\bfx,\bfy)}{\bfn_j(\bfy)}u^s(\bfy)\dd{s(\bfy)}
+\int_{\gamma\cap U_j}\left[
\pdiv{G_j(\bfx,\bfy)}{\bfn_\gamma(\bfy)}u^s(\bfy) -G_j(\bfx,\bfy) \pdiv{u^s(\bfy)}{\bfn_\gamma}
\right]\dd{s(\bfy)},\label{us_rep}
\end{align}
for $\bfx\in U_j\setminus \resub{\omega}$. 

We now apply Green's second identity to $u^i$ and $G_j(\bfx,\bfy)$ in $U_j\cap\resub{\omega} $ and obtain, for $\bfx\in D\cap U_j$,
\begin{align}
\Bigg(\int_{\gamma\cap U_j}-&\int_{\Gamma_j^\infty\cap\resub{\omega} }\Bigg)
\left[\pdiv{G_j(\bfx,\bfy)}{\bfn(\bfy)}u^i(\bfy)-G_j(\bfx,\bfy)\pdiv{u^i}{\bfn}(\bfy)
\right]\dd{s}(\bfy)
\nonumber
\\
&=\int_{U_j\cap\resub{\omega} }\left[ 
\Delta G_j(\bfx,\bfy)u^i(\bfy)- G_j(\bfx,\bfy)\Delta u^i(\bfy)\right]\dd V(\bfy)
\label{G2_ingamma}
=0, 
\end{align}
as $u^i$ and $\Phi_k(\bfx,\cdot)$ satisfy the Helmholtz equation \eqref{Helmholtz} in $\resub{\omega}$ for $\bfx\in D\cap U_j$.
The sign of the boundary integral differs on the two parts of $\partial(U_j\cap\resub{\omega} )=(\gamma\cap U_j)\cup(\Gamma_j^\infty\cap\resub{\omega})$ 
because the normal derivative $\pdivl{}{\bfn}$ involves the outward-pointing normal vector $\bfn_{\gamma}$ on $\gamma\cap U_j$ and the inward-pointing normal $\bfn_j$ on $\Gamma_j^\infty\cap\resub{\omega}$, as depicted in Figure~\ref{fig:UHPBR}.

We then use $u^s=u-u^i$ to expand the last term in \eqref{us_rep}: for $\bfx\in D\cap U_j$
\begin{align*}
&\int_{\gamma\cap U_j} \left[\pdiv{G_j(\bfx,\bfy)}{\bfn_\gamma(\bfy)}u^s(\bfy)
-G_j(\bfx,\bfy)\pdiv{u^s(\bfy)}{\bfn_\gamma}\right]\dd{s}(\bfy)
\\
=&\int_{\gamma\cap U_j}
\Bigg[\pdiv{G_j(\bfx,\bfy)}{\bfn_\gamma(\bfy)}\big(\underbrace{u(\bfy)}_{=0}-u^i(\bfy)\big)
-G_j(\bfx,\bfy)\pdiv{(u-u^i)(\bfy)}{\bfn_\gamma}\Bigg]\dd{s}(\bfy)
\\
\overset{\eqref{G2_ingamma}}=&
-\int_{\gamma\cap U_j}G_j(\bfx,\bfy) \pdiv{u(\bfy)}{\bfn_\gamma}\dd{s}(\bfy)
+\int_{\Gamma_j^\infty\cap\resub{\omega} }
\left[-\pdiv{G_j(\bfx,\bfy)}{\bfn_j(\bfy)}u^i(\bfy)
+G_j(\bfx,\bfy)\pdiv{u^i(\bfy)}{\bfn_\gamma}\right]\dd{s}(\bfy).
\end{align*}
Substituting this expression in \eqref{us_rep} and using again {\eqref{eq:GPhi}},
we obtain a representation for~$u^s$:
\begin{align}
u^s(\bfx)=&\;2\int_{\Gamma^\infty_j\setminus\resub{\omega}}
\pdiv{\Phi_k(\bfx,\bfy)}{\bfn_j(\bfy)}u^s(\bfy)\dd{s(\bfy)}
-\int_{\gamma\cap U_j}G_j(\bfx,\bfy)\pdiv{u(\bfy)}{\bfn_\gamma}\dd{s(\bfy)}\nonumber\\
&-2\int_{\Gamma_j^\infty\cap\resub{\omega} }
\pdiv{\Phi_k(\bfx,\bfy)}{\bfn_j(\bfy)}u^i(\bfy)\dd{s(\bfy)},
\qquad\bfx\in D\cap U_j.\label{us_rep_1}
\end{align}
The final term will be non-zero only if $\Gamma_j^\infty\cap\resub{\omega} \ne \emptyset$, namely, in case one of the components of $\gamma$ is  $\Gamma_j^\infty$ (see {e.g.\ the component $\resub{\omega_3}$ in} Figure~\ref{fig:UHPBR}).

This integral representation must be combined with one for $u^i$ to construct a useful representation for $\pdivl{u}{\bfn}$ on $\Gamma$. 
The half-plane representation of \citet[\S3]{Ch:97} can be applied to upward propagating plane waves. 
We consider first the case $\bfn_j\cdot\bfd\geq0$, {which means that} 
$\Gamma_j$ is in shadow, from \citet[(3.3)]{ChLa:07}:
\[
u^i(\bfx)=2\int_{\Gamma^\infty_j}\pdiv{\Phi_k(\bfx,\bfy)}{\bfn_j(\bfy)}u^i(\bfy)\dd{s(\bfy)}, \quad\bfx\in U_j.
\]
Adding this to \eqref{us_rep_1} and taking the Neumann trace on $\Gamma_j$, we obtain a representation for the solution
\begin{align}
\pdiv{u(\bfx)}{\bfn}=&\;2\int_{\Gamma^\infty_j\setminus\resub{\omega}}\frac{\partial^2\Phi_k(\bfx,\bfy)}{\partial \bfn_j(\bfx)\partial \bfn_j(\bfy)}u(\bfy)\dd{s(\bfy)}
\nonumber
\\
&-2\int_{\gamma\cap U_j}\pdiv{\Phi_k(\bfx,\bfy)}{\bfn_j(\bfx)}
\pdiv{u(\bfy)}{\bfn_\gamma}\dd{s(\bfy)},
\qquad\bfx\in\Gamma_j,\quad \bfn_j\cdot\bfd\ge0.
\label{u_rep1}
\end{align}
For a downward-propagating wave $\bfn_j\cdot\bfd<0$, {i.e.\ when $\Gamma_j$ is illuminated by $u^i$,} we can apply the same result to the lower half-plane $\mathbb{R}^2\setminus \overline{U}_j$ (where the direction of the normal is reversed)
\[
u^i(\bfx)=-2\int_{\Gamma^\infty_j}\pdiv{\Phi_k(\bfx,\bfy)}{\bfn_j(\bfy)}u^i(\bfy)\dd{s(\bfy)}, \quad\bfx\in \mathbb{R}^2\setminus \overline{U}_j.
\]
Now define $u^r(\bfx):=-u^i(\tbfx^j)$ for $\bfx\in U_j$. 
Intuitively, $u^r$ may be considered the reflection of $u^i$ by a sound-soft line at $\Gamma_j^\infty$. 
It follows that {$\pdivl{u^r}{\bfn_j}=\pdivl{u^i}{\bfn_j}$ on $\Gamma_j^\infty$ and, for $\bfx\in U_j$,}
\[
u^r(\bfx)=2\int_{\Gamma^\infty_j}\pdiv{\Phi_k(\tbfx^j,\bfy)}{\bfn_j(\bfy)}u^i(\bfy)\dd{s(\bfy)}=
-2\int_{\Gamma^\infty_j}\pdiv{\Phi_k(\bfx,\bfy)}{\bfn_j(\bfy)}u^i(\bfy)\dd{s(\bfy)}.
\]
Rearranging this and adding $u^i$ gives
\[
u^i(\bfx)=u^i(\bfx)+u^r(\bfx)+2\int_{\Gamma^\infty_j}\pdiv{\Phi_k(\bfx,\bfy)}{\bfn_j(\bfy)}u^i(\bfy)\dd{s(\bfy)}.
\]
Summing with \eqref{us_rep_1} and taking the Neumann trace gives the representation for $\pdivl{u}{\bfn}$ on~$\Gamma_j$:
\begin{align}
\pdiv{u(\bfx)}{\bfn}=&\;2\pdiv{u^i(\bfx)}{\bfn}+2\int_{\Gamma^\infty_j\setminus\resub{\omega}}\frac{\partial^2\Phi_k(\bfx,\bfy)}{\partial \bfn_j(\bfx)\partial \bfn_j(\bfy)}u(\bfy)\dd{s(\bfy)}
\nonumber
\\
&-2\int_{\gamma\cap U_j}\pdiv{\Phi_k(\bfx,\bfy)}{\bfn_j(\bfx)}
\pdiv{u(\bfy)}{\bfn_\gamma}\dd{s(\bfy)},\qquad\bfx\in\Gamma_j,\quad \bfn_j\cdot\bfd<0,
\label{u_rep2}
\end{align}
{where we used again \eqref{eq:GPhi} and $\pdivl{u^r}{\bfn_j}=\pdivl{u^i}{\bfn_j}$ on $\Gamma_j$.}

The representation \eqref{u_rep1}--\eqref{u_rep2} may be viewed as a correction to the Physical Optics approximation for a single scatterer, which is defined as
\begin{equation}\label{POA}
\Psi(\bfx):=\left\{\begin{array}{cl}
2{\partial u^i(\bfx)}/{\partial \bfn}, & \bfx\in\Gamma_j\subset\Gamma: \bfn_j(\bfx)\cdot\bfd<0,\\
0, &\bfx\in\Gamma_j\subset\Gamma: \bfn_j(\bfx)\cdot\bfd\geq0.
\end{array}\right.
\end{equation}
Specifically, this correction can be split into two parts. 
The first integral of \eqref{u_rep1} and \eqref{u_rep2} represents the waves diffracted by the corners of $\Gamma$ (diffraction is ignored by the Physical Optics approximation
), whilst the second integral represents the correction to the waves reflected by the sides of $\Gamma$, as a result of the presence of $\resub{\omega}$. 
Unless the distance between the scatterers is sufficiently large, it is reasonable to expect the second correcting term to be not negligible\resub{; see \citet[Lemma~4.2]{Gi:17} for a precise quantification of this fact.}

We now write more explicitly the integral representation \eqref{u_rep1}--\eqref{u_rep2} in terms of the parametrisations of the segments $\Gamma_j$ and of their extensions $\Gamma_j^\infty$.
From the standard properties of Bessel functions (see, e.g., \citet[\S10]{DLMF}), we have that for $\bfx\in\Gamma_j$, $\bfy\in\Gamma^\pm_j\setminus\resub{\omega}$,
\begin{equation*}
\frac{\partial^2\Phi_k(\bfx,\bfy)}{\partial \bfn(\bfx)\partial \bfn(\bfy)}=\frac{\imag H^{(1)}_1(k|\bfx-\bfy|)}{4|\bfx-\bfy|}=\frac{\imag k^2}{4}\e^{\imag k |\bfx-\bfy|}\mu(k|\bfx-\bfy|),
\quad\text{where} \quad\mu(z):=\e^{-\imag z}\frac{H^{(1)}_1(z)}z,
\end{equation*}
see \citet[(3.6)]{ChLa:07}. 
To make use of this identity, we parametrise $\Gamma$ by
\begin{equation}\label{Gamma_param}
{\bfx_\Gamma}(s) = \mathbf{P}_j+\frac{s-\tL_{j-1}}{L_j}(\textbf{P}_{j+1}-\mathbf{P}_j),\quad s\in [\tL_{j-1},\tL_{j}),\quad j=1,\ldots,\numSides,
\end{equation}
where $L_j$ is the length of the $j$th side, $\mathbf{P}_j$ is the $j$th corner of $\Gamma$, and $\tL_j:=\sum_{\ell=1}^{j}{L}_\ell$ is the arc length up to the $(j+1)$th corner, with $\mathbf{P}_{{\numSides}+1} := \mathbf{P}_1$.
We will also denote by $L_\Gamma:=\tL_{\numSides}$ the total length of $\Gamma$.
Similarly we parametrise $\Gamma_j^-\cup\Gamma_j\cup\Gamma_j^+$ by
\[
{\bfy_j}(s) = \textbf{P}_j+\frac{s-\tL_{j-1}}{L_j}(\textbf{P}_{j+1}-\textbf{P}_j),\quad s\in \R,\quad j=1,\ldots,\numSides.
\]
We use \eqref{u_rep1}--\eqref{u_rep2} to represent the solution on a single side $\Gamma_j$, extending the ansatz of \citet{ChLa:07,HeLaMe:13_} to multiple scattering problems
\begin{align}\label{dudn_rep}
&
\pdiv{u}{\bfn}\big(\bfx_\Gamma(s)\big)=
\Psi\big(\bfx_\Gamma(s)\big)+v_j^+(s-\tL_{j-1})\e^{\imag k s}+v_j^-(\tL_{j}-s)\e^{-\imag k s}+\calG_{\gtG_j}
\left[\left.\pdiv{u}{\bfn}\right|_\gamma\right](\bfx_\Gamma(s)),
\nonumber\\&
\hspace{20mm} s\in\left[\tL_{j-1},\tL_{j}\right],\;j=1,\ldots,\numSides;
\end{align}
we shall now discuss each term in the ansatz separately.
Here $\Psi$ is the {Physical} Optics approximation \eqref{POA}, with the envelopes of the dif\-fract\-ed waves on each side defined by
\begin{align}
v_j^+(s)&:=\frac{\imag k^2}{2}\int_{(0,\infty)\setminus Z_j^+}
\mu\big(k(s+t)\big)\e^{\imag k (t-\tL_{j-1})}u\big(\bfy_j(\tL_{j-1}-t)\big)\dd{t},\ s\in[0,L_j],\label{vp}\\
v_j^-(s)&:=\frac{\imag k^2}{2}\int_{(0,\infty)\setminus Z_j^-}
\mu\big(k(s+t)\big)\e^{\imag k (\tL_{j}+t)}u\big(\bfy_j(\tL_{j}+t)\big)\dd{t},\ s\in[0,L_j],\label{vm}
\end{align}
where $Z_j^+:=\{ t\in \R:\bfy_j(\tL_{j-1}-t)\in\gamma\}$ and $Z_j^-:=\{ t\in \R:\bfy_j(\tL_{j}+t)\in\gamma\}$ are used to exclude from the integral the points inside $\resub{\omega}$ (as is the case for $\resub{\omega_3}$ of Figure \ref{fig:UHPBR}), to remain consistent with \eqref{u_rep1}--\eqref{u_rep2}. 
The \emph{interaction operator} $\calG_{\gtG_j}:L^2(\gamma)\to L^2(\Gamma_j)$ used in \eqref{dudn_rep}  is based on the final term of \eqref{u_rep1}--\eqref{u_rep2}, and is defined by
\begin{equation}\label{G_def}
\calG_{\gtG_j}\varphi(\bfx):=
-2\int_{\gamma\cap U_j}\pdiv{\Phi_k(\bfx,\bfy)}{\bfn_j(\bfx)}
\varphi(\bfy)\dd{s(\bfy)},\quad\bfx\in\Gamma_j\subset\Gamma,
\end{equation}
for $\varphi\in L^2(\gamma)$. 
We extend this definition to $\calG_\gtG:L^2(\gamma)\to L^2(\Gamma)$ as 
\begin{equation}\label{def:fullG}
\calG_\gtG\varphi :=
\calG_{\gtG_j}\varphi\quad\text{ on }\Gamma_j \quad\text{ for }j=1,\ldots,\numSides, \text{ and } \varphi\in L^2(\gamma).
\end{equation}

\begin{remark}\label{rem:vpm_not_the_same}
The ansatz \eqref{dudn_rep} is an extension of \citet[(3.9)]{ChLa:07} and \citet[(3.2)]{HeLaMe:13_}, with an additional term which relates the solution on $\Gamma$ to the solution on $\gamma$. 
It is important to note that this additional term is not the only term influenced by the presence of $\gamma$ and that one cannot solve for $v_\pm$ on a single scatterer and then add the $\calG_\gtG[\pdivl{u}{\bfn|_\gamma}]$ term. 
The reason for this is clear from \eqref{vp}--\eqref{vm}: even if $Z_j^\pm$ were of measure zero, so that the equations for \eqref{vp}--\eqref{vm} were identical to the case of a single scatterer, the integral contains $u$, which depends on the configuration $\partial D$.
Intuitively this makes sense,
diffracted waves emanating from the corners of $\Gamma$ will also be influenced by the presence of additional scatterers. 
\end{remark}

Many of the bounds which follow are explicit only in $k$ or the parameters which determine meshwidth or polynomial degree of an approximation space. Henceforth we will use $A\lesssim B$ to mean $A\leq CB$, where $C$ is a constant that depends only on the geometry of $\Upsilon$. 
To gauge the size of the contribution to the reflected waves on $\Gamma$ arising from the presence of $\resub{\omega}$, we require the following bound on the operator $\calG_\gtG$. 
\begin{lemma}\label{le:G_bound}
For $\partial D=\Gamma\cup\gamma$ with $\Gamma$ and $\gamma$ disjoint, we have the following bound on the interaction operator $\calG_\gtG$ defined in \eqref{def:fullG}, given $k_0>0$:
\begin{equation*}
\|\calG_\gtG\|_{L^{2}(\gamma)\shortrightarrow{L^{2}(\Gamma)}}\leq
C_\calG(k)\lesssim \sqrt{k},\quad\text{for }k\geq k_0,
\end{equation*}
where
\begin{equation}\label{def:C_calG}
C_\calG(k):=\sqrt{\frac{L_\Gamma L_\gamma k}{2\pi\dist(\Gamma,\gamma)}}+{\frac{\sqrt{L_\Gamma L_\gamma}}{\pi\dist(\Gamma,\gamma)}},
\end{equation}
where $L_\Gamma$ and $L_\gamma$ denote the perimeters of $\resub{\Omega}$ and $\resub{\omega}$ respectively.
\end{lemma}
\begin{proof}
For $0\neq\varphi\in L^2(\gamma)$,
using the Cauchy--Schwarz inequality, we can write
\begin{align*}
\frac{\|\calG_\gtG\varphi\|_{L^{2}(\Gamma)}}{\|\varphi\|_{L^{2}(\gamma)}}
&=
\frac{1}{\|\varphi\|_{L^{2}(\gamma)}}
\left( \sum_{j=1}^{\numSides}\int_{\Gamma_j} \left| 2 \int_{\gamma\cap U_j} \pdiv{\Phi_k(\bfx,\bfy)}{\bfn_j(\bfx)}\varphi(\bfy)\dd{s(\bfy)}\right|^2\dd{s(\bfx)}\right)^{1/2}
\\
&\leq
\frac{2}{\|\varphi\|_{L^{2}(\gamma)}}\left( \int_\Gamma \left\|\pdiv{\Phi_k(\bfx,\cdot)}{\bfn(\bfx)}\right\|_{L^{2}(\gamma)}^2\|\varphi\|_{L^{2}(\gamma)}^2\dd{s(\bfx)}\right)^{1/2}\\
&=2\left(\int_\Gamma\int_\gamma\left|\pdiv{\Phi_k(\bfx,\bfy)}{\bfn(\bfx)}\right|^2\dd{s(\bfy)}\dd{s(\bfx)}\right)^{1/2}\\
&\leq2\left(\int_\Gamma\dd{s}\int_\gamma\dd{s}\right)^{1/2}\sup_{\bfx\in \Gamma,\bfy\in \gamma}\left|\pdiv{\Phi_k(\bfx,\bfy)}{\bfn(\bfx)}\right|.
\end{align*}
The result follows from ${H^{(1)}_0}'(z)=-H^{(1)}_1(z)$ and \citet[(1.23)]{ChGrLaLi:09}, which states that $|H^{(1)}_1(z)|$ $\leq\sqrt{2/(\pi z)}+{2/(\pi z)}$ for $z>0$.
\end{proof}

As intuition would suggest, Lemma \ref{le:G_bound} confirms that the norm of the interaction operator \eqref{def:fullG} decreases as the obstacles move further apart, i.e., as the interaction between them decreases.

\subsection{Estimates of the {$L^\infty$} norm of the Helmholtz solution in {$D$}}\label{ss:u_max}

A value that will feature in many of the estimates for this method is
\begin{equation}\label{u_M_def}
\MSM:=\|u\|_{L^\infty(D)}.
\end{equation}
{The dependence of  $\MSM$ on the wavenumber $k$ is of key importance, as $\MSM$ appears as a multiplicative constant in the $hp$ best approximation result derived in \S\ref{s:Galerkin_method}, alongside a term which decreases exponentially with $p$. To show exponential convergence of the method, we therefore require that $\MSM$ grows at most algebraically with $k$.}
To explore this dependence, we will make use of the current best available bounds on the Dirichlet-to-Neumann map (see, e.g., \citet[\S2.7]{ACTA}) for multiple obstacle configurations. Recently in \citet{ChSpGiSm:17} such bounds have been developed for $(R_0,R_1)$ configurations (of Definition \ref{def:R0R1}), enabling our analysis to cover a much broader range of configurations. To relate these to estimates for \eqref{u_M_def}, we require the following continuity bound for the single layer potential.

\begin{lemma}\label{lem:S_k_bd}
For a domain $D$ with bounded Lipschitz boundary $\partial D$, given $k_0>0$ the following bound on the single layer potential \eqref{def:Sk} holds
\begin{equation}\label{bd:SLgen}
\|S_k\|_{L^2(\partial D)\shortrightarrow L^\infty( D)}
\lesssim
k^{-1/2}\log^{1/2}(1+k\diam{(\partial D)}),\quad k\geq k_0.
\end{equation}

\end{lemma}
 \begin{proof}
 	It is straightforward to show (see, e.g., \citet[Lemma~4.1]{HeLaMe:13_})  that 
 	\begin{equation}\label{eq:SkNorm}
 	\Norm{S_k}_{L^2(\partial D)\to L^\infty(D)}\leq \esssup{\bfp\in D}\Norm{\Phi_k(\bfp,\cdot)}_{L^2(\partial D)}.
 	\end{equation}
 	
 	We shall exploit the Lipschitz property of $\partial D$, by defining a finite set of Lipschitz graphs which describe its geometry, and bounding \resub{the right-hand side of }\eqref{eq:SkNorm} in terms of the coordinates describing these graphs. Let $\{W_j\}$, $j=1,\ldots,N$, be a finite open cover of $\partial D$ as in the definition of a Lipschitz domain (see, e.g., \citet[3.28]{Mc:00}).
 	Assume without loss of generality that each $W_j\cap \partial D$ is connected. Each $W_j\cap \partial D$ is part of the graph of a Lipschitz real function $\ell_j$ in rotated Cartesian coordinates, which we denote $(x_j,y_j)$. The boundary $\partial D$ can thus be decomposed into $N_D$ arcs $\alpha_j$ (with disjoint relative interiors) that are the graph of $\ell_j:[a_j,b_j]\to\R$, i.e.\ $\alpha_j=\{(x_j,y_j)\in\R^2:\;a_j\le x_j\le b_j,\;y_j=\ell_j(x_j)\}\subset W_j\cap\partial D$, and $\partial D=\bigcup_{j=1}^{N_D}\alpha_j$. Denote by $C_\ell$ a constant which bounds above the Lipschitz constant of every Lipschitz graph function $\ell_j$. Fix any $\bfp\in\R^2$. For each $j=1,\ldots,N_D$ denote by $(p_{x,j},p_{y,j})$ the coordinates of $\bfp$ in the $(x_j,y_j)$ coordinate system.
 	We have $\max\{|a_j-p_{x,j}|,|b_j-p_{x,j}|\}\le \max_{\bfq\in\partial D}|\bfp-\bfq|$.
 	
 	Now we have established the necessary notation, we decompose the integral in the $L^2(\partial D)$ norm on the right-hand side of \eqref{eq:SkNorm} into the regions contained within the open sets $W_j$, each with its own Lipschitz graph $\alpha_j$:
 	\begin{align*}
 	\Norm{\Phi_k(\bfp,\cdot)}_{L^2(\partial D)}^2
 	&=\sum_{j=1}^{N_D} \int_{\alpha_j} |\Phi_k(\bfp,\bfy)|^2\dd s(\bfy)\\
 	&=\frac1{16}\sum_{j=1}^{N_D} \int_{\alpha_j} \Big|H^{(1)}_0\big(k|\bfp-\bfy|\big)\Big|^2\dd s(\bfy)\\
 	&=\frac1{16}\sum_{j=1}^{N_D} \int_{a_j}^{b_j} 
 	\Big|H^{(1)}_0\Big(k\sqrt{\big(x_j-p_{x,j}\big)^2+\big(f(x_j)-p_{y,j}\big)^2}\Big)\Big|^2
 	\sqrt{1+|\ell_j'(x_j)|^2}\dd x_j.
 	\end{align*}
 	Now we may appeal to the monotonicity of $|H^{(1)}_0|$, and bound the variation of the mapping to the Lipschitz graph $\ell'$ by the constant $C_\ell$ to obtain
 	\begin{align*}
 	\Norm{\Phi_k(\bfp,\cdot)}_{L^2(\partial D)}^2&\le\frac{\sqrt{1+C_\ell^2}}{16}\sum_{j=1}^{N_D} \int_{a_j}^{b_j} 
 	\Big|H^{(1)}_0\big(k|x_j-p_{x,j}|\big)\Big|^2\dd x_j \qquad\qquad\\
 	&=\frac1k\frac{\sqrt{1+C_\ell^2}}{16}\sum_{j=1}^{N_D} \int_{k(a_j-p_{x,j})}^{k(b_j-p_{x,j})} \Big|H^{(1)}_0(|s|)\Big|^2\dd s,
 	\end{align*}
 	where we have changed integration variables to simplify the integrand in the second step. Since $b_j-a_j\leq R_D:=\diam(\partial D)$, we can bound further
 	\begin{align}
 	\hspace{-1cm}	\Norm{\Phi_k(\bfp,\cdot)}_{L^2(\partial D)}^2 \leq&\frac1k\frac{\sqrt{1+C_\ell^2}}{16}\sum_{j=1}^{N_D} \int_{k(a_j-p_{x,j})}^{k(a_j-p_{x,j}+R_D)} \Big|H^{(1)}_0(|s|)\Big|^2\dd s\nonumber\\
 	\leq&\frac1k\frac{\sqrt{1+C_\ell^2}}{16}\sum_{j=1}^{N_D} \bigg(
 	\int_{(k(a_j-p_{x,j}),k(a_j-p_{x,j}+R_D))\cap(-1,1)} \Big|H^{(1)}_0(|s|)\Big|^2\dd s\label{intSplit_1}\\
 	&+\int_{(k(a_j-p_{x,j}),k(a_j-p_{x,j}+R_D))\setminus(-1,1)} \Big|H^{(1)}_0(|s|)\Big|^2\dd s\bigg).\label{intSplit_2}
 		\end{align}
 	We have split the integrals in order to bound the Hankel function, using $|H^{(1)}_0(z)|\le \hat c(1+|\log |z||)$ if $0<|z|\le1$ with \eqref{intSplit_1}, and $|H^{(1)}_0(z)|\le \hat c |z|^{-1/2}$ if $|z|>1$, by e.g.\ \citet[p.~638]{HeLaMe:13_} (with value $\hat c\approx2.09$) with \eqref{intSplit_2}. The integral \eqref{intSplit_1} is therefore bounded above by
 	\begin{equation}\label{intSplit_i}
 	2\hat{c}^2\int_0^1 (1+|\log s|)^2 = 10\hat c^2,
 	\end{equation}
 	where we have used $\int^t(1-\log s)^2\dd s=t(\log^2 t-4\log t+5)+$constant in the final step. The second integral \eqref{intSplit_2} is maximised either when (i) $k(a_j-p_{x,j})=1$ or when (ii) $k(a_j-p_{x,j})=-kR_D/2$. In case (i), the integral is bounded above by
 	\[
 	\hat c^2\int_1^{1+kR_D}s^{-1}\dd s =\hat c^2 \log(1+kR_D)
 	\]
 	and in case (ii) it is bounded above by
 	\[
 	2\hat c^2\int_1^{kR_D/2}s^{-1}\dd s = 2\hat c^2\log(kR_D/2),
 	\]
 	so in either case, \eqref{intSplit_2} is bounded above by $2 \hat c^2\log(1+kR_D)$. Combining this with \eqref{intSplit_i} yields
 	\[
 	\Norm{\Phi_k(\bfp,\cdot)}_{L^2(\partial D)}^2 \le \frac1k {N_D} \hat c^2\frac{\sqrt{1+C_\ell^2}}{8} \Big(5+\log (1+kR_D)\Big),
 	\]
 	This gives the explicit form of the simplified estimate in our claim, proving the assertion.
 	
 \end{proof}

Using this result, we can say more about the $k$-dependence of $\MSM$, for a large class of multiple scattering configurations of interest.
{\begin{theorem}\label{Mu_bdMS}
Suppose that $u$ satisfies the BVP \eqref{Helmholtz}--\eqref{SRC}, with plane wave incidence \eqref{eq:PW}. Then given $k_0{>0}$ independent of $k$, the following bounds hold:
\begin{enumerate}[(i)]
\item If $\Upsilon=\resub{\Omega}\cup\resub{\omega}$ is a non-trapping polygon (in the sense of Definition \ref{def:nontrapping}),
\[
\MSM\lesssim k^{1/2}\log^{1/2}(1+k\diam({\partial D})),\quad\text{for } k\geq k_0.
\]
\item Otherwise, if $\Upsilon=\resub{\Omega}\cup\resub{\omega}$ is an $(R_0,R_1)$ domain (in the sense of Definition \ref{def:R0R1}),
\[
\MSM\lesssim k^{5/2}\log^{1/2}(1+k\diam({\partial D})),\quad\text{for } k\geq k_0,
\]
\end{enumerate}
where $\MSM$ is as in \eqref{u_M_def}.
\end{theorem}}
\begin{proof}
We write the BVP \eqref{Helmholtz}--\eqref{SRC} for the scattered field $u^s$, with Dirichlet data $u^s=-u^i $ on the boundary $\partial D $, in terms of the Dirichlet-to-Neumann (DtN) map $P_{\text{DtN}}$ (see, e.g., \citet[\S2.7]{ACTA}) as $\pdivl{u^s}{\bfn}=-P_{\text{DtN}}\tau_+u^i$, where $\tau_+$ denotes the exterior Dirichlet trace.
The representation \eqref{Green} gives
\[
u^s=-S_k\left(\pdiv{}{\bfn}-P_\text{DtN}\tau_+\right)u^i ,\quad\text{in }D.
\]
This, together with $|\pdivl{u^i}{\bfn}|\leq k|u^i|$ (which follows immediately from \eqref{eq:PW}), enables us to bound $u^s$ as
\begin{equation}\label{bd:usDtN}
\|u^s\|_{L^\infty(D)}\leq \|{S}_k\|_{L^2(\partial D )\shortrightarrow {L^\infty(D)}}\left(1+\|P_\text{DtN}\|_{H^1_k(\partial D )\shortrightarrow L^2(\partial D )}\right)\|u^i \|_{H^1_k(\partial D )},
\end{equation}
where $\|\cdot\|_{H^1_k(\partial D)}$ denotes the the $k$-weighted norm of the Sobolev space $H^1(\partial D)$ 
\begin{equation}\label{def:HkNorm}
\|\varphi\|_{H^1_k(\partial D)}:=
\left(\int_{\partial D} k^2|\varphi|^2+|\nabla_S\varphi|^2\dd{V} \right)^{1/2}
\end{equation}
and $\nabla_S$ denotes the surface gradient operator on $\partial D$ (defined in \eqref{eq:SGkernel}).
{By the triangle inequality we have $\MSM\le \|u^i\|_{L^\infty(D)}+\|u^s\|_{L^\infty(D)}$, and from Lemma \ref{lem:S_k_bd} we can bound $\|{S}_k\|_{L^2(\partial D )\shortrightarrow L^\infty( D )}$. Hence we may write, for $k\geq k_0$,
\begin{equation}\label{us_bd}
\MSM\lesssim \|u^i\|_{L^\infty(D)}+k^{-1/2}\log^{1/2}(1+k\diam(\partial D))\|P_\text{DtN}\|_{H^1_k(\partial D )\shortrightarrow L^2(\partial D )}\; \|u^i \|_{H^1_k(\partial D )}.
\end{equation}
For the DtN maps, we may use \citet[Theorem~1.4]{BaSpWu:16} for the non-trapping polygon case \emph{(i)} $\|P_\text{DtN}\|_{H^1_k(\partial D )\shortrightarrow L^2(\partial D )}\lesssim 1$, whilst the $(R_0,R_1)$ obstacle case \emph{(ii)} $\|P_\text{DtN}\|_{H^1_k(\partial D )\shortrightarrow L^2(\partial D )}\lesssim k^2$ follows by \citet[Theorem~1.8]{ChSpGiSm:17}. It remains to bound the incident field $u^i$  at the boundary and in the domain. For plane wave incidence, it follows by the definitions  \eqref{def:HkNorm} and \eqref{eq:PW} that $\|u^i\|_{H^1_k(\partial D )}\leq2\sqrt{|\Gamma|}k$ and $\|u^i\|_{L^\infty(D)}=1$. 
The result follows by combining these bounds on $u^i$ with the components of \eqref{us_bd}.
}\end{proof}

Theorem \ref{Mu_bdMS} is a generalisation of \citet[Theorem~4.3]{HeLaMe:13_}, which bounds $\MSM$ for star-shaped polygons. Although more general, Theorem \ref{Mu_bdMS} differs from \citet[Theorem~4.3]{HeLaMe:13_} in that it is not fully explicit in terms of the geometric parameters of $\Upsilon$.
 We do not expect such a bound to hold for the most general configurations and incident fields, since it was shown in \citet[Theorem~2.8]{BeChGrLaLi:11} that there exist multiple obstacle configurations for which {$\|\calA_{k,\eta}^{-1}\|_{L^2(\partial D)\shortrightarrow L^2(\partial D)}$} is bounded below by a term which grows exponentially with $k$, in which case $\MSM$ would grow similarly. 
In particular though, Theorem \ref{Mu_bdMS}(i) is immediately applicable to the case of polygons which are non-convex, non-star-shaped and non-trapping, considered in \citet{ChHeLaTw:15} (see Definition~3.1 therein), for which the stronger result $\MSM=\mathcal O(1)$ for $k\to\infty$ was conjectured, in the (then) absence of any available algebraic bounds.
The bound of Theorem \ref{Mu_bdMS} is sufficient to guarantee algebraic growth of $u_{\max}(k)$ in $k$, and therefore exponential convergence of{ HNA-BEM for} such polygons.  

The following assumption generalises Theorem \ref{Mu_bdMS} to all configurations of interest.
\begin{assumption}\label{as:u_max_alg}
For the solution $u$ of the BVP \eqref{Helmholtz}--\eqref{SRC}, we assume that there exist $\beta\ge0${, $k_0>0$} and ${C_u}>0$, independent of $k$, such that
\[
\MSM\leq C_u k^\beta \qquad \text{for }k\geq k_0,
\]
that is $\MSM$ of \eqref{u_M_def} has at most algebraic dependence on the wavenumber $k$.
\end{assumption}
Clearly{ Assumption~\ref{as:u_max_alg} holds for configurations satisfying the conditions of
Theorem~\ref{Mu_bdMS} (see Remark \ref{re:params} for more details).}

\subsection{Analyticity and bounds for the envelope functions {$v^\pm_j$}} \label{sec:anal}
Additional notation is required for the estimates that follow. 
Denote by $\Omega_j$ the exterior angle at the corner $\mathbf{P}_j$ {of $\resub{\Omega}$} (see figure \ref{fig:extraParams} for an illustrative example). 
{Since $\resub{\Omega}$ is a} convex polygon, $\Omega_j\in(\pi,2\pi)$ for all $j=1,\ldots,\numSides$. Let $c_*>0$ be a constant such that $kL_j\geq c_*$ for all $j=1,\ldots,\numSides$ (e.g. $c_*=\min_{j=1,\ldots,\numSides}\{kL_j\}$).

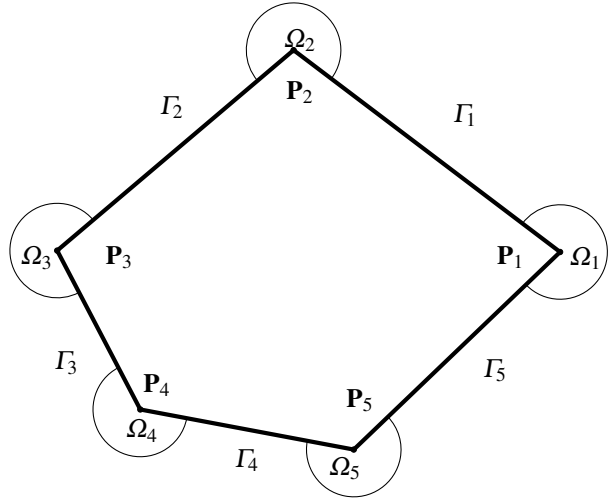
\begin{figure}[htb]
\centering
\definecolor{qqwuqq}{rgb}{0,0,0}
\definecolor{wqwqwq}{rgb}{0,0,0}
\begin{tikzpicture}[line cap=round,line join=round,>=triangle 45,x=1.0cm,y=1.0cm]
\clip(1.5,0.5) rectangle (10.,7.6);
\fill[line width=1.6pt,color=wqwqwq,fill=wqwqwq,fill opacity=0.] (2.4704761904761905,4.091428571428571) -- (5.6133333333333315,6.758095238095238) -- (9.156190476190478,4.072380952380952) -- (6.413333333333333,1.4438095238095239) -- (3.5752380952380944,1.9771428571428575) -- cycle;
\draw [shift={(5.6133333333333315,6.758095238095238)},color=qqwuqq,fill=qqwuqq,fill opacity=0.] (0,0) -- (-37.16447489351142:0.6285714285714284) arc (-37.16447489351142:220.3141001604973:0.6285714285714284) -- cycle;
\draw [shift={(2.4704761904761905,4.091428571428571)},color=qqwuqq,fill=qqwuqq,fill opacity=0.] (0,0) -- (40.31410016049732:0.6285714285714284) arc (40.31410016049732:297.58808136574584:0.6285714285714284) -- cycle;
\draw [shift={(3.5752380952380944,1.9771428571428575)},color=qqwuqq,fill=qqwuqq,fill opacity=0.] (0,0) -- (117.58808136574582:0.6285714285714284) arc (117.58808136574582:349.35712925529646:0.6285714285714284) -- cycle;
\draw [shift={(6.413333333333333,1.4438095238095239)},color=qqwuqq,fill=qqwuqq,fill opacity=0.] (0,0) -- (-190.64287074470357:0.6285714285714284) arc (-190.64287074470357:43.78112476486871:0.6285714285714284) -- cycle;
\draw [shift={(9.156190476190478,4.072380952380952)},color=qqwuqq,fill=qqwuqq,fill opacity=0.] (0,0) -- (-136.21887523513132:0.6285714285714284) arc (-136.21887523513132:142.8355251064886:0.6285714285714284) -- cycle;
\draw [line width=1.6pt,color=wqwqwq] (2.4704761904761905,4.091428571428571)-- (5.6133333333333315,6.758095238095238);
\draw [line width=1.6pt,color=wqwqwq] (5.6133333333333315,6.758095238095238)-- (9.156190476190478,4.072380952380952);
\draw [line width=1.6pt,color=wqwqwq] (9.156190476190478,4.072380952380952)-- (6.413333333333333,1.4438095238095239);
\draw [line width=1.6pt,color=wqwqwq] (6.413333333333333,1.4438095238095239)-- (3.5752380952380944,1.9771428571428575);
\draw [line width=1.6pt,color=wqwqwq] (3.5752380952380944,1.9771428571428575)-- (2.4704761904761905,4.091428571428571);
\begin{scriptsize}
\draw [fill=black] (2.4704761904761905,4.091428571428571) circle (1.0pt);
\draw [fill=black] (5.6133333333333315,6.758095238095238) circle (1.0pt);
\draw [fill=black] (9.156190476190478,4.072380952380952) circle (1.0pt);
\draw [fill=black] (6.413333333333333,1.4438095238095239) circle (1.0pt);
\draw [fill=black] (3.5752380952380944,1.9771428571428575) circle (1.0pt);
\end{scriptsize}

\draw (7.9,5.9) node{$\Gamma_1$};
\draw (4,6) node{$\Gamma_2$};
\draw (2.6,2.6) node{$\Gamma_3$};
\draw (5,1.3) node{$\Gamma_4$};
\draw (8.3,2.5) node{$\Gamma_5$};

\draw (8.5,4) node{$\bfP_1$}; \draw (9.5,4) node{$\Omega_1$};
\draw (5.7,6.2) node{$\bfP_2$}; \draw (5.7,6.9) node{$\Omega_2$};
\draw (3.3,4) node{$\bfP_3$}; \draw (2.2,4) node{$\Omega_3$};
\draw (3.8,2.3) node{$\bfP_4$}; \draw (3.6,1.7) node{$\Omega_4$};
\draw (6.5,2.1) node{$\bfP_5$}; \draw (6.3,1.2) node{$\Omega_5$};
\end{tikzpicture}
\caption{A convex polygon with the parameters introduced in \S\ref{sec:anal}.}
\label{fig:extraParams}
\end{figure}

We now aim to show, as in \citet{HeLaMe:13_} where only one (convex polygonal) scatterer $\resub{\Omega}$ is present, that the functions $v_j^\pm$ are complex-analytic, and moreover that they can be approximated much more efficiently than $\partial u/\partial \bfn|_\Gamma$. We update this to the multiple scattering configuration by adapting the intermediate results of \citet[\S3]{HeLaMe:13_}. We first consider the solution behaviour near the corners.
\begin{lemma}[Solution behaviour near the corners]\label{lem:near_corners}
Suppose that $u$ satisfies the BVP \eqref{Helmholtz}--\eqref{SRC} and $\bfx\in D$ satisfies $r:=|\bfx-\mathbf{P}_j|\in (0,1/k]$,  
and $r<\dist(\mathbf{P}_j,\gamma)$.
Then there exists a constant $C>0$, depending only on $\partial D$ and $c_*$, such that (with $\MSM$ as in \eqref{u_M_def}),
\[
|u(\bfx)|\leq C(kr)^{\pi/\Omega_j}\MSM.
\]
\end{lemma}
\begin{proof}
Follows identical arguments to \citet[Lemma 3.5]{HeLaMe:13_}, with the slight modification to the definition $R_j:=\min\{L_{j-1},L_j,\pi/(2k),\dist(\mathbf{P}_j,\gamma)\}$, which ensures only areas close to the corner $\mathbf{P}_j$ inside $D$ are considered.
\end{proof}

Now we may bound the singular behaviour of the diffracted envelopes $v_j^\pm$, which will enable us to choose a suitable approximation space for the numerical method.
\begin{theorem}\label{th:big_bounds}
Suppose that $u$ is a solution of the BVP \eqref{Helmholtz}--\eqref{SRC}, and that $c_r\in(0,1]$ is chosen such that $\dist(\{\mathbf{P}_j:{j=1,\ldots,\numSides}\},\gamma)>c_r/k$. Then the diffracted wave envelope components $v_j^\pm$ for $ j=1,\ldots, \numSides$, of the boundary representation \eqref{dudn_rep}, are analytic in the right {complex} half-plane $\re[s]>0$, where they satisfy the bounds
\begin{equation*}
|v_j^\pm(s)|\leq 
\begin{cases}
C_j^\pm \MSM \big(k|ks|^{-\delta^\pm_j}+k(k|s|+c_r)^{-1}\big),& 0 <|s|\leq1/k,\\
C_j^\pm \MSM k|ks|^{-1/2},& |s|>1/k,
\end{cases}
\end{equation*}
where $\delta^+_j,\delta^-_j \in (0,1/2)$ are given by $\delta^+_j:=1-\pi/\Omega_j$ and $\delta^-_j:=1-\pi/\Omega_{j+1}$. The constant $C^+_j$ depends only on $c_*$, $c_r$ and $\Omega_j$, whilst the constant $C^-_j$ depends only on $c_*$, $c_r$ and $\Omega_{j+1}$.
\end{theorem}
\begin{proof}
The analyticity of the functions $v^\pm_j(s)$ in $\re[s]>0$ follows from their definition \eqref{vp}--\eqref{vm} and the analyticity of $\mu(s)$ in the same set, which is shown in \citet[Lemma~3.4]{HeLaMe:13_}. 
The estimate of $|v_j^\pm(s)|$ for $|s|>1/k$ follows as in the proof of \citet[Theorem~3.2]{HeLaMe:13_}. Here we show for $v^+_j$, the proof for $v^-_j$ follows similar arguments.
For $|s|\le1/k$, the definition \eqref{vp}--\eqref{vm} of $v_j^+$ gives
\begin{align*}
|v_j^+(s)|&\leq  \frac{k^2}2 \int_{(0,c_r/k)} \big|\mu\big(k(s+t)\big)\big| 
\big|u\big(\bfy_j(\tL_{j-1}-t)\big)\big|\dd t\\
&\qquad + \frac{k^2}2 \int_{(c_r/k,\infty)\setminus Z_j^+} \big|\mu\big(k(s+t)\big)\big| 
\big|u\big(\bfy_j(\tL_{j-1}-t)\big)\big|\dd t.
\end{align*}
Since $c_r\le1$ and thanks to Lemma~\ref{lem:near_corners}, the first integral is bounded as in the proof of \citet[Theorem~3.2]{HeLaMe:13_}, leading to the term $\MSM k|ks|^{-\delta^\pm_j}$ in the assertion.
Using the bound on $\mu$ from \citet[Lemma~3.4]{HeLaMe:13_}, 
we control the second integral as 
\begin{align*}
\frac{k^2}2 \int_{(c_r/k,\infty)\setminus Z_j^+} &\big|\mu\big(k(s+t)\big)\big| 
\big|u\big(\bfy_j(\tL_{j-1}-t)\big)\big|\dd t\\
&\le  C \MSM k^2 \int_{c_r/k}^\infty 
\big|k(s+t)\big|^{-3/2}\Big(\big|k(s+t)\big|^{-1/2}+(\pi/2)^{1/2}\Big)\dd t\\
&\le C \MSM k^2 \Big(k^{-2}(|s|+c_r/k)^{-1}+k^{-3/2}(|s|+c_r/k)^{-1/2}\Big)\\
&= C \MSM  k\Big((k|s|+c_r)^{-1}+(k|s|+c_r)^{-1/2}\Big).
\end{align*}
The bound in the assertion follows by noting that $k|s|+c_r<2$.
\end{proof}

The constant $c_r$ is small when the scatterers are close together, relative to the  wavelength of the problem.
Thus the terms containing $c_r$ in the bound of Theorem~\ref{th:big_bounds} control the effect of the separation between $\resub{\Omega}$ and $\resub{\omega}$ on the singular behaviour of $v^\pm_j$.
However, the method we present is designed for high-frequency problems, and to maintain
$c_r=O(1)$ as $k$ increases, the separation of the scatterers is allowed to decrease inversely proportional to $k$. 
Hence, for the configurations that we consider of practical interest in the high-frequency regime, the condition \eqref{th:c_r_is_0} in the following corollary will hold.

\begin{corollary}\label{th:c_r_is_0}
Suppose that the conditions of Lemma \ref{lem:near_corners} hold, with the additional constraint that the \emph{separation condition}
\begin{equation}\label{as:c_r_is_0}
\dist(\Gamma,\gamma)\ge1/k,
\end{equation}
is satisfied. 
It then follows that the first bound of Theorem \ref{th:big_bounds} can be simplified to
\[
|v^\pm_j(s)|\le 
C_j^\pm \MSM k|ks|^{-\delta_j^\pm},\quad \text{for }0<|s|\leq1/k, \quad j=1,\ldots,\numSides.
\]
\end{corollary}

\begin{proof}
If the separation condition \eqref{as:c_r_is_0} holds, we can choose $c_r=1$ in Theorem \ref{th:big_bounds}, {from which $(k|s|+c_r)^{-1}\leq1$.
The term $k(k|s|+c_r)^{-1}$ is therefore dominated by the term $k|ks|^{\delta^\pm_j}$ for $0<|s|\leq1/k$.
}
\end{proof}

The separation condition \eqref{as:c_r_is_0} aligns the bounds of Theorem \ref{th:big_bounds} with the well-studied single scattering HNA configurations of \citet[Theorem 5.2]{HeLaMe:13_}. Hence, all best approximation results for the single scattering case may be applied to the approximation on $\Gamma$ in the multiple scattering problems we consider here.

\begin{remark}\label{re:MSHNA_uigen}
The result of Theorem \ref{th:big_bounds} may be extended to {an incident wave of source-type}, for example the point source emanating from $\bfs\in D$, $u^i(\bfx)=H^{(1)}_0(k|\bfx-\bfs|)$. 
This requires that the position of the source point $\bfs$ is separated by a distance of at least $1/k$ from $\resub{\Omega}$ (similar to the separation condition \eqref{as:c_r_is_0}), see \citet[\S3.2]{Gi:17} for details.
\end{remark}


\section{{$hp$} approximation space}\label{s:approx_space}
We will combine two approximation spaces: the HNA-BEM space on $\Gamma$ and a standard $hp$-BEM space on $\gamma$. Hereafter, using the parametrisation of the boundaries $\Gamma$ and $\gamma$, we identify $L^2(\Gamma_j)$ with $L^2(0,L_j)$, and $L^2(\gamma)$ with $L^2(0,L_\gamma)$.

\subsection{HNA-BEM approximation on {$\Gamma$}}\label{s:HNAspace}
As in previous HNA methods, on $\Gamma$ we approximate only the diffracted waves
\begin{equation}\label{diff_rep}
v_\Gamma(s):=\frac{1}{k}\left( v_j^+(s-\tL_{j-1})\e^{\imag k s}
+v_j^-(\tL_{j}-s)\e^{-\imag k s} \right),
\quad s\in\left[\tL_{j-1},\tL_{j}\right],\;j=1,\ldots,\numSides,
\end{equation}
where $v_j^\pm$ are as in \eqref{vp}--\eqref{vm}, and broadly speaking this is done using basis elements of the form
\begin{equation*}
v_\Gamma(s)\approx\left( P_j^+(s-\tL_{j-1})\e^{\imag k s}+P_j^-(\tL_{j}-s)\e^{-\imag k s} \right),\quad s\in\left[\tL_{j-1},\tL_{j}\right],\;j=1,\ldots,\numSides,
\end{equation*}
where $P_j^\pm$ are piecewise polynomials on a graded mesh. 
There are two well-studied classes of $hp$ approximation space we may use to do this. Both spaces consist of piecewise polynomials multiplied by oscillatory functions oscillating in both directions along the surface of $\Gamma$, and both spaces are constructed on meshes graded towards the singularities at the corners of $\Gamma$. 
{We briefly describe these approximation spaces here:}
\begin{enumerate}[(i)]
\item \emph{The overlapping-mesh space}, used in original HNA methods for single scatterers, this discrete space is the {sum} of two subspaces, each constructed on a separate mesh graded in opposite directions. 
{On $\Gamma_j$, the subspace on the mesh graded towards $\tL_{j-1}$ is used to approximate $v_j^+(s-\tL_{j-1})\e^{\imag k s}$ and the subspace on the mesh graded towards $\tL_j$ is used to approximate $v_j^-(\tL_{j}-s)\e^{-\imag k s}$.}
Details can be found in \citet[\S5]{HeLaMe:13_}.
\item \emph{The single-mesh space}, constructed on a single mesh graded towards both edges. 
{This space can easily be implemented by adapting a standard BEM code, as the mesh is of a more standard type.}
However, care must be taken close to the corners of $\Gamma$: certain elements must be removed from the approximation space to ensure the discrete system does not become too ill-conditioned. 
We will define this space shortly.
\end{enumerate}

A range of numerical experiments comparing both approximation spaces for collocation HNA-BEM can be found in \citet{Pa:15}. 
For either choice of mesh, we denote by $\numLayersj$ the number of grading layers and by $p_j$ the maximum polynomial degree on the $j$th side (in terms of the notation of \citet{ChLa:07} and \citet{HeLaMe:13_}, we choose $p_j=p^+_j=p^-_j$, $\numLayersj=\numLayersj^+=\numLayersj_-$ for simplicity).
We denote by $\sigma>0$ the grading parameter, so that the smallest mesh element of $\Gamma_j$ (touching the corners of $\Gamma_j$) has length $L_j\sigma^{n_j}$. 

The single-mesh space has been described in the {theses} \citet{Gi:17,Pa:15} and is used for the numerical experiments in \S\ref{s:results}; we define it here for convenience.


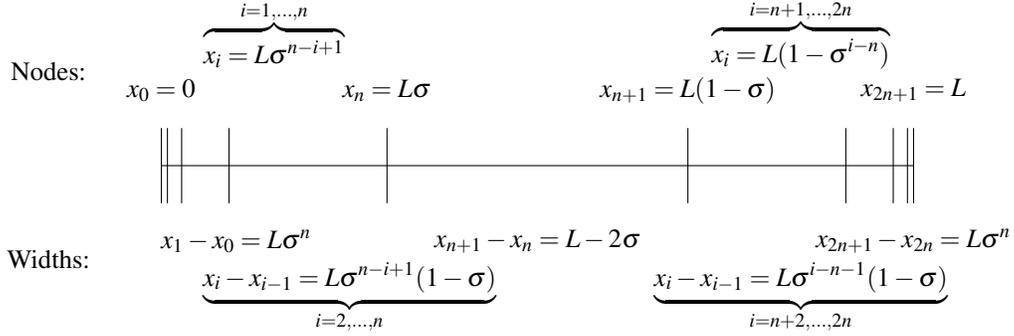
\begin{figure}[htb]\centering
\begin{tikzpicture}
\newcommand{\meshheight}{0.5}
\draw (-1.5,2.5*\meshheight) node {Nodes:};
\draw (-1.5,-2.5*\meshheight) node {Widths:};
\draw (0,-\meshheight) -- (0,\meshheight);
\draw (0,2*\meshheight) node {$x_0=0$};
\draw (1,-\meshheight-0.5) node {${x_1-x_0=L\sigma^{\numLayers}}$};
\draw (0.9,-\meshheight) -- (0.9,\meshheight);
\draw (1.5,3.5*\meshheight) node {$\overbrace{x_i=L\sigma^{\numLayers-i+1}}^{i=1,\ldots,\numLayers}$};
\draw (2.5,-3.5*\meshheight) node {$\underbrace{x_i-x_{i-1}=L\sigma^{\numLayers-i+1}(1-\sigma)}_{i= 2,\ldots,\numLayers}$};
\draw (3,-\meshheight) -- (3,\meshheight);
\draw (3,2*\meshheight) node {$x_{\numLayers}=L\sigma$};
\draw (5,-2*\meshheight) node {$x_{\numLayers+1}-x_{\numLayers}=L-2\sigma$};
\draw (7,-\meshheight) -- (7,\meshheight);
\draw (7,2*\meshheight) node {$x_{\numLayers+1}=L(1-\sigma)$};
\draw (9.1,-\meshheight) -- (9.1,\meshheight);
\draw (8.5,3.5*\meshheight) node {$\overbrace{x_{i}=L(1-\sigma^{{i-\numLayers}}) 
}^{i=n+1,\ldots,2n}$};
\draw (8.5,-3.5*\meshheight) node {$\underbrace{x_i-x_{i-1}=L\sigma^{i-\numLayers-1}(1-\sigma)}_{i=\numLayers+{2},\ldots,2\numLayers}$};
\draw (10,-\meshheight) -- (10,\meshheight);
\draw (10,2*\meshheight) node {$x_{2\numLayers+1}=L$};
\draw (10,-2*\meshheight) node {$x_{2\numLayers+1}-x_{2\numLayers}=L\sigma^{\numLayers}$};
\draw(0,0) -- (10,0);
\draw (0.27,-\meshheight) -- (0.27,\meshheight);
\draw (10-0.27,-\meshheight) -- (10-0.27,\meshheight);
\draw (0.081,-\meshheight) -- (0.081,\meshheight);
\draw (10-0.081,-\meshheight) -- (10-0.081,\meshheight);
\end{tikzpicture}
\caption{{The single-mesh space of Definition \ref{mesh_def} on a segment $[0,L]$.}}
\label{fig:mesh}
\end{figure}
\newpage
\begin{definition}\label{mesh_def}
Given ${L}>0$,
$\numLayers\in\N$ and a grading parameter $\sigma\in(0,1/2)$, we denote by $\calM _\numLayers(0,{L})=\{x_0,\ldots,x_{2n+1}\}$ the symmetric geometrically graded mesh on $[0,{L}]$ with $n$ layers in each direction, whose $2n+2$ meshpoints
$x_i$ are defined by
\begin{align*}
x_0:=&0,\\
x_i:=&{L}\sigma^{n-i+1},\quad&&\text{for }i=1,\ldots,\numLayers,\\
x_{i}:=&{L}(1-\sigma^{{i-\numLayers}}
),\quad&&\text{for }i=n+1,\ldots,2\numLayers,\\
x_{2n+1}:=&{L}.
\end{align*}
For a vector $\mathbf{p}=(p_1,\ldots,p_{\numLayers+1})\in(\N_0)^{\numLayers+1}$ we denote by $\calP _{\mathbf{p},\numLayers}$ the space of piecewise polynomials on $\calM _\numLayers(0,{L})$ with degree vector $\mathbf{p}$, i.e.
\begin{align*}
\calP _{\mathbf{p},n}(0,{L}):=\left\{
\begin{array}{l}\rho\in L^2(0,L):  
\rho|_{(x_{i-1},x_i)}\text{ and }\rho|_{(x_{2\numLayers+1-i},x_{2\numLayers-i+2})}\\
\quad\text{ are polynomials of degree at most $p_i$ for }i=1,\ldots,\numLayers+1\end{array}\right\}.
\end{align*}
\end{definition}

We first define two spaces for each side $\Gamma_j$, $j=1,\ldots,\numSides$, using $\numLayersj\in\N$ to determine the degree of mesh grading and the vectors $\mathbf{p}_j$ to determine the polynomial degree on each mesh element:
\begin{align*}
{V}^+_j:=\Big\{&v\in L^2(0,L_\Gamma): 
v|_{(\tL_{j-1},\tL_j)}(s)=\trho(s-\tL_{j-1})\e^{\imag ks},
\trho\in\calP _{\mathbf{p}_j,\numLayersj}(0,L_j),\\
&\qquad\rho|_{(0,{L}_{\Gamma})\setminus(\tL_{j-1},\tL_j)}=0
\Big\},
\\
{V}^-_j:=\Big\{&v\in L^2(0,L_\Gamma): 
v|_{(\tL_{j-1},\tL_j)}(s)=\trho(\tL_j-s)\e^{-\imag ks},
\trho\in\calP _{\mathbf{p}_j,\numLayersj}(0,L_j),\\
&\qquad\rho|_{(0,{L}_{\Gamma})\setminus(\tL_{j-1},\tL_j)}=0
\Big\}.
\end{align*}
As is explained in Remark \ref{why_remove}, to avoid ill-conditioning of the discrete system we must remove certain basis functions {supported on} {the elements within a given distance from the corners:}
\[
\tV_j:=\spank\left(\Big\{v\in  V^-_j : v|_{[\tL_{j-1},\tL_{j-1}+x_{\tn_j}]}=0\Big\}\cup\Big\{v\in  V^+_j : v|_{[\tL_{j}-x_{\tn_j},\tL_{j}]}=0\Big\}\right)
\]
where 
\begin{align*}
x_{\tn_j}:=
\max\Big\{x_i\in\calM _{\numLayersj}(0,{L}_j) \text{ such that }{x_{i}}\leq\alpha_{j}\frac{2\pi}{k}\Big\}
\end{align*}
and $\alpha_{j}$ is a parameter chosen such that $0<\alpha_{j}<L_j k/(4\pi)$, bounded independently of $k$ and $\bfp_j$, used to fine-tune the space. 
Put simply, there are two basis functions on {(large)} elements sufficiently far from the corners, and one basis element on {(small)} elements close to the corners. 
The parameter $\alpha_{j}$ determines what is meant by \emph{sufficiently far}. Hence the single-mesh approximation space with dimension $\DOFsHNA$ is defined as
\[
\VHNA:=\spank\bigcup_{j=1}^{\numSides}\tV_j.
\]

\begin{remark}[Why basis elements of the single-mesh space are removed]
\label{why_remove}
Since the mesh is strongly graded to approximate the singularities of $v^\pm_j$, some of its elements are much smaller than the wavelength of the problem, thus on these elements $\e^{\pm\imag ks}$ are roughly constant and the functions of ${V}^+_j$ supported on these elements are numerically indistinguishable from those on ${V}^-_j$, leading to an ill-conditioned discrete system of Galerkin equations set in ${V}_j^+\cup {V}_j^-$. 
To avoid this, in these elements we maintain only one of these two contributions. Intuitively, $\alpha_{j}$ can be thought of as the value such that
in all mesh elements with distance from one of the segment endpoints smaller than $\alpha_{j}$, the space $\tV_j$ supports polynomials multiplied with only one of the waves $\e^{\pm\imag ks}$.
As the parameter {$\alpha_{j}$} increases, fewer degrees of freedom are used and the conditioning of the discrete system is improved, but the accuracy of the method is reduced, hence care must be taken when selecting $\alpha_{j}$. 
\end{remark}

In much of what follows, the choice of single- or overlapping-mesh HNA space is irrelevant, hence we shall use $\VHNA$ to denote {either}, but will make clear the cases for which the choice is significant. For the overlapping-mesh space, best approximation estimates were derived in \citet[Theorem~5.4]{HeLaMe:13_}. 
The following result from \citet[Corollary~2.11]{Gi:17} compares the best approximation of the single-mesh and overlapping-mesh spaces, on $\Gamma$.

\begin{theorem}\label{co:bestApproxGamma}
Suppose that the obstacles $\resub{\Omega}$ and $\resub{\omega}$ are sufficiently far apart so that the separation condition \eqref{as:c_r_is_0} holds. 
{Let $\VHNA$ be an HNA space as above, $c_j>0$ be such that the polynomial degrees $p_j$ and the numbers of layers $n_j$ satisfy}
\begin{equation}\label{cjCondish}
n_j\geq c_jp_j,\quad\text{for }j=1,\ldots,\numSides,
\end{equation}
{and denote $p_\Gamma:=\min_j\{p_j\}$.}
Then we have the following best approximation estimate for the diffracted wave $v_\Gamma$ (of \eqref{diff_rep}):
\begin{equation*}
\inf_{w_{N_\Gamma}\in \VHNA}\|v_\Gamma-w_{N_\Gamma}\|_{L^2(\Gamma)}\leq C_\Gamma k^{-1/2}\MSM J(k)\e^{-p_\Gamma\tau_\Gamma},
\end{equation*}
where $C_\Gamma$ is a constant independent of $k$ and
\[
J(k):=
\begin{cases}
(1+kL_*)^{1/2-\delta_*}+\log^{1/2}(2+kL_*),& \VHNA \text{ overlapping-mesh,}\\
(1+kL_*)^{1/2-\delta_*}+\log^{1/2}(2+kL_*)+\sqrt k(kI_*)^{-\delta_*},& \VHNA \text{ single-mesh.}
\end{cases}
\]
with $I_*$ and $\tau_\Gamma$ independent of $n_j,p_j,k$ (both are defined precisely in \citet[Corollary~2.11]{Gi:17}), 
$\delta_*:=\min_{j,\pm}\{\delta^\pm_j\}$ (with $\delta_j^\pm$ as in Theorem~\ref{th:big_bounds}), whilst $L_*:=\max_j L_j$ the length of the longest side of $\resub{\Omega}$.
For the single-mesh space, it follows that $C_\Gamma=\max_j\{C_j\}$ for $C_j$ of \citet[Theorem~2.9]{Gi:17}. For the overlapping-mesh space, $C_\Gamma$ is equal to the constant $C_4$ of \citet[Theorem~5.5]{HeLaMe:13_}.
\end{theorem}

Theorem \ref{co:bestApproxGamma} shows that we obtain exponential convergence of the best approximation to $v_\pm$ {with respect to $p_\Gamma$, which controls both polynomial degree and mesh grading (via \eqref{cjCondish}),} across all wavenumbers $k$.
To maintain accuracy as $k$ increases
{one needs to increase $p_\Gamma$ in proportion to $\log k$, and hence the total number of degrees of freedom (which is proportional to $p_\Gamma^2$) in proportion to $\log^2 k$.}

\begin{remark}\label{deg_vec}
It is shown in \citet[Theorem~A.3]{HeLaMe:13} for the overlapping-mesh HNA space that it is possible to reduce the number of degrees of freedom on $\Gamma$, whilst maintaining exponential convergence, by reducing the polynomial degree in the smaller mesh elements{, as is standard in $hp$ schemes}.
For example, given a polynomial degree $p_j>1$, we {can} define for each side $\Gamma_j$, $j=1,\ldots,\numSides $, a degree vector $\bfp_j$ by
\[
(\bfp_j)_i:=\left\{
\begin{array}{cc}
p_j-\left\lfloor\frac{n_j+1-i}{n_j}p_j\right\rfloor,&\quad 1\leq i\leq n_j,\\
p_j,&\quad i= n_j+1,
\end{array}\right.
\]
where {$n_j$ is} as in Definition \ref{mesh_def} of the single-mesh space. This may be applied to either the single or overlapping mesh, and results in a linear reduction of polynomial degree on mesh elements closer to the corners of $\Gamma_j$. Numerical experiments in \S\ref{s:results} suggest that exponential convergence is maintained for the single-mesh HNA space if the degrees of freedom are reduced in this way, although we do not prove this here.
\end{remark}

\subsection{Standard {$hp$}-BEM approximation on {$\gamma$}}
\label{s:hpBEM}
If Assumption \ref{as:u_max_alg} holds, as is the case in the configurations of Theorem \ref{Mu_bdMS}, it follows from Theorem \ref{co:bestApproxGamma} that it is sufficient for the number of DOFs in $\VHNA$ to grow logarithmically with $k$, to accurately approximate $v_\pm$. However, this tells us nothing about the DOFs required on $\gamma$. 
To account for the 
contribution from $\gamma$, we parametrise $\bfx_\gamma:[0,L_\gamma]\rightarrow\gamma$ and construct an appropriate (depending on the geometry of $\resub{\omega}$) $\DOFshp$-dimensional approximation space $\Vhp\subset L^2(0,L_\gamma)$ for
\begin{equation}\label{lil_gamma_par}
v_\gamma(s):=\frac1k\pdiv{u}{\bfn}\big(\bfx_\gamma(s)\big),\quad s\in[0,L_\gamma].
\end{equation}
While a representation analogous to \eqref{dudn_rep} holds on $\gamma$ when $\resub{\omega}$ is a convex polygon, this approach is not suitable for the present multiple scattering approximation. 
If such a representation were used on multiple polygons, the system to solve would need to be written as a Neumann series and solved iteratively.
This alternative approach is outlined briefly in \citet[\S4.4.1]{Gi:17}. 
Instead we approximate the full solution $v_\gamma$, rather than any of its individual components as listed in \eqref{dudn_rep}. An advantage of the approach in this paper is that the only restriction imposed on $\gamma$ is that it must be Lipschitz and piecewise analytic. 
The disadvantage is that the number of DOFs required to approximate the solution on $\gamma$ has to increase with frequency to maintain accuracy, as is typical of standard $hp$-schemes. For all $k$ such that $|\gamma|$ is small compared with the wavelength $2\pi/k$, one would not expect this increase in DOFs to be significant.
Here, we take $\Vhp$ to be a standard $hp$-BEM approximation space consisting of piecewise polynomials to approximate $v_\gamma$, with mesh and degree vector dependent on the geometry of $\resub{\omega}$. 

	We now aim to bound the approximation of the solution on $\gamma$, in terms of key parameters, for the case where $\gamma$ is analytic. This will enable us to quantify the $k$-dependence of our method, which we expect to be mild when $|\gamma|$ is small compared with the wavelength. A range of tools were developed in \citet{LoMe:11} for $hp$-BEM approximations for problems of scattering by analytic surfaces, provided bounds on $\calA_{k,\eta}^{-1}$ are available. For this, we are able to use recently developed theory of $(R_0,R_1)$ configurations (of Definition \ref{def:R0R1}) for which we have from \citet[(1.28)]{ChSpGiSm:17}: if $\eta=O(k)$, then given $k_0>0$,
\begin{equation}\label{ineq:AinvBd}
\|\calA_{k,\eta}^{-1}\|_{L^2(\Gamma\cup\gamma)\shortrightarrow L^2(\Gamma\cup\gamma)} \lesssim k^2, \quad\text{for }k\geq k_0.
\end{equation}
In the class of problems we consider, the total boundary $\Gamma\cup\gamma$ is not analytic, because $\Gamma$ is the boundary of a polygon. Therefore we could not apply the theory of \citet{LoMe:11} to a standard $hp$ approximation on $\Gamma\cup\gamma$. However, in our method the standard $hp$ approximation is only on  $\gamma$, which in this section we will restrict to be analytic; Theorem \ref{co:bestApproxGamma} provides a best approximation estimate for the HNA space on the polygon $\Gamma$. As we shall see, this is sufficient to get a best approximation estimate for $v_\gamma$ in the standard $hp$ space $\Vhp$. The main idea is to consider an equivalent problem of scattering by (only) the obstacle $\resub{\omega}$, with the contribution from $\resub{\Omega}$ absorbed into the incident field. We can rewrite the representation \eqref{Green}
\[
u(\bfx) = u^i(\bfx) - \int_\Gamma\Phi_k(\bfx,\bfy)\pdiv{u}{\bfn}(\bfy)\dd{s}(\bfy)
- \int_\gamma\Phi_k(\bfx,\bfy)\pdiv{u}{\bfn}(\bfy)\dd{s}(\bfy),\quad\bfx\in D,
\]
separating the contribution from the convex polygon $\Gamma$. To construct an equivalent problem, we consider the additional component of the incident field to be the contribution from $\Gamma$:
\begin{equation}\label{eq:fieldDef}
u^i_\Gamma(\bfx) := - \int_\Gamma\Phi_k(\bfx,\bfy)\pdiv{u}{\bfn}(\bfy)\dd{s}(\bfy) = -\int_\Gamma \Phi_k(\bfx,\bfy)\calA_{k,\eta}^{-1}f_{k,\eta}(\bfy)\dd{s}(\bfy),\quad\bfx\in T_\gamma,
\end{equation}
where $T_\gamma$ is a tubular neighbourhood of $\gamma$, i.e. for some $\epsilon>0$ we have
\[
T_\gamma:=\{\bfx\in\R^2|\dist(\bfx,\gamma)<\epsilon\},
\]
with $\epsilon$ chosen such that $\dist(T_\gamma,\Gamma)>0$.  Our equivalent problem is therefore scattering of $u^i+u^i_\Gamma$ by $\resub{\omega}$, in $T_\gamma$. It is straightforward to see that the solution to this equivalent problem is the same as the solution to the BVP \eqref{Helmholtz}--\eqref{SRC} (restricted to $T_\gamma$). To use the $hp$ theory developed in \citet{LoMe:11}, we must show that the solution to our scattering problem is in the space of \citet[Definition~1.1]{LoMe:11}:
\begin{equation}\label{def:squigglyU}
\mathcal{U}(\psi,\xi,T_\gamma\setminus\gamma):=\{
\|\nabla^n\varphi\|_{L^2(T_\gamma\setminus\gamma)}\leq \xi^n\psi(k)\max\{n+1,|k|\}^n,\ \forall n\in\N_0
\}
\end{equation}
for some $\xi$ independent of $k$, $h$, $p$ and
\begin{equation}\label{def:nabuiG} 
|\nabla^nu(\bfx)|^2 := \sum_{\alpha\in\N_0^2:|\alpha| = n}\frac{n!}{\alpha!}|D^\alpha u(\bfx)|^2.
\end{equation}
A prerequisite for $u\in\mathcal{U}(\psi,\xi,T_\gamma\setminus\gamma)$ is that the incident field to our equivalent problem $u^i+u^i_\Gamma$ is also in $\mathcal{U}(\psi,\xi,T_\gamma\setminus\gamma)$, possibly for different parameters $\psi$ and $\xi$.

\begin{lemma}\label{lem:incGinU}
	If $\resub{\Omega}\cup\resub{\omega}$ is an $(R_0,R_1)$ configuration, then 
	\[
	u^i_\Gamma\in \mathcal{U}(\psi,1,T_\gamma)
	\]
	where $\psi(k):=C k^{7/2}\log^{1/2}(k\diam(\Gamma)+1)$ with
	$C>0$ a constant independent of $k$.
\end{lemma}
\begin{proof}
	Throughout the proof we let $C$ denote an arbitrary constant independent of $k$ and $n$. It follows from standard mapping properties of the single-layer operator (e.g., \citet[Theorem~2.15(i)]{ACTA}) that $u^i_\Gamma\in H^{1}(\omega)$, where $\omega$ is a bounded open subset of $\R^2$ containing $\resub{\Omega}\cup\resub{\omega}$. We may therefore bound using \citet[Theorem~B.6]{Me:12}, choosing zero forcing term to obtain
	\begin{equation}\label{eq:Hsp2}
	\|u^i_\Gamma\|_{H^{n+2}(T_\gamma)}\leq C k^{n+2}\|u^i_\Gamma\|_{L^2(\omega)},\quad\text{for }k\geq k_0, \quad n\in\N_0,
	\end{equation}
	given $k_0>0$, where $\omega$ is a bounded open set compactly containing $T_\gamma$ and $\resub{\Omega}$. From \eqref{def:nabuiG}, we see that the norm is the sum of $n+1$ terms, hence
	\begin{align}
	\|\nabla^n u^i_\Gamma\|_{L^2(T_\gamma)}^2 &\leq (n+1)!\|u^i_\Gamma\|_{H^{n+2}(T_\gamma)}^2,\\
	&\leq C (n+1)^nk^{n+2}\|u^i_\Gamma\|_{L^2(\omega)}^2,\quad\text{for }k\geq k_0, \quad n\in\N_0,\label{eq:nabnu}
	\end{align}
	given $k_0>0$, which follows by combining with \eqref{eq:Hsp2} and $(n+1)!\leq (n+1)^n$. We now bound $u^i_\Gamma$ in terms of known quantities,
	\begin{equation*}
		\|u^i_\Gamma\|_{L^2(\omega)}\leq |\omega|^{1/2}\|S_k\|_{L^2(\Gamma)\shortrightarrow L^\infty(\omega)}\|\calA_{k,\eta}^{-1}\|_{L^2(\Gamma)\shortrightarrow L^2(\Gamma)}\|f_{k,\eta}\|_{L^2(\Gamma)}.
	\end{equation*}
	We may bound these norms using {Lemma~\ref{bd:SLgen}}, \eqref{ineq:AinvBd} and {\eqref{def:f}} (choosing $\eta = O(k)$) to obtain
	\begin{equation}\label{uiG_bound}
	\|u^i_\Gamma\|_{L^2(\omega)}\leq Ck^{5/2}\log^{1/2}(k\diam(\Gamma)+1).
	\end{equation}
	Finally, we can combine the bound \eqref{uiG_bound} with \eqref{eq:nabnu} to obtain
	\[
	\|\nabla^n u^i_\Gamma\|_{L^2(T_\gamma)} \leq C k^{7/2}\log^{1/2}(k\diam(\Gamma)+1)  \max\{n+1,k\}^n,\quad\text{for }k\geq k_0,\quad n\in\N_0,
	\]
	proving the assertion.
\end{proof}

Now we have shown sufficient conditions on the growth of the derivatives of $u^i_\Gamma$, we are ready to obtain best approximation estimates on $\gamma$.
\newcommand{\MelenkSigma}{\zeta}
\begin{proposition}\label{co:hpBAE}
Suppose $\Upsilon=\resub{\Omega}\cup\resub{\omega} $ is an $(R_0,R_1)$ configuration (in the sense of Definition \ref{def:R0R1}) and $\resub{\omega}$ has an analytic boundary $\gamma$. If $\Vhp$ is constructed on a quasi-uniform mesh (in the sense of \citet[][\S1]{LoMe:11}) with $kh/p_\gamma\lesssim 1$, where $h$ and $p_\gamma$ denote maximum mesh width and polynomial degree respectively, then given positive constants $k_0$, $\MelenkSigma$ independent of $k, p_\gamma$ and $h$ we have the following best approximation estimate:
\[
\inf_{w_{\DOFshp}\in \Vhp}\|v_\gamma-w_{\DOFshp}\|_{L^2(\gamma)}\leq C_\gamma(k) \e^{-\tau_\gamma(k) p_\gamma},\quad\text{for }k\geq k_0,
\]
where 
\begin{equation}\label{taUgammAhp}
\tau_\gamma(k)=\log\left(\min\left\{\frac{\MelenkSigma p_\gamma}{kh},\frac{\MelenkSigma+h}{h}\right\}\right),\quad\text{and}\quad
C_\gamma(k):=C k^{8}\log({k\diam(\Gamma)+1}),
\end{equation}
with $C>0$ a constant independent of $k,p_\gamma$ and $h$.
\end{proposition}
\begin{proof}
By Lemma \ref{lem:incGinU} we have that $u^i_\Gamma\in \mathcal{U}(\psi,1,T_\gamma\setminus\gamma)$, and it is straightforward to see that $u^i\in \mathcal{U}(1,1,T_\gamma\setminus\gamma)$. Choosing
\[
g_1 := -\frac{\imag \eta}{k}(u^i + u^i_\Gamma),\quad g_2:=u^i + u^i_\Gamma,\quad\text{in }T_\gamma\cap D,
\]
with $g_1=g_2=0$ otherwise, we have that $g_1,g_2\in  \mathcal{U}(\psi,1,T_\gamma\setminus\gamma)$. Noting again \eqref{ineq:AinvBd}, we may appeal to \citet[Lemma~2.6]{LoMe:11} to deduce that the solution $u$ of the BVP \eqref{Helmholtz}--\eqref{SRC} (which is the same as the solution to the equivalent problem of scattering by $u^i+u^i_\Gamma$), is in the space $\mathcal{U}(\psi^*,1,T_\gamma\setminus\gamma)$, for all $k\geq k_0$ given $k_0>0$, where
\[
\psi^*(k):=\psi(k)(1+k^{5/2}\|\calA_{k,\eta}^{-1}\|_{L^2(\Gamma\cup\gamma)\shortrightarrow L^2(\Gamma\cup\gamma)} )\leq C k^8\log(k\diam(\Gamma)+1), \quad\text{for }k\geq k_0.
\]
Hence, the best approximation estimate of \citet[Lemma~3.16]{LoMe:11} may be applied to 
\[u/C_\gamma(k)\in\mathcal{U}(1,1,T_\gamma\setminus\gamma), \quad\text{for }k\geq k_0,\]
(noting \citet[Definition~3.3]{LoMe:11}), yielding the best approximation result after rescaling by $C_\gamma(k)$.

\end{proof}

{
	We do not expect the above result to be sharp, however to the best knowledge of the authors, it is the only $hp$-BEM estimate currently available for such a configuration.	We now generalise Proposition \ref{co:hpBAE} in the form of an assumption, which states that we observe exponential convergence to the solution $\partial u /\partial \bfn$ in $\Vhp$. It follows immediately from Proposition \ref{co:hpBAE} that this assumption holds for analytic $\gamma$, under appropriate conditions.} For the case of polygonal $\gamma$, the numerical experiments of \S\ref{s:results} suggest the assumption also holds, provided that we fix $\DOFshp=O(k)$.

\begin{assumption}\label{as:standard_BEM}
Denoting by $v_\gamma$ the restriction to $\gamma$ of the solution of the BIE \eqref{BIE}, we assume that the sequence of approximation spaces \[\left(\Vhp\right)_{\DOFshp \in\N}\] is such that
\[
\inf_{w_{\DOFshp}\in \Vhp}\|v_\gamma-w_{\DOFshp}\|_{L^2(\gamma)}\leq {C}_\gamma(k)\e^{-\tau_\gamma(k) p_\gamma},
\]
where the positive constants $C_\gamma(k),\tau_\gamma(k)$ may depend on $k$, and $p_\gamma$ is the polynomial degree of the space $\Vhp$.
\end{assumption}

\subsection{Combined approximation space on {$\Gamma\cup\gamma$}}\label{se:CombApprox}

The approximation space is based on the representation of the Neumann trace
\begin{equation}\label{ansatz}
\pdiv{u}{\bfn} =\left\{\begin{array}{cc}
\Psi+ kv_\Gamma+k\mathcal{G}_\gtG v_\gamma,&\text{ on }\Gamma,\\
kv_\gamma,&\text{ on }\gamma,
\end{array}\right.
\end{equation}
where $v_\Gamma$ and $v_\gamma$ are the unknowns that we solve for using the approximation spaces of \S\ref{s:HNAspace} and \S\ref{s:hpBEM}, whilst $\Psi$ denotes the Physical Optics Approximation \eqref{POA} and $\mathcal{G}_\gtG$ denotes the Interaction Operator \eqref{G_def}. 
Hence the approximation lies in the space 
\begin{equation}\label{eq:Vspace}
\V:=\VHNA\times\Vhp,
\end{equation}
where the total number of degrees of freedom is $\DOFsTotal=\DOFsHNA+\DOFshp$. 
For problems of one large polygon and one (or many) small polygon(s), the single-mesh HNA space $\VHNA$ is particularly practical, as only a small modification is required to implement both this and a standard $hp$-BEM space on $\Vhp$.

The following notation will be used to describe the problem in block operator form.

\begin{definition}[Operator restriction]\label{def:Txy_bits}
For {the}{ operator $\calA_{k,\eta}:L^2(\partial D)\to L^2(\partial D)$} (of Definition \ref{def:BIE_A}) and relatively open
$X,Y\subset \partial D$, we define the operator $\mathcal{A}_{Y\shortrightarrow X}:L^2(Y)\rightarrow L^2(X)$ by
\[
\mathcal{A}_{Y\shortrightarrow X}\varphi:=\left(\mathcal{A}_{k,\eta}\circ\mathcal{Q}_{Y}\varphi\right)|_{X},
\quad {\varphi\in L^2(Y),}
\]
{where $\mathcal{Q}_Y:L^2(Y)\to L^2(\partial D)$ is the zero-extension operator}, such that $(\mathcal{Q}_{Y}\varphi)|_Y=\varphi$ and $(\mathcal{Q}_{ Y}\varphi)|_{\partial D\setminus Y}=0$. 
For the case of the identity operator $I_{X\shortrightarrow X}:L^2(X)\rightarrow L^2(X)$ we simplify the notation by writing $I_X$.
\end{definition}

Inserting \eqref{ansatz} into the BIE \eqref{BIE}, we can write the problem to solve in block form: Find $v\in L^2(\Gamma)\times L^2(\gamma)$ such that
\begin{align}\label{cts_prob}
&\mathcal{A}_\square\mathcal{G}_\square v=
\left[\begin{array}{c}
f|_\Gamma-\mathcal{A}_\GtG\Psi\\
f|_\gamma-\mathcal{A}_\Gtg\Psi
\end{array}\right],\\
&\text{where}\quad
\mathcal{A}_\square:=\left[\begin{array}{c c}
\mathcal{A}_\GtG & \mathcal{A}_\gtG
\\ \mathcal{A}_\Gtg
& \mathcal{A}_\gtg
\end{array}\right]
\quad\text{and}\quad
\mathcal{G}_\square :=k\left[\begin{array}{c c}
\mathcal{I}_\Gamma & \mathcal{G}_\gtG\\ 0
& \mathcal{I}_\gamma
\end{array}\right].
\nonumber
\end{align}

\resub{Stated in a variational form equivalent to \eqref{eq:var_form}, our problem is as follows:
Find $v\in L^2(\Gamma\cup\gamma)$ such that
\begin{align}
&\Big(\mathcal{A}_\GtG\big[\left.v\right|_\Gamma\big],w{|_\Gamma}\Big)_{L^2(\Gamma)} 
+\Big( [\mathcal{A}_\gtG+\mathcal{A}_\GtG\mathcal{G}_\gtG]\big[\left.v\right|_\gamma\big]
,w{|_\Gamma}\left[\big.v\right|_\gamma\big]\Big)_{L^2(\Gamma)}
=
\frac{1}{k}\Big(f-\mathcal{A}_\GtG\Psi,w{|_\Gamma} \Big)_{L^2(\Gamma)},\label{eq:new_var1}\\ 
&\Big(\mathcal{A}_\Gtg\big[\left.v\right|_\Gamma\big],w{|_\gamma}\Big)_{L^2(\gamma)}  
+\Big( [\mathcal{A}_\gtg
+\mathcal{A}_\Gtg\mathcal{G}_\gtG]\big[\left.v\right|_\gamma\big],w{|_\gamma}\Big)_{L^2(\gamma)} =
\frac{1}{k}\Big(f-\mathcal{A}_\Gtg\Psi,w{|_\gamma} \Big)_{L^2(\gamma)},\label{eq:new_var2}
\end{align}
for all $w\in L^2(\Gamma\cup\gamma)$. This problem forms the basis of our Galerkin method.}

\section{Galerkin method}
\label{s:Galerkin_method}
In this section, we derive error bounds for the approximation of equation \eqref{cts_prob} by the Galerkin method on the discrete space $\V$ (defined in \eqref{eq:Vspace}).
Under certain assumptions, we will show that exponential convergence is achieved. 
We intend to approximate the unknown components of the solution on $\Gamma$ and $\gamma$, that is
\begin{equation*}
v_\DOFsTotal:=\left(\begin{array}{c}
\bigSol\\
\smlSol
\end{array}\right)
\approx
\left(\begin{array}{c}
 v_\Gamma\\
v_\gamma
\end{array}\right)
=:v,
\end{equation*}
where $v$ is the solution to \eqref{cts_prob}.
Recall (from \S\ref{se:CombApprox}) that we use an HNA approximation space $\VHNA$ (single- or overlapping-mesh) on $\Gamma$, with a standard $hp$-approximation space $\Vhp$ on $\gamma$. The discrete problem to solve is: find $v_\DOFsTotal\in \V$ such that
\begin{align}
&\Big(\mathcal{A}_\GtG\bigSol ,\resub{w^N_\Gamma}\Big)_{L^2(\Gamma)} 
+\Big( [\mathcal{A}_\gtG+\mathcal{A}_\GtG\mathcal{G}_\gtG]\smlSol 
,\resub{w^N_\Gamma}\Big)_{L^2(\Gamma)}
=
\frac{1}{k}\Big(f-\mathcal{A}_\GtG\Psi,\resub{w^N_\Gamma}\Big)_{L^2(\Gamma)},
\label{Galerkin_eqns1}\\ 
&\Big(\mathcal{A}_\Gtg\bigSol,\resub{w^N_\gamma}\Big)_{L^2(\gamma)}  
+\Big( [\mathcal{A}_\gtg
+\mathcal{A}_\Gtg\mathcal{G}_\gtG]\smlSol,\resub{w^N_\gamma}\Big)_{L^2(\gamma)} 
=
\frac{1}{k}\Big(f-\mathcal{A}_\Gtg\Psi,\resub{w^N_\gamma} \Big)_{L^2(\gamma)},\label{Galerkin_eqns2}
\end{align}
for all $\resub{({w^N_\Gamma},{w^N_\gamma})}\in\V$. To implement the Galerkin method, we choose suitable bases $\Lambda_\Gamma$ and $\Lambda_\gamma$, with 
\[
\spank\Lambda_\Gamma=\VHNA \quad\text{and} \quad\spank\Lambda_\gamma=\Vhp.
\]
To determine $v_\DOFsTotal $ we seek ${\mathbf{a}}\in\C^{\DOFsTotal }$ which solves the block matrix system $B\mathbf{a}=\mathbf{b}$, where 
\begin{equation}\label{B_blocks}
B:=\overset{\varphi\in \Lambda_\Gamma\hspace{5cm}\varphi\in \Lambda_\gamma}{\left[\begin{array}{c|c}
{\left(\mathcal{A}_\GtG\varphi ,\phi\right)_{L^2(\Gamma)}} &\left( [\mathcal{A}_\gtG+\mathcal{A}_\GtG\mathcal{G}_{\gamma\rightarrow\Gamma}]\varphi ,\phi\right)_{L^2(\Gamma)}
\\ \hline
\left(\mathcal{A}_\Gtg \varphi,\phi\right)_{L^2(\gamma)}  & \left( [\mathcal{A}_\gtg+\mathcal{A}_\Gtg\mathcal{G}_\gtG]\varphi,\phi\right)_{L^2(\gamma)} 
\end{array}\right]}
\begin{array}{c}
{ }^{\phi\in\Lambda_\Gamma}\\
{ }_{\phi\in\Lambda_\gamma}\end{array}\
\end{equation}
and
\begin{equation}\label{F_blocks}
\mathbf{b}:=\frac{1}{k}\left[
\begin{array}{c}
\left(f-\mathcal{A}_\GtG\Psi,\phi \right)_{L^2(\Gamma)}
\\ \hline
\left(f-\mathcal{A}_\Gtg\Psi,\phi \right)_{L^2(\gamma)}
\end{array}
\right]
\begin{array}{c}
{ }^{\phi\in\Lambda_\Gamma}\\
{ }_{\phi\in\Lambda_\gamma}
\end{array}.
\end{equation}
For further details on implementation, see Remark \ref{re:quadrature}.

For the remainder of the section, we present approximation estimates of quantities of practical interest. 
We assume that as $\DOFsTotal$ increases, so do $\DOFsHNA$ and $\DOFshp$, such that the following convergence conditions hold:
	\begin{align}
\lim_{\DOFsTotal\shortrightarrow\infty}\quad\inf_{w_\DOFsTotal\in \V} \|w-w_\Gamma^{\DOFsTotal}\|_{L^2(0,L_\Gamma)}=0&\quad\text{for all }w\in C^\infty(0,L_\Gamma),\label{eq:Ncondish1}\\
\lim_{\DOFsTotal\shortrightarrow\infty}\quad\inf_{w_N\in \V} \|w-w_{\gamma_i}^{\DOFsTotal}\|_{L^2(0,L_{\gamma_i})}=0&\quad\text{for all }w\in C^\infty(0,L_{\gamma_i}), \label{eq:Ncondish2}
\end{align}
for $i=1,\ldots\numScats$, where $w_\DOFsTotal:=(w_\Gamma^{\DOFsTotal},w_{\gamma_1}^{\DOFsTotal},\ldots,w_{\gamma_\numScats}^{\DOFsTotal})$. If $\Vhp$ is a standard $hp$-BEM space then \eqref{eq:Ncondish2} holds. It follows by identical arguments to \citet[Theorem~5.1]{ChLa:07} and the definition of the single- and overlapping-mesh spaces (see \S\ref{s:approx_space}) that Condition \eqref{eq:Ncondish1} holds.
We present a lemma concerning the stability of the system \eqref{Galerkin_eqns1}--\eqref{Galerkin_eqns2}.

\begin{lemma}[Stability of discrete system]
\label{uniquenessProj}
Suppose the convergence conditions \eqref{eq:Ncondish1}--\eqref{eq:Ncondish2}  hold. 
Then there exist positive constants $C_q(k)$ and $\Nought(k)$ such that {for $\DOFsTotal\ge \Nought$} the solution $v_\DOFsTotal$ of \eqref{Galerkin_eqns1}--\eqref{Galerkin_eqns2} exists.
Moreover
\[
\|v-v_\DOFsTotal\|_{L^2(\partial D)}
\leq C_q(k)
\min_{w_\DOFsTotal\in \V}\|v-w_\DOFsTotal \|_{L^2(\partial D)},\quad\text{for }\DOFsTotal\geq \Nought(k).
\]
\end{lemma}
\begin{proof}
First we show that $\mathcal{A}_\square $ is a compact perturbation of an operator which is {Fredholm of zero index.} 
We have from \citet[p.~620]{ChLa:07} that $\mathcal{A}_\GtG$ is a compact perturbation of a {Fredholm} operator (of index zero), and the same arguments can be applied to each $\mathcal{A}_{\gamma_i\shortrightarrow\gamma_i}$ for $i=1,\ldots \numScats$. As the kernels of $\mathcal{A}_{\Gtg_i}$, $\mathcal{A}_{\gamma_i\shortrightarrow\Gamma}$ and $\calA_{\gamma_i\shortrightarrow\gamma_\ell}$ for $i\neq\ell$ are continuous for $i=1,\ldots,\numScats$, these operators are also compact, hence $\mathcal{A}_{k,\eta}$ is a compact perturbation of a coercive {Fredholm of zero index} operator.

Let $\mathcal{P}_\DOFsTotal$ be {the} orthogonal projection operator from $L^2(\Gamma)\times L^2(\gamma)$ onto $\V$. 
Given the convergence condition \eqref{eq:Ncondish1}, it follows by the density of $C^\infty(0,L_\Gamma)$ in $L^2(0,L_\Gamma)$ for $j=1,\ldots,\numSides$ that we have convergence of the best approximation to any $L^2(0,L_\Gamma)$ function in $\VHNA$. Similar arguments follow for convergence on $\gamma$, by the convergence condition \eqref{eq:Ncondish2}.
Then \citet[Theorem~5.2]{ChLa:07} shows the existence of a solution to the discrete problem \eqref{Galerkin_eqns1}--\eqref{Galerkin_eqns2}{, for $N$ sufficiently large,} via a bound on
\begin{equation}\label{Cq_def}
\|(\mathcal{I}+\mathcal{P}_\DOFsTotal K)^{-1}\|_{L^2(\partial D)\shortrightarrow L^2(\partial D)}=:C_q<\infty,
\end{equation}
where 
\[
K:=\mathcal{A}_\square\mathcal{G}_\square-\mathcal{I}
\quad\text{with}\quad
\mathcal{I}:=\left[\begin{array}{c c}
\mathcal{I}_\Gamma & 0
\\ 0
& \mathcal{I}_\gamma
\end{array}\right].
\]
To show that our method converges to the true solution, we proceed as in \citet[Theorem~5.3]{ChLa:07}, noting that
\[\mathcal{P}_\DOFsTotal (\mathcal{I}+K)v=
\mathcal{P}_\DOFsTotal \left[\begin{array}{c}
f|_\Gamma-\mathcal{A}_\GtG\Psi\\
f|_\gamma-\mathcal{A}_\Gtg\Psi
\end{array}\right],\]
which we combine with \eqref{cts_prob} to obtain
\[
v_\DOFsTotal +\mathcal{P}_\DOFsTotal Kv_\DOFsTotal =\mathcal{P}_\DOFsTotal (\mathcal{I}+K)v.
\]
Rearranging and adding $v$ to both sides yields
\[
(\mathcal{I}+\mathcal{P}_\DOFsTotal K)(v-v_\DOFsTotal )=(\mathcal{I}-\mathcal{P}_\DOFsTotal )v,
\]
hence we can bound
\begin{align*}
\left\|v-v_\DOFsTotal \right\|_{L^2(\partial D)\shortrightarrow L^2(\partial D)}\leq&\left\|(\mathcal{I}+\mathcal{P}_\DOFsTotal K)^{-1}\right\|_{L^2(\partial D)\shortrightarrow L^2(\partial D)} \left\|v-\mathcal{P}_\DOFsTotal v\right\|_{L^2(\partial D)\shortrightarrow L^2(\partial D)}
\end{align*}
and the bound follows from the definition of $\mathcal{P}_\DOFsTotal $ and \eqref{Cq_def}.
\end{proof}

For our operator $\calA_{k,\eta}$, there is little that can be said about the constants $C_q(k)$ and $N_0(k)$ for the scattering configurations considered in this paper. In the appendix we introduce an alternative BIE formulation which is coercive provided that $|\gamma|$ is of the order of one wavelength. For this coercive formulation, $N_0(k)=1$ and $C_q(k)$ can be made explicit.

Recalling that we are actually approximating the (dimensionless) diffracted waves on $\Gamma$ and the (dimensionless) Neumann trace of the solution on $\gamma$, the full approximation to the Neumann trace follows by inserting $v_\DOFsTotal $ into \eqref{ansatz} and is denoted
\begin{equation}\label{anz_approx}
\nu_\DOFsTotal :=
\left\{\begin{array}{cc}
\Psi+k \bigSol+k\mathcal{G}_\gtG \smlSol,&\quad\text{on }\Gamma,\\
k\smlSol,&\quad\text{on }\gamma.
\end{array}\right.
\end{equation}
The following theorem can be used to determine the error of the full approximation.

\begin{theorem}\label{th:anz_approx_bd}
Suppose that
\begin{enumerate}[(i)]
\item the separation condition \eqref{as:c_r_is_0} holds,
\item the convergence conditions \eqref{eq:Ncondish1}--\eqref{eq:Ncondish2} hold,
\item Assumption \ref{as:standard_BEM} (exponential convergence of $\Vhp$) holds,
\item Assumption \ref{as:u_max_alg} (algebraic growth of the solution of the BVP \eqref{Helmholtz}--\eqref{SRC}) holds.
\end{enumerate}
Then we have the following bound on the error of the approximation \eqref{anz_approx} to the solution $\pdivl{u}{\bfn}$:
\begin{equation*}
\left\|\pdiv{u}{\bfn} -\nu_\DOFsTotal \right\|_{L^2(\partial D)}\leq C_q(k)k\BDdudn,
\end{equation*}
for $\DOFsTotal\geq\Nought$, where
\begin{enumerate}[(i)]
\item $\Nought$ and $C_q$ are as in Lemma \ref{uniquenessProj},
\item $C_\calG$ as in Lemma \ref{le:G_bound},
\item $C_u$ and $\beta$ are the constants from Assumption \ref{as:u_max_alg},
\item $p_\Gamma$, $J(k)$, $C_\Gamma$ and $\tau_\Gamma$ are as in Theorem \ref{co:bestApproxGamma},
\item $p_\gamma$, $C_\gamma$ and $\tau_\gamma$ are as in Assumption \ref{as:standard_BEM}.
\end{enumerate}
\end{theorem}

\begin{proof}
First we focus on the best approximation of $\pdivl{u}{\bfn}$ by an element $w=(w_\Gamma,w_\gamma)$ of $\V$. 
By the definition \eqref{ansatz} we have
\begin{align*}\nonumber
&\inf_{w\in\V}{\bigg(
\bigg\| \pndiv{u}-\big(\Psi+k w|_\Gamma+k\calG_\gtG w|_\gamma\big) \bigg\|_{L^2(\Gamma)}
+\bigg\| \pndiv{u}- kw|_\gamma \bigg\|_{L^2(\gamma)}
\bigg)}\\
&=
k\inf_{w\in\V}\left( 
\left\|[v_\Gamma-w|_\Gamma] + \calG_\gtG[v_\gamma-w|_\gamma]\right\|^2_{L^2(\Gamma)} 
+ \left\|v_\gamma-w|_\gamma\right\|^2_{L^2(\gamma)}\right)^{1/2}\nonumber\\
&\leq
k\inf_{w\in\V}\left( \left\|v_\Gamma-w|_\Gamma\right\|_{L^2(\Gamma)}  
+\left[1+\left\|\calG_\gtG\right\|_{L^2(\gamma)\shortrightarrow L^2(\Gamma)}\right]
\left\|v_\gamma-w|_\gamma\right\|_{L^2(\gamma)}\right).
\end{align*}
Applying Lemma \ref{uniquenessProj} {and recalling the definition \eqref{eq:Vspace} of $\V$}, we can write
\begin{align*}
&\left\|\pdivl{u}{\bfn}-\nu_\DOFsTotal\right\|_{L^2(\partial D)}\nonumber\leq C_q(k)\\
& \times k\left(\inf_{w_\Gamma\in\VHNA}\|v_\Gamma-w_\Gamma\|_{L^2(\Gamma)}  +\inf_{w_\gamma\in\Vhp}\left[1+\|\calG_\gtG\|_{L^2(\gamma)\shortrightarrow L^2(\Gamma)}\right]\|v_\gamma-w_\gamma\|_{L^2(\gamma)}\right).
\end{align*}
The assertion follows by combining this inequality with Lemma \ref{le:G_bound}, Assumption \ref{as:u_max_alg}, Theorem \ref{co:bestApproxGamma} and Assumption \ref{as:standard_BEM}.
\end{proof}

For a fixed frequency, Theorem \ref{th:anz_approx_bd} suggests that the proposed method is well suited to problems for which $\resub{\Omega}$ is a convex polygon, and $\resub{\omega}$ has a size parameter much smaller than $\resub{\Omega}$. This is because the number of DOFs required to maintain accuracy in the approximation space on $\Gamma$ grows only logarithmically with $k$.  The method will hence be particularly effective if $\resub{\omega}$ has a size parameter of the order of one wavelength, since in this case the oscillations on $\gamma$ are resolved whilst $\DOFsHNA$ does not need to be large to account for high frequencies due to the (almost) frequency independence of the approximation on~$\Gamma$.

{
\begin{remark}[Dependencies of parameters of Theorem \ref{th:anz_approx_bd}]\label{re:params}
In the following situations the bounding constants of Theorem \ref{th:anz_approx_bd} can be made either fully explicit, or $k$-explicit. 
\begin{enumerate}[(i)] 
\item The terms $C_\Gamma$, $\tau_\Gamma$ and $J(k)$ are fully explicit given $k$, the geometry of $\resub{\Omega}$ and the parameters of $\VHNA$. This follows from the separation condition \eqref{as:c_r_is_0}.
\item In the appendix we present an alternative boundary integral equation which is coercive, under certain geometric restrictions. In such a case $C_q(k)$ is known and $N_0(k)=1$.
 \item By Theorem \ref{Mu_bdMS}, if $\resub{\Omega}\cup\resub{\omega}$ is a non-trapping polygon (in the sense of Definition~\ref{def:nontrapping}), then we can choose $\beta=1/2+\epsilon$ for any $\epsilon>0$.
 \item If $\Upsilon$ is an $(R_0,R_1)$ configuration, then by Theorem \ref{Mu_bdMS} we obtain $\beta=5/2+\epsilon$ for $\epsilon>0$. Furthermore, if $\gamma$ is also analytic and $\Vhp$ satisfies the conditions of Proposition \ref{co:hpBAE} we have {$C_\gamma(k)=C k^{8}\log({k\diam(\Gamma)+1})$}, and $\tau_\gamma$ is given by \eqref{taUgammAhp}.
 \end{enumerate}
 \end{remark}
 }

An approximation $u_\DOFsTotal $ to the solution $u$ of the BVP \eqref{Helmholtz}--\eqref{BC} in $D$ is obtained by combining $\nu_\DOFsTotal $ with the representation formula \eqref{Green},
\begin{align}
u_\DOFsTotal (\bfx):=\uinc (\bfx)&-\int_0^{L_\Gamma}\Phi\left(\bfx,\bfy_\Gamma(s)\right)
\left(\Psi\left(\bfy_\Gamma(s)\right)+k\bigSol(s)+k[\mathcal{G}_\gtG\smlSol](s)\right)\dd{s}
\nonumber\\
&-k\int_0^{L_\gamma}\Phi\left(\bfx,\bfy_\gamma(s)\right)\smlSol(s)\dd{s},
\hspace{10mm}\text{for }\bfx\in D.
\label{u_M}
\end{align}
Here the parametrisation $\bfy_\Gamma$ is {as in \eqref{Gamma_param} and  $\bfy_\gamma$ as in \S\ref{s:hpBEM}.}
Expanding further, we can extend the definition of $\mathcal{G}_\gtG$ to a parametrised form by
\begin{equation*}
{\big(}\mathcal{G}_\gtG\smlSol{\big)}(s):=
\int_0^{L_\gamma}{\chi_\gamma(s,t)}
\pdiv{\Phi_k(\bfy_\Gamma(s),\bfy_\gamma(t))}{\bfn(\bfy_\Gamma(s))}\smlSol(t)\dd{t},
\quad{s\in[0,L_\Gamma],}
\end{equation*}
where the indicator function
\[
\chi_\gamma(s,t):=
\left\{\begin{array}{cc}
1,& \bfy_\Gamma(s)\in\Gamma_j \text{ and }\bfy_\gamma(t)\in U_j,\\
0,& \text{otherwise},
\end{array}\right.
\]
is used to ensure the path of integration remains inside the relative upper half-plane $U_j$.

\begin{corollary}\label{cor:domainErr}
{Assume conditions \emph{(i)--(iv)} of Theorem \ref{th:anz_approx_bd} hold.
Then given $k_0~>~0$, the HNA-BEM approximation to the BVP \eqref{Helmholtz}--\eqref{SRC} satisfies the error bound
\begin{align*}
\|u-u_\DOFsTotal \|_{L^\infty(D)}
\lesssim&
C_q(k)k^{1/2}\log^{-1/2}(1+k\diam(\partial D))\\
&\times\BDdudn,
\end{align*}
for $N\ge N_0$ and $k\geq k_0$. The terms in the bound are as in Theorem \ref{th:anz_approx_bd}.

}\end{corollary}
\begin{proof}
The result follows from the representation \eqref{def:Sk}, the bounds on $\|S_k\|_{L^2(\partial D)\shortrightarrow L^\infty(D)}$ given in {Lemma~\ref{lem:S_k_bd}},
Theorem \ref{th:anz_approx_bd}, and
\[
{\|u-u_\DOFsTotal \|_{L^\infty(D)}}=\left\|S_k\left(\pdiv{u}{\bfn} -\nu_\DOFsTotal \right)\right\|_{L^\infty(D)}\leq\|S_k\|_{L^2(\partial D)\shortrightarrow L^\infty(D)}\left\|\pdiv{u}{\bfn} -\nu_\DOFsTotal \right\|_{L^2(\partial D)}.
\]
\end{proof}

A quantity of practical interest is the \emph{far-field pattern} of the scattered field $u^s$, which describes the distribution of energy of the scattered field $u^s$ (of a solution to the BVP \eqref{Helmholtz}--\eqref{SRC}) far away from $\resub{\Omega}\cup\resub{\omega}$. 
We can represent the asymptotic behaviour of the scattered field (as in \citet[\S6]{HeLaMe:13_}) by
\[
u^s(\bfx)\sim u^\infty(\theta)\frac{\e^{\imag(kr +\pi/4)}}{2\sqrt{2\pi k r}},\quad\text{for }\bfx = r(\cos\theta,\sin\theta),\quad\text{as }r\rightarrow\infty,
\]
where the term $u^\infty(\theta)$ denotes the \emph{far-field pattern} at observation angle $\theta\in[0,2\pi)$, which we can represent via the solution to the BIE \eqref{BIE}:
\begin{equation}\label{def:FarField}
u^\infty(\theta):=-\int_{\partial D}\e^{-\imag k[y_1\cos\theta + y_2\sin\theta ]}\pdiv{u}{\bfn}(\bfy)\dd{s}(\bfy),\quad\theta\in[0,2\pi), \quad\bfy=(y_1,y_2).
\end{equation} 
We may define an approximation $u^\infty_{\DOFsTotal}$ to the far-field pattern $u^\infty$ by inserting $\nu_\DOFsTotal$ into \eqref{def:FarField} {in place of $\pdivl{u}\bfn$}.

\newpage
\begin{corollary}\label{cor:FF_conv}
{Under the assumption of Theorem \ref{th:anz_approx_bd}, the far-field pattern $u^\infty_{\DOFsTotal}$ computed from the HNA-BEM solution approximates $u^\infty$ with the error bound
\begin{align*}
&\|u^\infty-u^\infty_{\DOFsTotal} \|_{L^\infty(0,2\pi)}\\
&\leq  C_q(k) k\sqrt{L_\Gamma+L_\gamma}\BDdudn.
\end{align*}
}
{The terms in the bound are as in Theorem \ref{th:anz_approx_bd}.}
\end{corollary}
{
 \begin{proof}
We have
 \begin{equation*}
 |u^\infty(\theta)-u^\infty_{\DOFsTotal} (\theta)|
 \leq\int_{\partial D}\left|\pdiv{u}{\bfn} -\nu_\DOFsTotal \right|\dd{s}
 \leq(L_\Gamma+L_\gamma)^{1/2}\left\|\pdiv{u}{\bfn}  -\nu_\DOFsTotal \right\|_{L^2(\partial D)}
 \end{equation*}
 and the result follows by Theorem \ref{th:anz_approx_bd}.
 \end{proof}
}

\section{Numerical results}\label{s:results}

Here we present numerical results for the solution of the discrete problem \eqref{Galerkin_eqns1}--\eqref{Galerkin_eqns2}.
\resub{Experiments were run over a range wavenumbers $k\in\{20,40,80,160\}$, incident angles $\bfd$ and maximal polynomial degrees $p\in\{1,\ldots,8\}$, for three scattering configurations, which we shall refer to as Experiments One, Two and Three. Each configuration consists of an equilateral triangle $\Omega$ with perimeter $L_\Gamma=6\pi$, and some small scatterer(s) $\omega$. In Experiment One (\S\ref{sec:exp1}), $\omega$ consists of a single small triangular scatterer with perimeter $L_\gamma=3\pi/5$, with the obstacles separated by a fixed distance of $\dist(\Gamma,\gamma)=\sqrt3\pi/5$, as in Figure \ref{fig:exp1}\protect\subref{fig:exp1D}. In Experiment Two (\S\ref{sec:exp2}) we reduce the distance between the obstacles in proportion to the problem wavelength. In Experiment Three (\S\ref{sec:exp3}), $\omega$ consists of two disjoint triangular scatterers.}

In terms of observed error, each value of $\bfd$ tested gave very similar results, hence we focus here on the case $\bfd=(1,1)/\sqrt{2}$, which
allows some re-reflections between the obstacles and partial illumination of $\Gamma$, see Figures \ref{fig:exp1}\protect\subref{fig:exp1D}, \ref{fig:exp2}\subref{fig:exp3D} and \ref{fig:exp3}\subref{fig:exp3D}.

\resub{We now describe the approximation parameters common to all three experiments.} To construct the approximation space {$\V$}, we first choose $\VHNA$ to be the single-mesh approximation space of \S\ref{s:approx_space} with $p_j=p$ for each side $j=1,\ldots,\numSides =3$, reducing the polynomial degree close to the corners of $\Gamma$ in accordance with Remark~\ref{deg_vec}, hence $p$ now refers to the polynomial degree on the largest mesh elements. We also remove basis elements close to the corners of the mesh on $\Gamma$ in accordance with Remark \ref{why_remove}, choosing $\alpha_j=\max\{(1+p)/4,2\}$, to improve conditioning of the discrete system \eqref{B_blocks}. A grading parameter of $\sigma=0.15$ is used (as in \citet{HeLaMe:13}, where the rationale for this choice is discussed), with $n_j=2p$ layers on each graded mesh, for $j=1,{2,3}$ (hence we may choose the constant from Theorem \ref{co:bestApproxGamma} as $c_j=2$ ).

Theorem \ref{co:bestApproxGamma} ensures that we will observe exponential convergence on $\Gamma$ if the polynomial degree is consistent across the mesh, and Proposition \ref{co:hpBAE} ensures that we observe exponential convergence on $\gamma$\resub{,} if $\gamma$ is analytic. In these numerical experiments we test problems where these two conditions are not met, and encouragingly still observe exponential convergence. As hypothesised by Remark \ref{deg_vec} and Assumption \ref{as:standard_BEM}, our experiments suggest that our method converges exponentially under conditions much broader than those guaranteed by our theory.

For the standard $hp$-BEM space $\Vhp$, we use the same parameters \resub{$p_\gamma=p$}, $\sigma$ and $c_j$ to grade towards the corners of $\gamma$, so the construction of the mesh on $\gamma$ is much the same as on $\Gamma$.
The key difference is that on $\gamma$ every mesh element is sufficiently subdivided to resolve the oscillations. The polynomial degree $p_j$ is decreased on smaller elements, as on $\Gamma$, in accordance with Remark \ref{why_remove}.

\resub{Figures \ref{fig:exp1}\protect\subref{fig:exp1L2}, \ref{fig:exp2}\protect\subref{fig:exp2L2} and \ref{fig:exp3}\protect\subref{fig:exp3L2} show $L^2$ convergence on the boundary $\partial D=\Gamma\cup\gamma$, as $p$ increases, for different values of~$k$.
The markers
correspond to the increasing polynomial degree $p=1,\ldots,7$ and the horizontal axis represents the total number of DOFs $\DOFsTotal$, which depends on both $p$ and $k$.
The reference solution, denoted $\nu_{\DOFsTotal^*}$, is computed with $p=8$.
	Additional checks were performed against a high order standard BEM approximation to validate the reference solution.
	In each experiment that follows, the increased number of oscillations appears to be handled by the increase in $\DOFshp $ for each $k$ (here $\DOFsHNA $ remains roughly fixed as $k$ increases, and $\DOFshp$ increases less than linearly with $k$) with exponential convergence in $p$ observed in each case, as predicted by Theorem~\ref{th:anz_approx_bd} (for analytic~$\gamma$). Given exponential convergence in $L^2(\partial D$), corollaries \ref{cor:domainErr} and \ref{cor:FF_conv} are sufficient to guarantee exponential point-wise convergence of the domain approximation \eqref{u_M} and the far-field approximation \eqref{def:FarField}.}
	
\begin{remark}[Quadrature]\label{re:quadrature}
The integrals in \eqref{B_blocks} and \eqref{F_blocks} and the $L^2$ norms used to estimate the error in \resub{Figures \ref{fig:exp1}\protect\subref{fig:exp1L2}, \ref{fig:exp2}\protect\subref{fig:exp2L2} and \ref{fig:exp3}\protect\subref{fig:exp3L2}} may be oscillatory and singular.
In particular, care must be taken when evaluating the triple integral $\left(\mathcal{A}_\Gtg\mathcal{G}_\gtG v,w \right)_{L^2(\gamma)}$, which contains a singular oscillatory integrand on elements for which $\calG_{\gamma\shortrightarrow\Gamma}v$ is supported. 
Standard composite quadrature routines require a large number of weights and nodes. 
Hence, at higher frequencies, oscillatory quadrature rules should be used (see \citet{DeHuIs:18} for a review of such methods), while singular integrals should be computed using a suitable quadrature rule (e.g.\ \citet{HuCo:09}).
\end{remark}

\resub{\subsection{Experiment one}\label{sec:exp1}}

\begin{figure}
	\centering
	\subfloat[][]{\psfrag{[x]}[t]{$s$}
		\psfrag{[x]}[c]{Total number of DOFs $\DOFsTotal$, for $p=1,\ldots,7$}\psfrag{[t]}[Bc]{}\psfrag{[y]}[Bc]{\vspace{-20mm}$\frac{\| \nu_{\DOFsTotal^*}-\nu_{\DOFsTotal}\|_{L^2({\partial D})}}{\|\nu_{\DOFsTotal^*}\|_{L^2({\partial D})}}$}{\includegraphics[width=.475\linewidth,trim={2cm 0cm 2cm 0cm},clip]{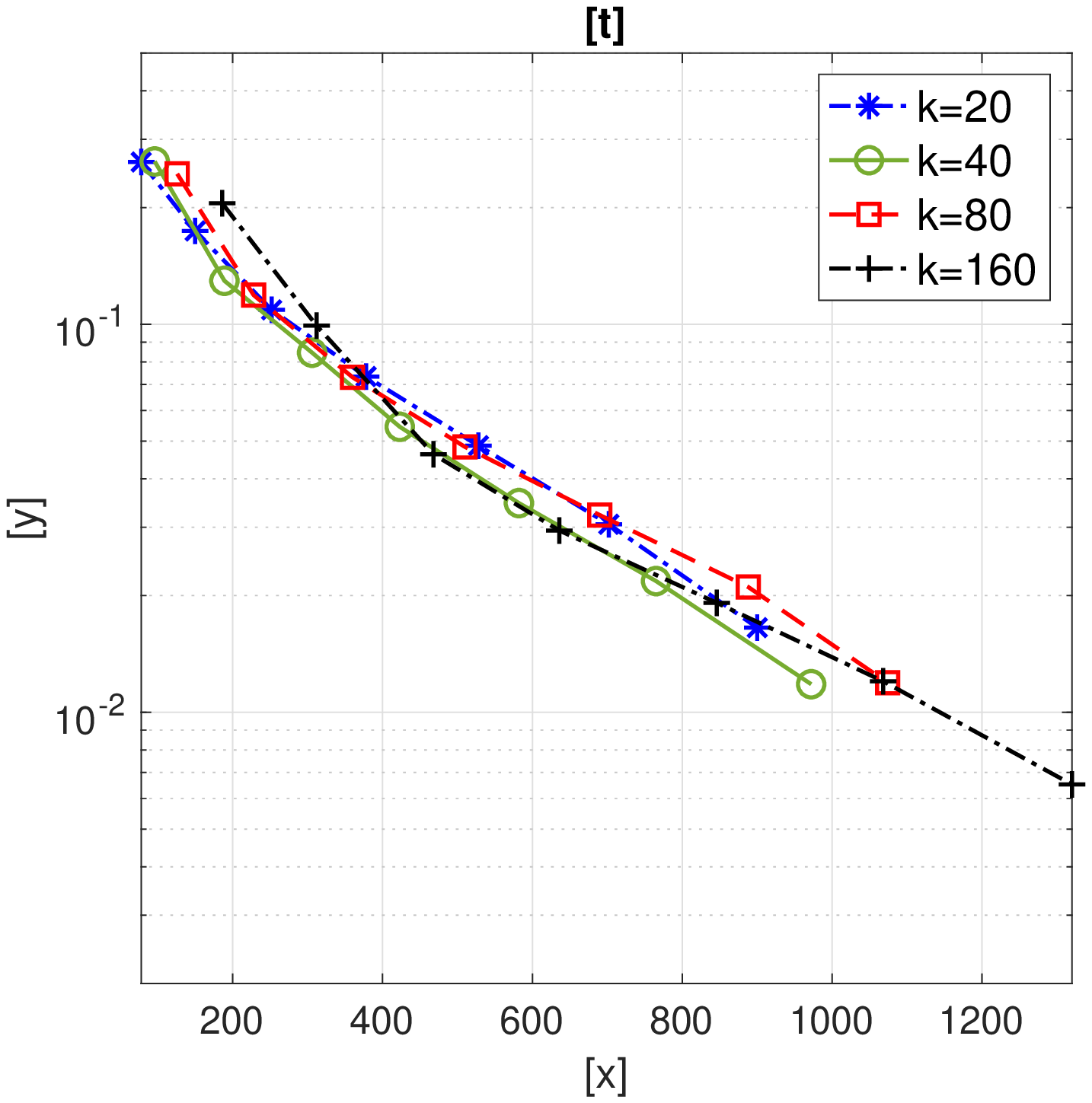}}\label{fig:exp1L2}}
	\subfloat[][]{\includegraphics[width=.525\linewidth, trim={2cm 1.2cm 1cm 1cm},clip]{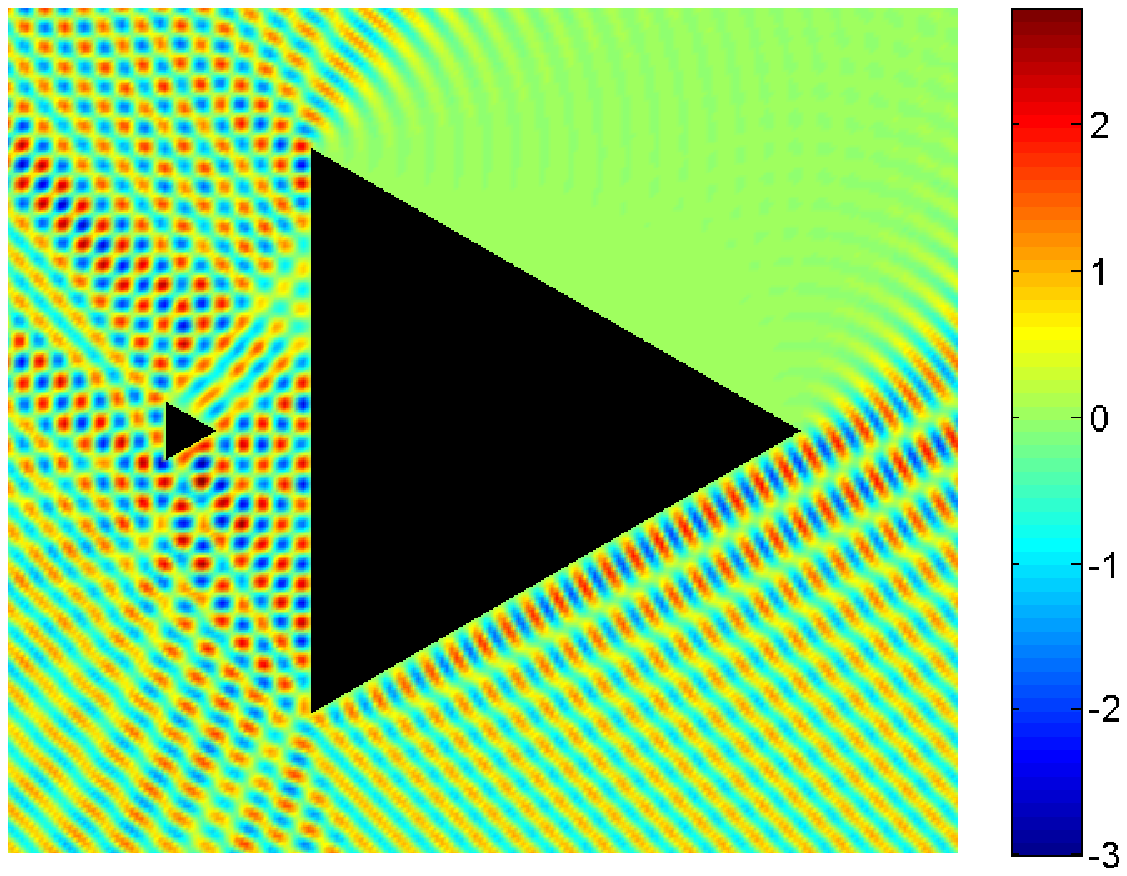}\label{fig:exp1D}}
	\caption{\resub{\protect\subref{fig:exp1L2} Convergence in $L^2(\partial D)$ and \protect\subref{fig:exp1D} the real component of the domain approximation for Experiment One \S\ref{sec:exp1} with {$k=20$}, $L_\Gamma=6\pi$, $L_\gamma=3\pi/5$, $k=20$, $\bfd=(1,1)/\sqrt2$, $N=1122$}.}\label{fig:exp1}
\end{figure}
\begin{figure}
	\centering
	\begin{tikzpicture}
	\node[inner sep=0pt] (plot1) at (0,0)
	{\subfloat[a][]{\psfrag{[x]}[t]{$s$}
	\psfrag{[y]}[Bcb]{$\re\{\nu_{\DOFsTotal}(\bfx_\Gamma(s))\}$}\psfrag{[t]}{}\includegraphics[width=.5\linewidth]{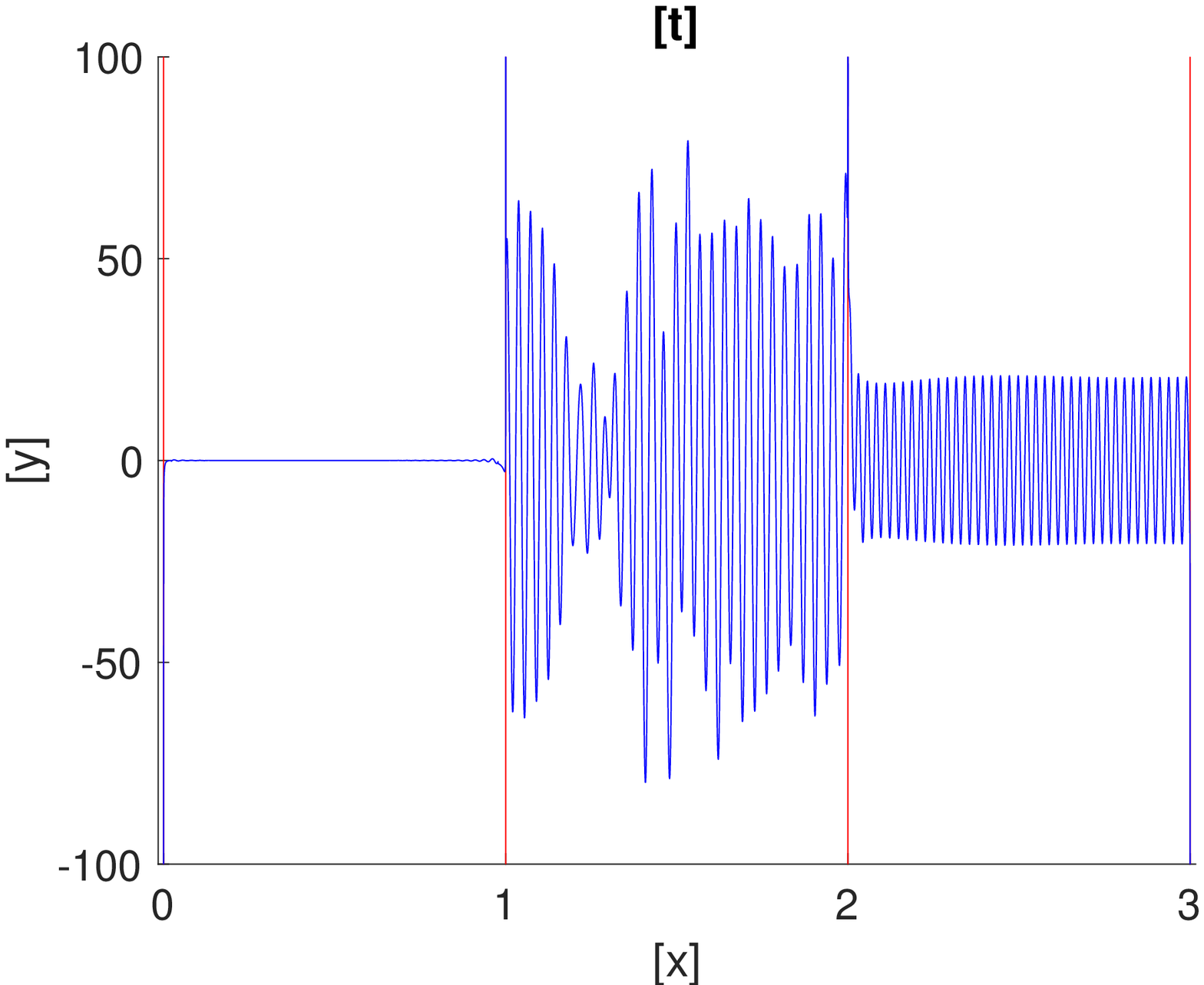}\label{fig:big}}
	\subfloat[b][]{\psfrag{[x]}[t]{$s$}
	\psfrag{[y]}[Bcb]{$\re\{\nu_{\DOFsTotal}(\bfx_\gamma(s))\}$}\psfrag{[t]}{}\includegraphics[width=.5\linewidth]{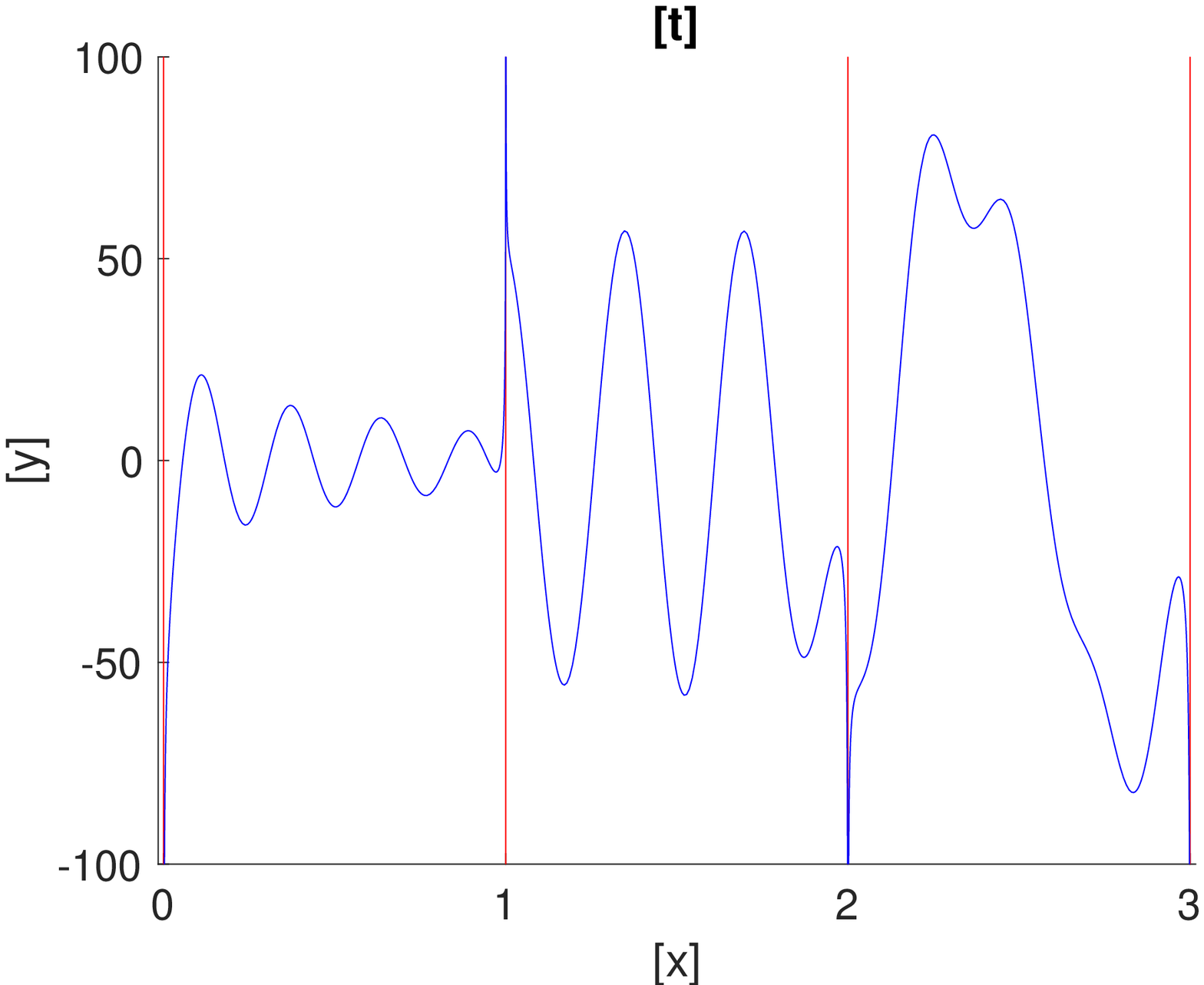}\label{fig:small}}};
	\draw(-5.6,0.5) node {shadow} ; 
	\draw (-3.6,2.1) node {facing $\gamma$};
	\end{tikzpicture}
	\caption{{The real component of the solution on the scatterer boundaries} $\Gamma$ \resub{\protect\subref{fig:big}} and $\gamma$ \resub{\protect\subref{fig:small}} for the configuration in Figure \ref{fig:exp1}\protect\subref{fig:exp1D}, with {$k=40$}.}
	\label{BoundarySoln}
\end{figure}
\resub{The configuration tested consists of an equilateral triangle $\Omega$ with perimeter $L_\Gamma=6\pi$ and a single small triangular scatterer with perimeter $L_\gamma=3\pi/5$, with the obstacles separated by a fixed distance of $\dist(\Gamma,\gamma)=\sqrt3\pi/5$. The \resub{configuration can be seen in} \ref{fig:exp1}\protect\subref{fig:exp1D}, which shows the real part of the approximation in the domain \eqref{u_M} for $p=8$.

It follows that there are exactly $k$ wavelengths on each side of $\Gamma$ and $k/10$ on each side of $\gamma$. Experiments were run for $k\in\{20,40,80,160\}$ ({so the number of wavelengths across the perimeter $\deD$ ranges from 66 to 528}). In Figure~\ref{BoundarySoln}, we show the real part of the solution $v_N, (N=1122)$ on $\Gamma$ and $\gamma$, for $k=40$.  On $\Gamma$, the first side $(s/(2\pi)\in[0,1])$ is the side in shadow, and the third side $(s/(2\pi)\in[2,3])$ is the illuminated side on the right in Figure~\ref{fig:exp1}\protect\subref{fig:exp1D}. On these two sides, the effect of the presence of $\resub{\omega}$ is negligible. However, on the middle side $(s/(2\pi)\in[1,2])$, the effect of $\resub{\omega}$ can clearly be seen.}

For a fixed number of DOFs $\DOFsTotal$, the \resub{$L^2(\partial D)$ error} is approximately the same for each $k$.
For each value of $k$ tested, we achieve approximately $1\%$ relative error with approximately $1000$ DOFs. For $k=160$ the combined boundary $\Gamma\cup\gamma$ is 528 wavelengths long, corresponding to approximately two DOFs per wavelength.
This illustrates why the method is particularly well suited to problems with one large polygon (for which the high-frequency asymptotics are well understood), and one (or many) small nearby obstacle(s) on which the high frequency asymptotics do not need to be known.
\newline
\resub{

\subsection{Experiment two}\label{sec:exp2}
Now we test the accuracy of our method as the separation (between the large and small obstacle) shrinks with increasing frequency, keeping all other parameters the same as in Experiment One (\S\ref{sec:exp1}). We choose the separation to be
\begin{equation}\label{eq:sep}
\dist(\Gamma,\gamma) = 3\pi/k,
\end{equation}
as is depicted for $k=20$ in Figure \ref{fig:exp2}\protect\subref{fig:exp2D}. Note the decrease in distance when compared with Figure \ref{fig:exp1}\protect\subref{fig:exp1D}.
\begin{figure}
	\centering
	\subfloat[a][]{
			\centering
		\psfrag{[x]}[c]{Total number of DOFs $\DOFsTotal$, for $p=1,\ldots,7$}\psfrag{[t]}[Bc]{}\psfrag{[y]}[Bc]{\vspace{-20mm}$\frac{\| \nu_{\DOFsTotal^*}-\nu_{\DOFsTotal}\|_{L^2({\partial D})}}{\|\nu_{\DOFsTotal^*}\|_{L^2({\partial D})}}$}{\includegraphics[width=.5\linewidth]{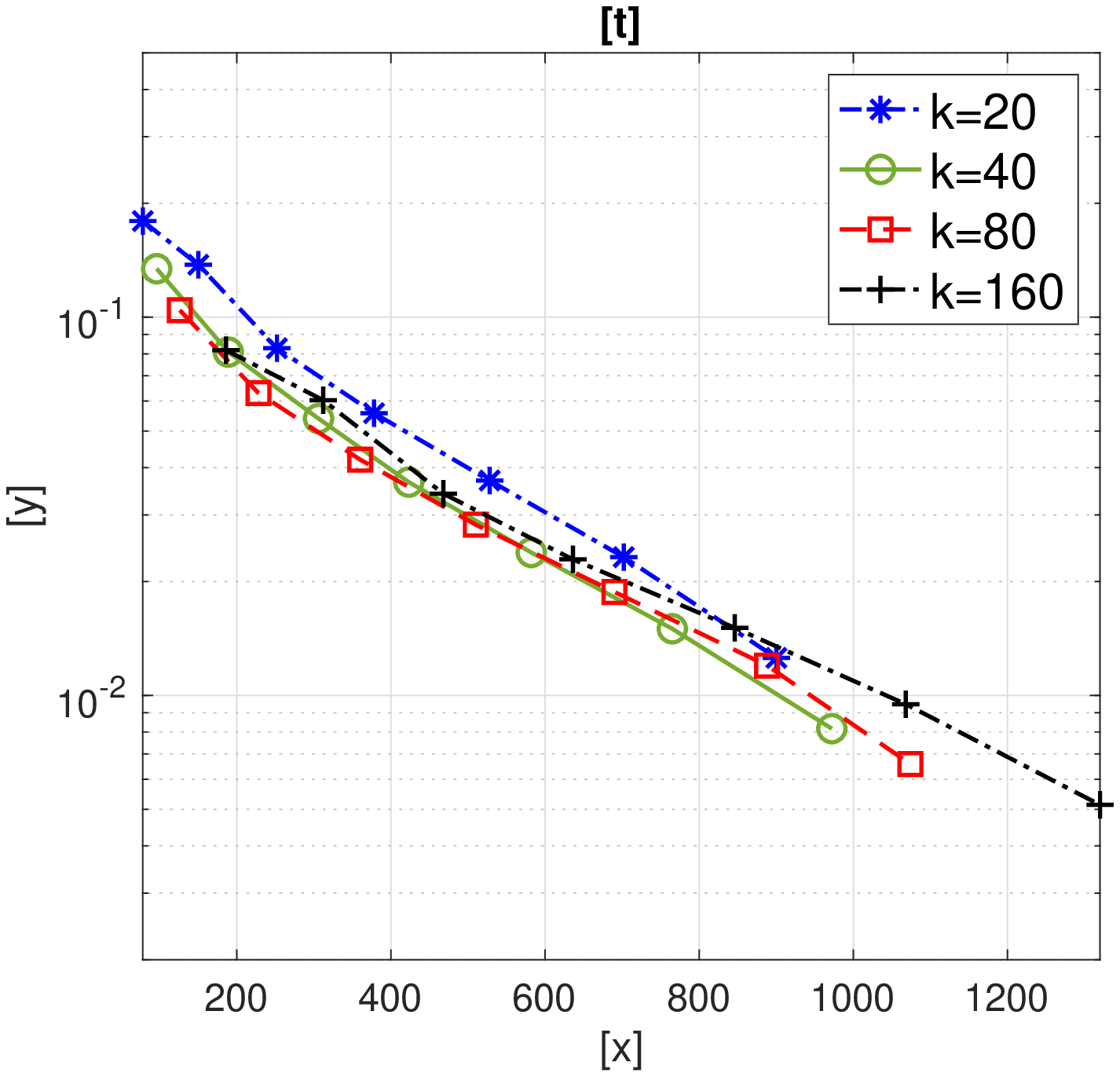}}\label{fig:exp2L2}}
		\subfloat[][]{\includegraphics[width=.5\linewidth, trim={4cm 1.7cm 1cm 1cm},clip]{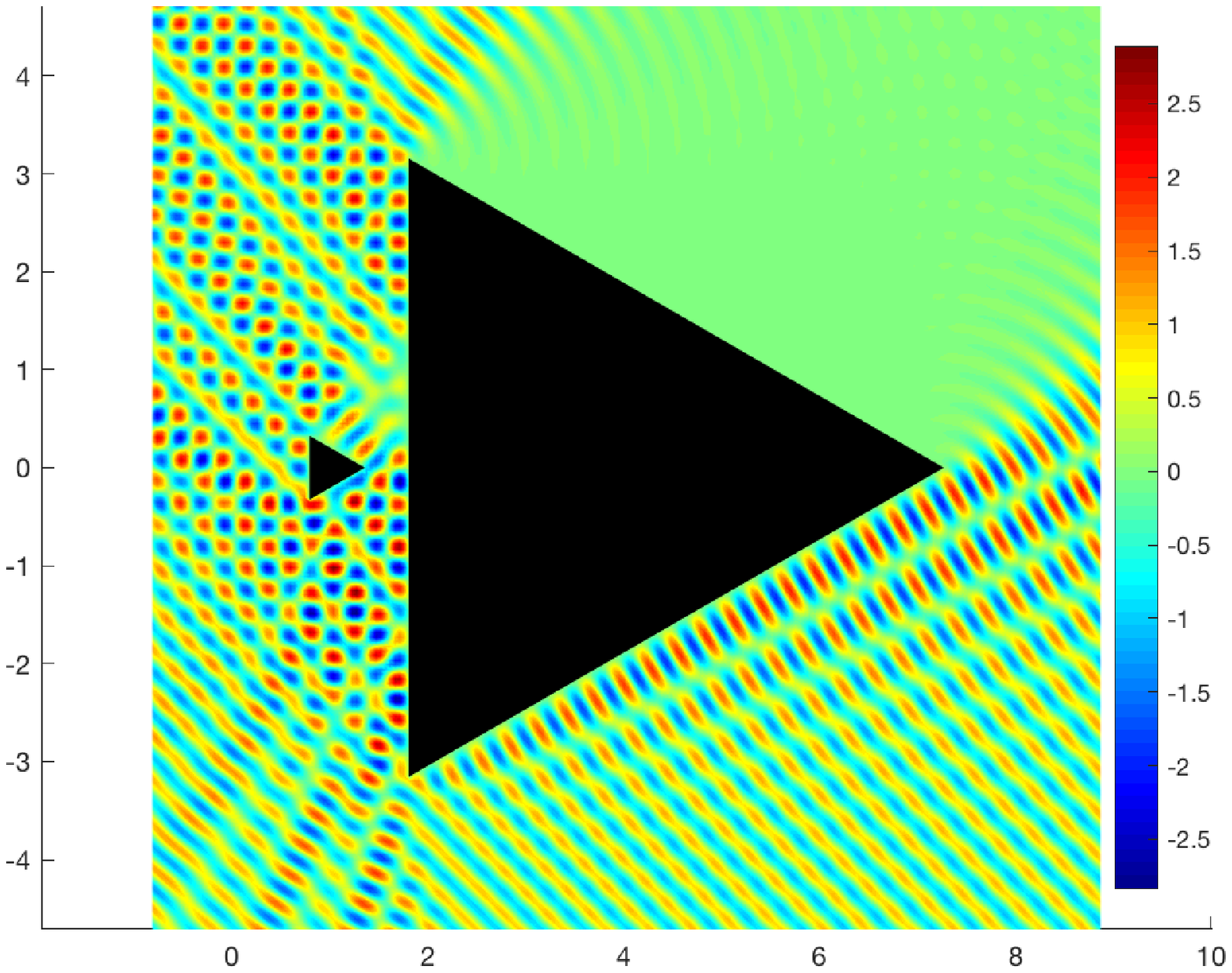}\label{fig:exp2D}}
		\caption{\resub{\protect\subref{fig:exp2L2} Convergence in $L^2(\partial D)$ and \protect\subref{fig:exp2D} the real component of the domain approximation for Experiment Two \S\ref{sec:exp2} with {$k=20$}, $L_\Gamma=6\pi$, $L_\gamma=3\pi/5$, $k=20$, $\bfd=(1,1)/\sqrt2$, $N=1122$}.}
	
	\label{fig:exp2}
\end{figure}
Despite the obstacles becoming very close together, with a separation of just $3\pi/160<0.06$ at the highest frequency tested, we observe reassuringly similar $L^2(\partial D)$ convergence rates (Figure \ref{fig:exp1}\protect\subref{fig:exp2L2}) to Experiment One (Figure \ref{fig:exp1}\protect\subref{fig:exp1L2}). This should not be unexpected, given that \eqref{eq:sep} satisfies the separation condition \eqref{as:c_r_is_0}. Upon closer inspection, the $L^2(\partial D)$ error is actually smaller when the obstacles are closer together, notably the final ($p=7$) data points for $k=40,80,160$.

This experiment demonstrates that our method can be applied to high frequency problems in which the obstacles are very close together. This is particularly encouraging when compared with iterative approaches for multiple scattering, which break down when the obstacles are too close together (as discussed in \S\ref{s:intro}).

\subsection{Experiment three}\label{sec:exp3}

Finally, we apply our method to a problem where the small obstacle consists of two small disjoint triangles $\omega_\gamma=\omega_1\cup\omega_2$. Here we take $\omega_1$ to be the smaller triangle from Experiment One (\S\ref{sec:exp1}), translated by $(0,1/2)$, and we take $\omega_2$ is the smaller triangle from Experiment One flipped horizontally and translated by $(0,-1/2)$. A key difference when compared with the previous two experiments is that this configuration will induce parabolic trapping. As with the previous experiments we have $\dist(\Gamma,\gamma)=\sqrt{3}\pi/5$, although now $L_\gamma=6\pi/5$. A consequence of this is that there will be twice as many DOFs in the standard basis $\Vhp$ than were required for the previous experiments, however with this adjustment we observe similar convergence rates (see Figure \ref{fig:exp3}\protect\subref{fig:exp3L2}). Figure \ref{fig:exp3}\protect\subref{fig:exp3D} shows the configuration and the real part of the domain approximation \eqref{u_M} for $p=8$ and $k=20$.

It is clear from Figure \ref{fig:exp3}\protect\subref{fig:exp3D} that the amplitude reaches four times that of the incoming wave, in the region of trapping between the three triangles. The results of this experiment confirm that our method can be extended to configurations of one large obstacle and multiple small obstacles, and is therefore well-suited for efficient modelling of a wide range of trapping phenomena.

\begin{figure}
	\centering
	\subfloat[][]{\psfrag{[x]}[t]{$s$}
	\psfrag{[x]}[c]{Total number of DOFs $\DOFsTotal$, for $p=1,\ldots,7$}\psfrag{[t]}[Bc]{}\psfrag{[y]}[Bc]{\vspace{-20mm}$\frac{\| \nu_{\DOFsTotal^*}-\nu_{\DOFsTotal}\|_{L^2({\partial D})}}{\|\nu_{\DOFsTotal^*}\|_{L^2({\partial D})}}$}{\includegraphics[width=.5\linewidth]{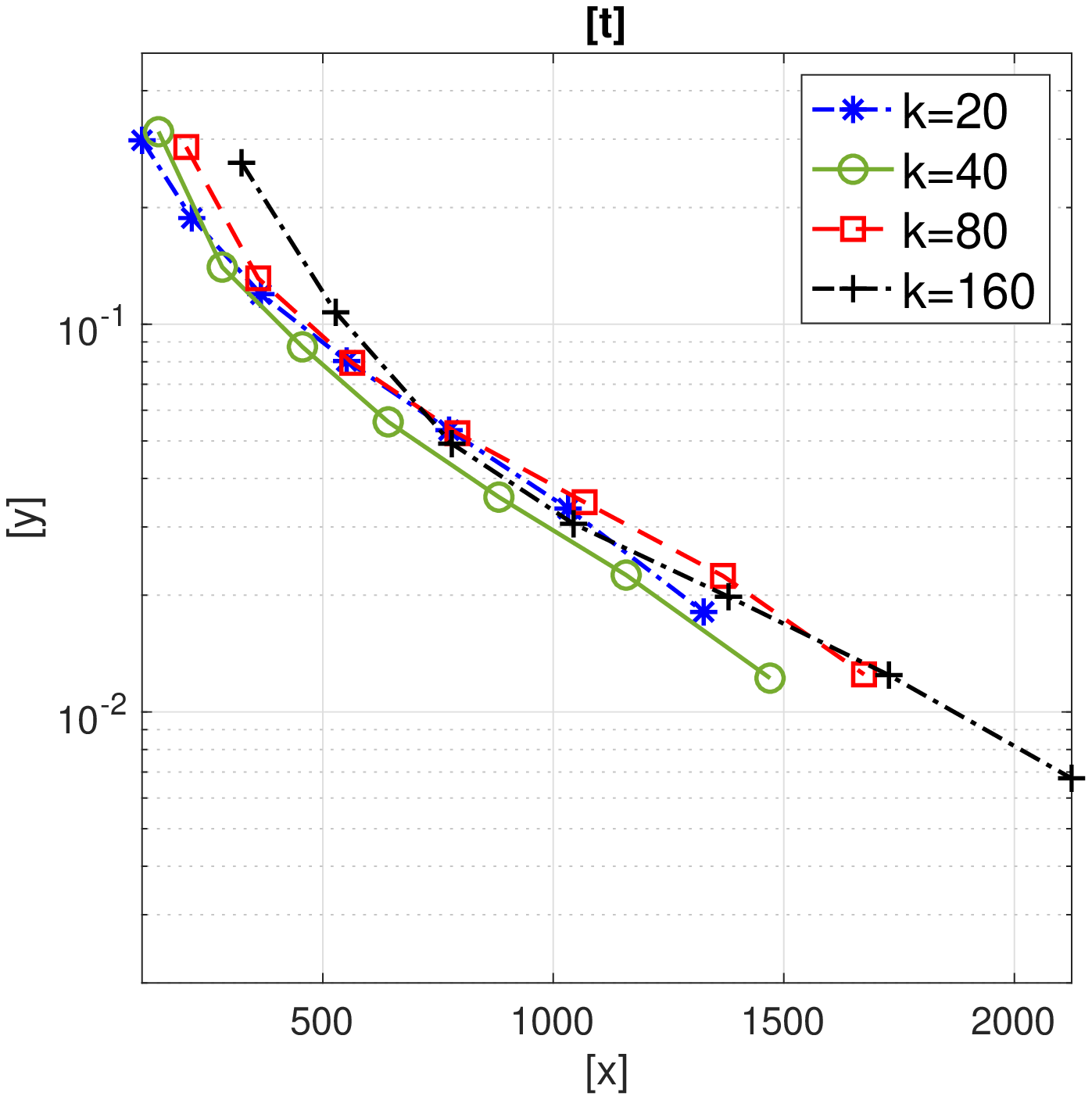}}\label{fig:exp3L2}}
	\subfloat[][]{\includegraphics[width=.5\linewidth, trim={4cm 1.7cm 1cm 1cm},clip]{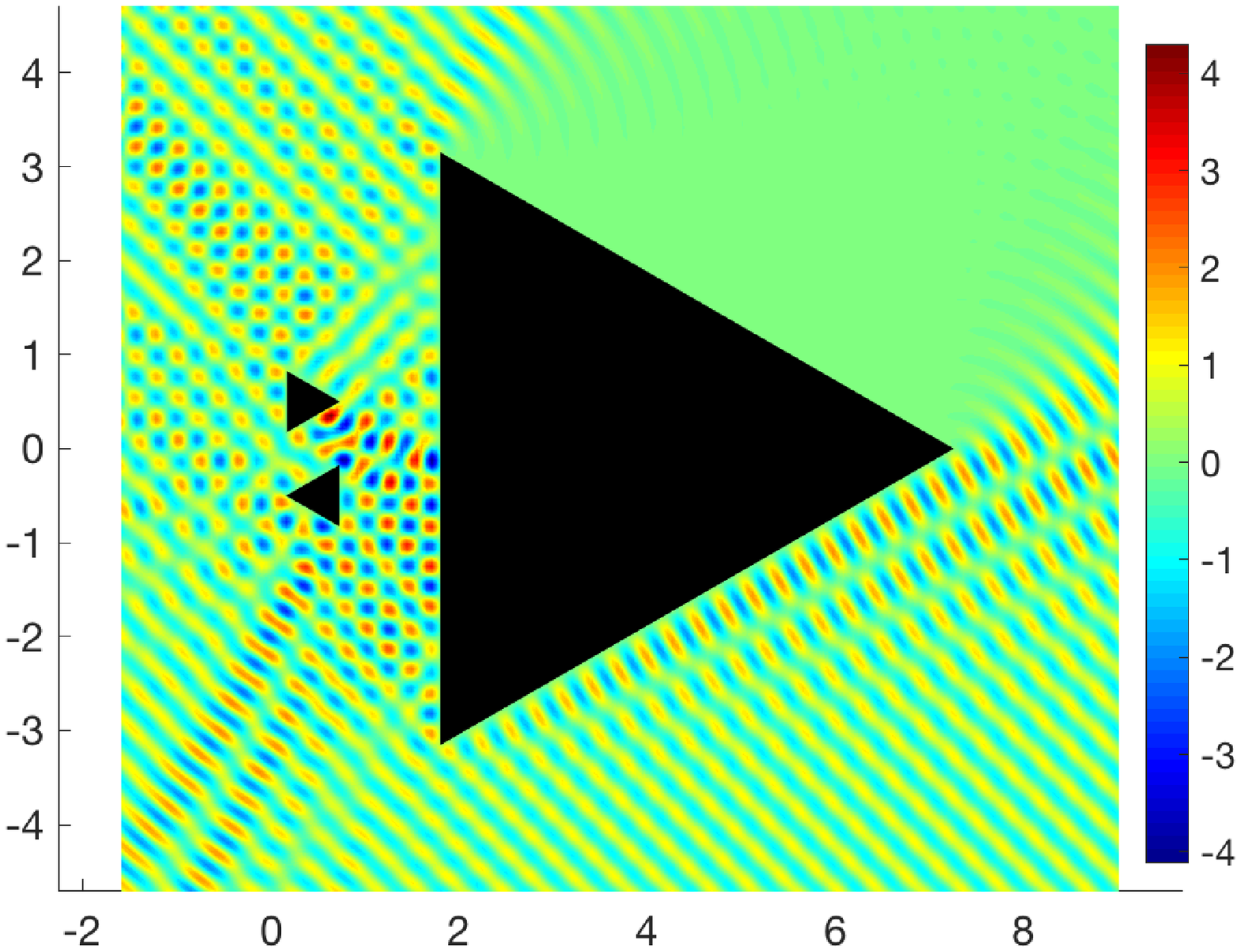}\label{fig:exp3D}}
	\caption{\resub{\protect\subref{fig:exp3L2} Convergence in $L^2(\partial D)$ and \protect\subref{fig:exp3D} real part of domain approximation for Experiment Three \S\ref{sec:exp3}, with $L_\Gamma=6\pi$, $L_\gamma=6\pi/5$, $k=20$, $\bfd=(1,1)/\sqrt2$, $N=1656$}.}\label{fig:exp3}
\end{figure}
}

\section{Conclusions and further work}
For a particular class of multiple scattering configurations, we have presented a numerical method which offers a significant reduction in degrees of freedom required at high frequencies, when compared to standard methods.
In particular, our method is most effective when one obstacle is much larger than the others.
The theoretical estimates presented in \S\ref{s:Galerkin_method} rely on a small number of reasonable assumptions, {which we prove to hold under certain conditions. H}owever the numerical results of \S\ref{s:results} show exponential convergence and stability with respect to the wavenumber in the broader setting where the small obstacle $\gamma$ is not analytic.

As suggested in Remark \ref{re:quadrature}, sophisticated quadrature rules are required in conjunction with the proposed method, but these rules can be difficult to implement for oscillatory  and singular double and triple integrals.
Alternatively, the approximation space of \S\ref{se:CombApprox} may be implemented as a collocation BEM (following the approach of \citet{GiHeHuPa:19}), which would reduce the dimension of each integral by one, making for easier implementation of oscillatory and singular quadrature rules.

The approach detailed in this paper requires at least one (ideally the largest) of the scatterers to be a convex polygon, but extension of this approach to a far broader class of configurations is possible.
The key requirement is that the high frequency asymptotics are understood on $\resub{\Omega}$, which with further work could instead be, e.g., a two-dimensional screen \citep{ChHeLa:14}, a non-convex obstacle \citep{ChHeLaTw:15}, or a penetrable obstacle \citep{GrHeLa:17}.
Such extensions would not be trivial, 
however we believe the framework established in this paper lays appropriate groundwork.

\resub{In \citep{ChLaMo:12}, the HNA method is extended from Dirichlet to impedance boundary conditions. By combining such an extension with the approach taken in the present paper, HNA methods may be designed for multiple obstacles with impedance (or Neumann) boundary conditions.}

A final area for future work is the case where {$\Gamma\cup\gamma$ is} connected, such that $\Gamma$ represents the surface of an obstacle on which an HNA basis can be used, whilst $\gamma$ is the component for which we cannot absorb the high-frequency asymptotics into the approximation space.
This extension would require more sophisticated bounds on the operator defined by \eqref{def:fullG}.

\section*{Acknowledgements}
This work was supported by EPSRC (EP/K000012/1 to S.L., PhD studentship to A.G.). \resub{A.G. also acknowledges support from the KU Leuven project C14/15/055}. A.M. acknowledges support from GNCS-INDAM and from MIUR through the
``Dipartimenti di Eccellenza" Programme (2018-2022) -- Dept. of
Mathematics, University of Pavia. The authors would also like to thank Prof. Markus Melenk for helpful discussions\resub{, and the anonymous referees for insightful comments and suggestions.}

\bibliographystyle{IMANUM-BIB}
\bibliography{AllRefsV3}

\appendix

\section{A coercive multiple scattering formulation}\label{s:appendix}

In \S\ref{s:Galerkin_method} it was noted that there exists a boundary integral formulation of the BVP \eqref{Helmholtz}--\eqref{SRC} which is \emph{coercive} (sometimes called $V$-elliptic), provided $|\gamma|$ is of the order of one wavelength. With a coercive formulation, it follows by the Lax--Milgram Theorem that the corresponding discrete problem (equivalent to \eqref{B_blocks}--\eqref{F_blocks}) is well posed, on any finite dimensional subspace of $L^2(\Gamma\cup\gamma)$. We now present this formulation.

For problems of scattering by a single star-shaped obstacle, it was shown in \citet{SpChGrSm:11} that the \emph{star combined} formulation is coercive for problems on a single star-shaped obstacle. In the thesis \citet{Gi:17} this formulation was extended to the \emph{constellation combined} formulation, where it was shown to be coercive for certain configurations consisting of multiple star-shaped obstacles. We present a version with sharper bounds here, specialising the coercivity result to the case of one large obstacle $\resub{\Omega}$ and one or many small obstacles $\resub{\omega}$. We begin by formally defining the configurations of interest:
\begin{definition}[Star- and constellation-shaped]\label{const_def}
	
	A bounded open set $\Upsilon$ with boundary $\partial \Upsilon$ is \emph{star-shaped} if there exists $\bfx^c\in\Upsilon$ and a Lipschitz continuous 
	$
	g:S^1\rightarrow\R,
	$
	where $S^1:=\{\hat\bfx\in\R^2:|\hat\bfx|=1\}$, such that $g(\hat\bfx)>0$ for all $\hat\bfx\in S^1$ with
	\[
	\partial\Upsilon=\{\bfx^c+g(\hat\bfx)(\hat\bfx-\bfx^c):\hat\bfx\in S^1\}.
	\]
	Intuitively, this may be interpreted as the following: Given any $\bfx\in\Upsilon$, one can draw a straight line from $\bfx^c$ to $\bfx$, without leaving $\Upsilon$.
	
	We say a domain is \emph{constellation-shaped} if it can be represented as the finite union of multiple star-shaped, pairwise disjoint obstacles. In such a case, for each star-shaped component we denote the above $\bfx^c$ parameter by $\bfx^c_i$, where $i$ is the index of that component.
\end{definition}
We will use the integral operator
\begin{equation}\label{def:SurfGrad}
\nabla_{S}\mathcal{S}_k\varphi(\bfx):=\int_{\partial \Upsilon}\nabla_{S}\Phi_k(\bfx,\bfy)\varphi(\bfy)\dd{s}(\bfy),\quad\text{for }\varphi\in L^2(\partial \Upsilon),\quad\bfx\in\partial\Upsilon,
\end{equation}
with the surface gradient operator of the fundamental solution as its kernel:
\begin{equation}\label{eq:SGkernel}
\nabla_{S}\Phi_k(\bfx,\bfy) := \nabla\Phi_k(\bfx,\bfy)-\bfn(\bfx)\pdiv{\Phi_k(\bfx,\bfy)}{\bfn(\bfx)},
\end{equation}
where $\Phi_k$ is as in \eqref{def:Sk}. Now we define our new BIE:
\begin{definition}[Constellation combined formulation]\label{def:star_comb}
	For a constellation-shaped domain $\Upsilon$ with boundary $\partial\Upsilon=\cup_{i=1}^{N_\Upsilon}\partial\Upsilon_i$, with $\partial\Upsilon_i$ the boundary of each star-shaped component, we define the constellation-combined operator $\mathcal{A}_k:L^2(\partial\Upsilon)\rightarrow L^2(\partial\Upsilon)$ as
	\[
    \mathcal{A}_k:=(\mathbf{Z}\cdot\bfn)\left(\frac{1}{2}\mathcal{I}+\mathcal{D}_k'\right)+\mathbf{Z}\cdot\nabla_{S}\mathcal{S}_k-\imag\hat{\eta}\mathcal{S}_k,
	\]
	where $\mathbf{Z}(\bfx)=\bfx-\bfx^c_i$ (with $\bfx^c_i\in\Upsilon_i$ chosen as $\bfx^c$ for each star-shaped component in Definition \ref{const_def}) on $\partial\Upsilon_i$, for $i=1,\ldots,N_{\Upsilon}$ and $\hat{\eta}(\bfx):=k|\mathbf{Z}(\bfx)|+\imag/2$. This operator yields an alternative BIE to \eqref{BIE}, namely
	\[
	\calA_k\pdiv{u}{\bfn}=f_k,\quad\text{on }\partial\Upsilon,
	\]
where the right-hand side data is
	\[
	f_k:=\left(\mathbf{Z}\cdot\nabla-\imag\hat{\eta}\right) \uinc ,\quad\text{on }\partial\Upsilon.
	\]
\end{definition}
Invertibility of $\calA_k$ follows by \citet[Theorem~2.41]{ACTA} and is shown in \citet[Theorem~5.6]{Gi:17}. For single star-shaped obstacles, the following is the key result of \citet{SpChGrSm:11}.
\begin{theorem}\label{th:starCombCoerc}
	Suppose $\Upsilon$ is star-shaped and $\calA_k$ is defined as in Definition \ref{def:star_comb}. Then the following coercivity result holds:
	\begin{equation}\label{starCoercive}
	\left|\left(\calA_k\varphi,\varphi\right)_{L^2(\partial\Upsilon)}\right|\geq \alpha_{\partial\Upsilon} \|\varphi\|_{L^2(\partial\Upsilon)}^2,\quad\text{for all }\varphi\in L^2(\partial \Upsilon),
	\end{equation}
	where
	\[
	\alpha_{\partial\Upsilon}:=	\frac{1}{2}\essinf{\bfx\in\partial\Upsilon}(\bfx\cdot\bfn(\bfx))>0.
	\]
\end{theorem}

In the thesis \citet{Gi:17}, the above result was extended to configurations of multiple star-shaped obstacles, under additional geometric constraints. These essentially required the obstacles to be sufficiently far apart, when compared with the wavelength and combined perimeter of the configuration. One way to interpret this is by decomposing $\calA_k$ into block operator form (as in \eqref{cts_prob}), where each off-diagonal block corresponds to the interaction between two disjoint obstacles, and the diagonal blocks correspond to self interactions. It follows by Theorem \ref{th:starCombCoerc} that the diagonal operators will be coercive in a constellation-shaped domain. If the interaction between the obstacles is sufficiently small, then any contribution from the off-diagonal terms will be small, and the full block operator will be coercive.
 It follows from \eqref{eq:SGkernel} that the kernel of the integral component
\[
\calA_k-(\bfZ\cdot\bfn)\frac{1}{2}\calI  =(\mathbf{Z}\cdot\bfn)\mathcal{D}_k'+\mathbf{Z}\cdot\nabla_{S}\mathcal{S}_k-\imag\hat{\eta}\mathcal{S}_k
\]
is
\[
(\mathbf{Z}(\bfx)\cdot\bfn(\bfx))\pdiv{\Phi_k(\bfx,\bfy)}{\bfn(\bfx)}+\bfZ(\bfx)\cdot\left(\nabla\Phi_k(\bfx,\bfy)-\bfn(\bfx)\pdiv{\Phi_k(\bfx,\bfy)}{\bfn(\bfx)}\right)-\imag\hat{\eta}(\bfx)\Phi_k(\bfx,\bfy)
\]which simplifies to
\[
K(\bfx,\bfy):=\bfZ(\bfx)\cdot\nabla\Phi_k(\bfx,\bfy)-\imag\hat{\eta}(\bfx)\Phi_k(\bfx,\bfy).
\]
We now consider disjoint, star-shaped boundaries $X$ and $Y$, with $\bfx\in X$ and $\bfy\in Y$. We can bound the kernel $K$ by considering Definition \ref{def:star_comb} and \eqref{eq:SGkernel}, noting $|\bfZ(\bfx)|\leq\diam(X)$, and upper bounds on the Hankel functions from \citet[(1.22),(1.23)]{ChGrLaLi:09}
\begin{align*}
&|K(\bfx,\bfy)| \\
& \leq k \diam(X)\left[\sqrt{\frac{1}{8\pi k\dist(X,Y)}} + \frac{1}{2\pi k \dist(X,Y)}\right] + \left(k\diam(X)+\frac{1}{2}\right)\sqrt{\frac{1}{8\pi k \dist(X,Y) }}\\
&\leq \left(k\diam(X)+\frac{1}{2}\right)\left[\sqrt{\frac{1}{2\pi k \dist(X,Y)}} + \frac{1}{2\pi k \dist(X,Y)}\right].
\end{align*}
It follows by the definition of the operator norm, and the Cauchy--Schwarz inequality (see  \citet[Lemma~5.13]{Gi:17} for a more general result) that for disjoint Lipschitz boundaries $X$ and $Y$,
\begin{align}
\|\calA_{k,Y\shortrightarrow X}\|_{L^2(Y)\shortrightarrow L^2(X)}&\leq \sqrt{|X||Y|}\esssup{\bfx\in X,\bfy\in Y}|K(\bfx,\bfy)|\nonumber\\
\leq&\sqrt{|X||Y|}\left(k \diam(X)+\frac{1}{2}\right)\left[\sqrt{\frac{1}{2\pi k \dist(X,Y)}} + \frac{1}{2\pi k \dist(X,Y)}\right].\label{eq:starCombInteraction}
 \end{align}
The bound \eqref{eq:starCombInteraction} quantifies the interaction between two disjoint Lipschitz boundaries $X$ and $Y$. The following theorem exploits this bound, deriving a coercivity result for a subclass of configurations considered in this paper - one large and one (or many) small obstacles.

\begin{theorem}\label{th:ConstCoerc}
	Suppose we have a multiple scattering configuration consisting of one large star-shaped obstacle with boundary $\Gamma$, and $\numScats$ small star-shaped obstacles $\gamma_i$ with boundary $\gamma=\cup_i\gamma_i$. Suppose further that obstacles are pairwise disjoint, such that the minimum distance between any two obstacles is bounded below by $R>0$. Assuming $|\Gamma|\geq|\gamma|$, if
	\begin{equation}\label{c_gCondish}
	{|\gamma|}<\left(\frac{\essinf{\bfx\in\Gamma\cup\gamma}{\left\{\bfZ(\bfx)\cdot\bfn(\bfx)\right\} } }{\left(k|\Gamma|+1\right)\sqrt{|\Gamma|}(2+\sqrt{\numScats})\left(\sqrt{\frac{1}{2\pi k R}}+\frac{1}{2 \pi k R}\right)}\right)^2,
	\end{equation}
	then the Constellation Combined operator of Definition \ref{def:star_comb} is coercive (i.e. satisfies a bound of the form \eqref{starCoercive}) with coercivity constant
\[
\alpha_{\Gamma\cup\gamma}:=\frac{1}{2}\essinf{\bfx\in\Gamma\cup\gamma}{\left\{\bfZ(\bfx)\cdot\bfn(\bfx)\right\} }  - \sqrt{|\Gamma||\gamma|}\left(k|\Gamma|+1\right)(2+\sqrt{\numScats})\left(\sqrt{\frac{1}{8\pi k R}}+\frac{1}{4 \pi k R}\right)
\]	
\end{theorem}
\begin{proof}
	To simplify the notation, we shall write $\|\calA_{k,Y\shortrightarrow X}\|$ to mean $\|\calA_{k,Y\shortrightarrow X}\|_{L^2(Y)\shortrightarrow L^2(X)}$. We begin by decomposing the operator into a sum of operators defined on subsets of $\Gamma\cup\gamma$,
	\begin{equation}\label{eq:Asplit}
		\left(\calA_{k}\varphi,\varphi\right)_{L^2(\Gamma\cup\gamma)} = 	\left(\calA_{\text{diag}}\varphi,\varphi\right)_{L^2(\Gamma\cup\gamma)} + \left(\calA_{\text{cross}}\varphi,\varphi\right)_{L^2(\Gamma\cup\gamma)},
	\end{equation}
	in which we have split the operator into diagonal and off-diagonal terms
	\begin{align}
		\calA_{\text{diag}} :=\calA_{k,\Gamma\shortrightarrow\Gamma}+\sum_{i=1}^{\numScats}\calA_{k,\gamma_i\shortrightarrow\gamma_i},\quad
		\calA_{\text{cross}}:=	\calA_{k,\gamma\shortrightarrow\Gamma}+\calA_{k,\Gamma\shortrightarrow\gamma}+\sum^{\numScats}_{i=1}\calA_{k,\gamma_i\shortrightarrow(\gamma\setminus\gamma_i)},
	\end{align}
	where we have abused the notation of Definition \ref{def:Txy_bits}, which is used differently here to mean:
	\[
	\calK_{X\shortrightarrow Y}\varphi:=
	\mathbbm{1}_Y\calK[\mathbbm{1}_X\varphi],
	\]
	where $\mathbbm{1}_X$ is an indicator function, equal to one on $X$ and zero otherwise.
	The diagonal terms can all be bounded via Theorem \ref{th:starCombCoerc}, yielding
	\begin{equation}\label{wantPositive}
	\left|\left(\calA_{k}\varphi,\varphi\right)_{L^2(\Gamma\cup\gamma)} \right|\geq
	\quad\frac{1}{2}\essinf{\bfx\in\Gamma\cup\gamma}{\left\{\bfZ(\bfx)\cdot\bfn(\bfx)\right\} }\|\varphi\|^2_{L^2(\Gamma\cup\gamma)}-
	\left|(\calA_{\text{cross}}\varphi,\varphi)_{L^2(\Gamma\cup\gamma)}\right|.
	\end{equation}
	We want to find conditions under which the right-hand side of the above inequality is positive, hence we require the negative term to be sufficiently small. We bound these off-diagonal terms
	\begin{equation}
	\left|(\calA_{\text{cross}}\varphi,\varphi)_{L^2(\Gamma\cup\gamma)}\right|\leq
	\|\calA_{\text{cross}}\|_{L^2(\Gamma\cup\gamma)\shortrightarrow L^2(\Gamma\cup\gamma)}\|\varphi\|_{L^2(\Gamma\cup\gamma)}^2.
	\end{equation}
	We now split the above norm on $\calA_{\text{cross}}$ using the triangle inequality noting the terms in \eqref{eq:Asplit}, and apply the bound \eqref{eq:starCombInteraction} to each component,
	\begin{align}
&		\|\calA_{\text{cross}}\|_{L^2(\Gamma\cup\gamma)\shortrightarrow L^2(\Gamma\cup\gamma)}\leq \|\calA_{k,\gamma\shortrightarrow\Gamma}\|+\|\calA_{k,\Gamma\shortrightarrow\gamma}\|+\sum^{\numScats}_{i=1}\|\calA_{k,\gamma_i\shortrightarrow(\gamma\setminus\gamma_i)}\|\nonumber
		\\\leq &
		\frac{1}{2}\left(\sqrt{|\Gamma||\gamma|}+\sqrt{|\gamma||\Gamma|}+\sqrt{|\gamma|}\sum_{i=1}^{\numScats}\sqrt{|\gamma_i|} \right)\left(k|\Gamma|+1\right)\left(\sqrt{\frac{1}{2\pi k R}}+\frac{1}{2 \pi k R}\right),\label{ineq:normCrossSplit}
	\end{align}
	where we have used $|\Gamma|\geq\{|\gamma|,2\diam(\Gamma),2\diam({\gamma})\}$ to simplify terms. Appealing also to the Cauchy-Schwarz inequality, we can write
	\[
	\sum_{i=1}^{\numScats}\sqrt{|\gamma_i|}\leq \sqrt{\numScats|\gamma|}\leq \sqrt{\numScats|\Gamma|},
	\]
	which can be used to simplify \eqref{ineq:normCrossSplit} to obtain
	\begin{align*}
    \|\calA_{\text{cross}}\|_{L^2(\Gamma\cup\gamma)\shortrightarrow L^2(\Gamma\cup\gamma)}\leq 
		\frac{1}{2}\sqrt{|\Gamma||\gamma|}\left(k|\Gamma|+1\right)(2+\sqrt{\numScats})\left(\sqrt{\frac{1}{2\pi k R}}+\frac{1}{2 \pi k R}\right).
	\end{align*}
Noting \eqref{wantPositive}, we require that
\[
\frac{1}{2}\essinf{\bfx\in\Gamma\cup\gamma}{\left\{\bfZ(\bfx)\cdot\bfn(\bfx)\right\} } - \frac{1}{2}\sqrt{|\Gamma||\gamma|}\left(k|\Gamma|+1\right)(2+\sqrt{\numScats})\left(\sqrt{\frac{1}{2\pi k R}}+\frac{1}{2 \pi k R}\right)>0,
\]
which is equivalent to the condition \eqref{c_gCondish}.
\end{proof}

We do not expect the above result to be sharp. A key consequence is the following: if $|\gamma|$ is no more than a fixed fraction of a wavelength, the constellation combined formulation is coercive. We conclude this appendix with bounds on two of the key constants of the Galerkin method as outlined in \S\ref{s:Galerkin_method}, if the constellation combined formulation is used instead of the standard combined formulation. With the standard formulation, we are unable to bound these constants given current available theory.

\begin{corollary}\label{cor:Cea}
	Suppose we reformulate the Galerkin method of \S\ref{s:Galerkin_method} instead using the constellation combined formulation of Definition \ref{def:star_comb}, and that our scattering configuration $\Upsilon=\resub{\Omega}\cup\resub{\omega}$ satisfies the conditions of Theorem \ref{th:ConstCoerc}. Then the constants $C_q(k)$ and $N_0(k)$ of Lemma \ref{uniquenessProj}, Theorem \ref{th:anz_approx_bd} and Corollary \ref{cor:domainErr} satisfy
	\[
	C_q(k)=\frac{C\sqrt{k}}{\alpha_{\Gamma\cup\gamma}}\quad\text{and}\quad N_0(k)=1,
	\]
	where $C>0$ is a constant which depends only on the geometry of $\Gamma$ and $\gamma$ and $\alpha_{\Gamma\cup\gamma}$ is the coercivity constant from Theorem \ref{th:ConstCoerc}.
\end{corollary}
\begin{proof}
Given that the conditions of Lemma \ref{uniquenessProj} hold, our formulation is coercive. It follows by the Lax--Milgram Theorem that $N_0(k)=1$. It follows by C\'ea's Lemma that the quasi-optimality constant is
\begin{equation}\label{eq:Cea}
C_q(k)=\frac{\|\calA_k\|_{L^2(\Gamma\cup\Gamma)\shortrightarrow L^2(\Gamma\cup\Gamma)}}{\alpha_{\Gamma\cup\gamma}}.
\end{equation}
The norm in the numerator of \eqref{eq:Cea} is $O(k^{1/2})$ for all $k\geq k_0$ \citep[Theorem~4.2]{SpChGrSm:11}.
\end{proof}

Finally, we remark that for a given geometry $\resub{\Omega}\cup\resub{\omega}$, there exists a $k_1>0$ such that for all $k\geq k_1$, Theorem \ref{th:ConstCoerc} cannot guarantee coercivity, and consequentially the statements of Corollary \ref{cor:Cea} may not be valid. This is because the negative component of $\alpha_{\Gamma\cup\gamma}$ (as defined in Theorem \ref{th:ConstCoerc}) will become larger in magnitude as $k$ increases, whilst the positive component remains fixed; we require $\alpha_{\Gamma\cup\gamma}>0$ to ensure coercivity.
\end{document}